\documentclass[reqno]{amsart}
\usepackage{amsmath, amsthm, amssymb, color}
\usepackage{graphicx}
\usepackage{latexsym}
\usepackage{hyperref}
\usepackage{xcolor}
\usepackage{enumerate}
\usepackage{setspace}
\usepackage{cases}

\numberwithin{equation}{section}
\newtheorem{Theorem}{Theorem}[section]
\newtheorem{Lemma}[Theorem]{Lemma}

\theoremstyle{definition}
\newtheorem{remark}[Theorem]{Remark}
\newtheorem{Proposition}[Theorem]{Proposition}

\newcommand{\R}{\mathbb{R}}

\newcommand{\FL}{\mathbf{L}}
\newcommand{\FI}{\mathbf{I}}

\newcommand{\la}{\lambda}

\def\f{\frac}
\def\fc{\mathfrak{c}}
\def\v{\varepsilon}

\newcounter{RomanNumber}

\allowdisplaybreaks[4]

\title [Newtonian Limit for the relativistic Boltzmann equation] {The Global Well-posedness and Newtonian Limit for the Relativistic Boltzmann Equation in a Periodic Box}

\author[C. Q. Cao]{Chuqi Cao}
\address[C. Q. Cao]{Department of Applied Mathematics, The Hong Kong Polytechnic University, Hong Kong, China}
\email{chuqicao@gmail.com}

\author[J. Ouyang]{Jing Ouyang}
\address[J. Ouyang]{Academy of Mathematics and Systems Science, Chinese Academy of Sciences, Beijing 100190, China}
\email{ouyangjing@mass.ac.cn}

\author[Y. Wang]{Yong Wang}
\address[Y. Wang]{Academy of Mathematics and Systems Science, Chinese Academy of Sciences, Beijing 100190, China; School of Mathematical Sciences, University of Chinese Academy of Sciences, Beijing 100049, China}
\email{yongwang@amss.ac.cn}

\author[C. G. Xiao]{Changguo Xiao}
\address[C. G. Xiao]{School of Mathematics and Statistics, Guangxi Normal University, Guilin, Guangxi 541004, China}
\email{changguoxiao@mailbox.gxnu.edu.cn}

\thanks{* Corresponding author: changguoxiao@mailbox.gxnu.edu.cn}
\bigbreak

\date{\today}
 
\begin{document}
\begin{abstract}
In this paper, we study the Newtonian limit for relativistic Boltzmann equation in a periodic box $\mathbb{T}^3$. We first establish the global-in-time mild solutions of relativistic Boltzmann equation with uniform-in-$\mathfrak{c}$ estimates and time decay rate. Then we rigorously justify the global-in-time Newtonian limits from the relativistic Boltzmann solutions to the solution of Newtonian Boltzmann equation  in  $L^1_pL^{\infty}_x$. Moreover, if the initial data of Newtonian Boltzmann equation belong to $W^{1,\infty}(\mathbb{T}^3\times\mathbb{R}^3)$, based on a  decomposition and  $L^2-L^\infty$ argument, the global-in-time Newtonian limit is proved in $L^{\infty}_{x,p}$.  The convergence rates of Newtonian limit are obtained both in $L^1_pL^{\infty}_x$ and $L^{\infty}_{x,p}$.
\end{abstract}

\maketitle
\setcounter{tocdepth}{1}
\tableofcontents

%%%%%%%%%%%%%%%%%%%%%%%%%%%%%%%%%%%%%%%%%%%%%%%%%%%%%%%%%%%%%%%
	\section{Introduction}
%%%%%%%%%%%%%%%%%%%%%%%%%%%%%%%%%%%%%%%%%%%%%%%%%%%%%%%%%%%%%%%
	
%%%%%%%%%%%%%%%%%%%%%%%%%%%%%%%%%%%%%%%%%%%%%%%%%%%%%%%%%%%%%%%
	\subsection{The relativistic Boltzmann equation}
%%%%%%%%%%%%%%%%%%%%%%%%%%%%%%%%%%%%%%%%%%%%%%%%%%%%%%%%%%%%%%%
	
	The dynamics of single-species special relativistic particles whose speed is comparable to
	the speed of light is governed by the special relativistic Boltzmann equation
	\begin{equation}\label{1.1-0}
		p^{\mu}\partial_{\mu}F=\mathcal{C}(F, F).
	\end{equation}
	The unknown $F(t,x,p)\ge 0$ stands for the density distribution function for relativistic particles with position $x=(x_1,x_2,x_3)\in \mathbb{T}^3$ and particle momentum $p=(p^1,p^2,p^3)\in \mathbb R^3$ at time $t>0$. The collision term $\mathcal{C}(h_1, h_2)$ is defined by
	\begin{equation*}%\label{1.2}
		\mathcal{C}(h_1, h_2)=\frac{1}{2} \int_{\mathbb{R}^{3}} \frac{d q}{q^{0}} \int_{\mathbb{R}^{3}} \frac{d p^{\prime}}{p^{\prime 0}} \int_{\mathbb{R}^{3}} \frac{d q^{\prime}}{q'^{0}} W\left(p, q \mid p^{\prime}, q^{\prime}\right)\left[h_1\left(p^{\prime}\right) h_2\left(q^{\prime}\right)-h_1(p) h_2(q)\right],
	\end{equation*}
     where the transition rate $W\left(p, q \mid p^{\prime}, q^{\prime}\right)$ is
	\begin{align}\label{1.2-0}
		W\left(p, q \mid p^{\prime}, q^{\prime}\right)=\mathfrak{s} \sigma(g, \theta)\delta(p^0+q^0-p'^{0}-q'^{0}) \delta^{(3)}(p+q-p'-q').
	\end{align}
    The function $\sigma(g,\theta)$ in \eqref{1.2-0} is called the relativistic differential cross-section or scattering kernel which describes the collisions
    between relativistic particles. 
	The streaming term of the relativistic Boltzmann equation \eqref{1.1-0} is given by
	\begin{align*}
		p^\mu\partial_\mu=\frac{p^0}{\mathfrak{c}}\partial_t+p\cdot \nabla_x,
	\end{align*}
	where $\mathfrak{c}$ denotes the speed of light and $p^0$ denotes the energy of a relativistic particle with
	$$p^{0}=\sqrt{m^2\mathfrak{c}^2+|p|^{2}}.$$
	Here $m$ denotes the rest mass of particle. 
	
	We can rewrite \eqref{1.1-0} as
	\begin{align}\label{1.3-0}
		\partial_t F+\hat{p}\cdot\nabla_x F=\mathcal{Q}(F,F),
	\end{align}
	where $\hat{p}$ denotes the normalized particle velocity
	\begin{align*}
		\hat{p}:=\mathfrak{c}\frac{p}{p^0}=\frac{\mathfrak{c}p}{\sqrt{m^2\mathfrak{c}^2+|p|^2}}.
	\end{align*}
	The collision term $\mathcal{Q}(h_1,h_2)$ in \eqref{1.3-0} has the form
	\begin{equation*}%\label{1.5}
		\mathcal{Q}(h_1, h_2)=\frac{\mathfrak{c}}{2} \frac{1}{p^{0}}\int_{\mathbb{R}^{3}} \frac{d q}{q^{0}} \int_{\mathbb{R}^{3}} \frac{d p^{\prime}}{p^{\prime 0}} \int_{\mathbb{R}^{3}} \frac{d q^{\prime}}{q'^{0}} W\left(p, q \mid p^{\prime}, q^{\prime}\right)\left[h_1\left(p^{\prime}\right) h_2\left(q^{\prime}\right)-h_1(p) h_2(q)\right].
	\end{equation*}
	The relativistic Boltzmann equation \eqref{1.3-0} is supplemented with the following initial data
	\begin{align}\label{1.4-0}
		F(t,x,p)\mid_{t=0}=F_0(x,p).
	\end{align}

	We denote the energy-momentum 4-vector as $p^{\mu}=(p^0,p^1,p^2,p^3)$. The energy-momentum 4-vector with the lower index is written as a product in the Minkowski metric $p_{\mu}=g_{\mu \nu} p^{\nu}$, where the Minkowski metric is given by $g_{\mu \nu}=\operatorname{diag}(-1,1,1,1)$. The inner product of energy-momentum 4-vectors $p^{\mu}$ and $q_{\mu}$ is defined via the Minkowski metric
	\begin{align*}
		p^{\mu} q_{\mu}=p^{\mu} g_{\mu \nu} q^{\nu}=-p^{0} q^{0}+\sum_{i=1}^{3} p^{i} q^{i}.
	\end{align*}
    Note that the momentum for each particle satisfies the mass shell condition 
    $$p^{\mu} p_{\mu}=-m^2\mathfrak{c}^2$$
    with $p^0 > 0$. Also, the product $p^{\mu} q_{\mu}$ is Lorentz invariant.
	%  $p^{\mu} q_{\mu}=\Lambda p^{\mu} \Lambda q_{\mu}$, where $\Lambda$ is any Lorentz transformation.
	
	The quantity $\mathfrak{s}$ and $g$ are respectively the square of the energy and the relative momentum in the \textit{center of momentum system}, $p+q=0$, and they are given by
	\begin{align*}%\label{1.6}
		\mathfrak{s}=\mathfrak{s}(p, q)=-\left(p^{\mu}+q^{\mu}\right)\left(p_{\mu}+q_{\mu}\right)=2\left(p^{0} q^{0}-p \cdot q+m^{2} \mathfrak{c}^{2}\right) \geq 4 m^{2} \mathfrak{c}^{2}
	\end{align*}
	and
	\begin{align*}%\label{1.7}
		g=g(p, q)=\sqrt{\left(p^{\mu}-q^{\mu}\right)\left(p_{\mu}-q_{\mu}\right)}=\sqrt{2\left(p^{0} q^{0}-p \cdot q-m^{2} \mathfrak{c}^{2}\right)} \geq 0.
	\end{align*}
	Notice that we have 
	\begin{align*}
		\mathfrak{s}=g^2+4m^2\mathfrak{c}^2.
	\end{align*}
    For simplicity, we normalize $m$ to one in the rest of this paper.
    
	The post-collision momentum pair $(p'^{\mu},q'^{\mu})$ and the pre-collision momentum pair $(p^{\mu},q^{\mu})$ satisfy the relation
	\begin{align}\label{1.8}
		p^{\mu}+q^{\mu}=p^{\prime \mu}+q^{\prime \mu}.
	\end{align}
	One may also write \eqref{1.8} as
	\begin{align}
		p^0+q^0&=p^{\prime 0}+q^{\prime 0},\label{1.9}\\
		p+q&=p^{\prime}+q^{\prime }\label{1.10},
	\end{align}
	where \eqref{1.9} represents the principle of conservation of energy and \eqref{1.10} represents the conservation of momentum after a binary collision.

	Throughout the present paper, we consider the ``hard ball" particles 
	\begin{align*}
		\sigma(g, \theta) = \text{constant}.
	\end{align*}
	Without loss of generality, we take $\sigma(g, \theta) = 1$ for simplicity. The Newtonian limit in this situation, as $\mathfrak{c}\rightarrow \infty$,  is the Newtonian hard-sphere Boltzmann collision operator \cite{Strain}.
	
	We introduce the relativistic global Maxwellian $J_{\mathfrak{c}}(p)$ as
	\begin{align*}
		J_{\mathfrak{c}}(p):=\f{1}{4\pi \fc K_2(\fc^2)}e^{-\fc p^0},
	\end{align*} 
	where $K_2(z)$ is the modified Bessel function of the second kind (see \eqref{e2.15}). A direct calculation shows that
	\begin{align*}
		\int_{\mathbb{R}^3}J_{\mathfrak{c}}(p)dp=1.
	\end{align*} 

	Let $F_{\mathfrak{c}}$ be a solution of the relativistic Boltzmann equation \eqref{1.3-0}--\eqref{1.4-0} with initial data $F_{0,\mathfrak{c}}$. Formally, $F_{\mathfrak{c}}$ satisfies the conversations of mass, momentum, and energy
	\begin{align}\label{1.8-10}
		\int_{\mathbb{T}^3} \int_{\mathbb{R}^3}[F_{\mathfrak{c}}(t, x, p)-J_{\mathfrak{c}}(p)]
		\begin{pmatrix}
			1  \\  p \vspace{1.0ex}\\  p^0
		\end{pmatrix}
	    d p d x&=\int_{\mathbb{T}^3} \int_{\mathbb{R}^3}[F_{0,\mathfrak{c}}(x, p)-J_{\mathfrak{c}}(p)]
	    \begin{pmatrix}
	    	1  \\  p \vspace{1.0ex}\\  p^0
	    \end{pmatrix}
	    d p d x
	    =\begin{pmatrix}
	    	M_{\mathfrak{c}} \vspace{1.0ex} \\  \mathbf{J}_{\mathfrak{c}} \vspace{1.0ex}\\  E_{\mathfrak{c}}
	    \end{pmatrix}.
    \end{align}
%	    =\int_{\mathbb{T}^3} \int_{\mathbb{R}^3}\left[F_{0,\mathfrak{c}}-J_{\mathfrak{c}}(p)\right] d p d x:=M_{\mathfrak{c}}, \label{1.8-10}\\
%		\int_{\mathbb{T}^3} \int_{\mathbb{R}^3} p[F_{\mathfrak{c}}(t, x, p)-J_{\mathfrak{c}}(p)] d p d x &=\int_{\mathbb{T}^3} \int_{\mathbb{R}^3} p\left[F_{0,\mathfrak{c}}-J_{\mathfrak{c}}(p)\right] d p d x:=\mathbf{J}_{\mathfrak{c}},  \label{1.9-10}\\
%		\int_{\mathbb{T}^3}\int_{\mathbb{R}^3} p^0[F_{\mathfrak{c}}(t, x, p)-J_{\mathfrak{c}}(p)] d p d x & =\int_{\mathbb{T}^3} \int_{\mathbb{R}^3} p^0\left[F_{0,\mathfrak{c}}-J_{\mathfrak{c}}(p)\right] d p d x:=E_{\mathfrak{c}}. \label{1.10-10}

%%%%%%%%%%%%%%%%%%%%%%%%%%%%%%%%%%%%%%%%%%%%%%%%%%%%%%%%%%%%%%%
\subsection{A brief history on the Newtonian limit of  relativistic Boltzmann equation}

In the context of the Newtonian limit of relativistic particles, Calogero \cite{Calogero} previously established local-in-time solutions for the relativistic Boltzmann equation within a periodic box, independent of the light speed $\mathfrak{c}$. He also demonstrated that these solutions converge to the Newtonian Boltzmann solutions as $\mathfrak{c}\to \infty$. Subsequently, Strain \cite{Strain} extended this work, proving the existence of unique global-in-time mild solutions and justifying the Newtonian limit over arbitrary time intervals $[0, T]$, with a convergence rate $\mathfrak{c}^{2-\varepsilon}$ applicable to any $\varepsilon\in (0,2)$.

%Within the framework of perturbation around the local Maxwellian, 
Wang-Xiao \cite{Wang-Xiao} established  the Hilbert expansion of relativistic Boltzmann equation with uniform-in-$\mathfrak{c}$ estimates, and hence proved both the hydrodynamic limit and Newtonian limit which bridges the gap between relativistic Boltzmann solutions and classical Euler solutions.

%%%%%%%%%%%%%%%%%%%%%%%%%%%%%%%%%%%%%%%%%%%%%%%%%%%%%%%%%%%%%%%
\subsection{Main results}
%%%%%%%%%%%%%%%%%%%%%%%%%%%%%%%%%%%%%%%%%%%%%%%%%%%%%%%%%%%%%%%

In the present paper, we study Newtonian limit of relativistic Boltzmann equation in torus. For this, we first focus on the global existence of \eqref{1.3-0}--\eqref{1.4-0} with uniform-in-$\mathfrak{c}$ estimates. For later use, we assume  $\mathfrak{c}\ge 1$ throughout the paper. To emphasize the dependence on $\mathfrak{c}$ for relativistic Boltzmann operators, initial data or solutions, we denote $\mathcal{Q}(\cdot)$ as $\mathcal{Q}_{\mathfrak{c}}(\cdot)$, $F_0(\cdot)$ as $F_{0,\mathfrak{c}}(\cdot)$ and denote the solutions of \eqref{1.3-0} as $F_{\mathfrak{c}}$. 

For each $\beta\ge 0$, we define the weight function $w_{\beta}(p)$ as
\begin{align}\label{1.9-01}
	w_{\beta}(p):=\left(1+|p|^2\right)^{\frac{\beta}{2}}.
\end{align}
%Here for simplicity of presentation, we have omitted the explicit dependence of $w$ on $\beta$. 
We look for solutions in the perturbation form
    \begin{align}\label{1.10-0}
    	f_{\mathfrak{c}}(t, x, p):=\frac{F_{\mathfrak{c}}(t, x, p)-J_{\mathfrak{c}}(p)}{\sqrt{J_{\mathfrak{c}}(p)}}.
    \end{align}
    
    \begin{Theorem}[Global existence]\label{thm1.1}
    	Let $\beta>4$ and $(M_{\mathfrak{c}},\mathbf{J}_{\mathfrak{c}},E_{\mathfrak{c}})=(0,\mathbf{0},0)$. There exist constants $\varepsilon_{0}>0$, $C_{0}>0$ and $\lambda_{0}>0$, which are all independent of the light speed $\mathfrak{c}\ (\gg1)$, such that if $F_{0,\mathfrak{c}}(x, p)=J_{\mathfrak{c}}(p)+\sqrt{J_{\mathfrak{c}}(p)}f_{0,\mathfrak{c}}(x, p) \ge 0$ 
%    	satisfies
%    	{\small
%    	\begin{align}\label{1.28}
%    		\int_{\mathbb{T}^3} \int_{\mathbb{R}^{3}} f_{0,\mathfrak{c}}(x, p) \sqrt{J_{\mathfrak{c}}(p)}dp d x=\int_{\mathbb{T}^3} \int_{\mathbb{R}^{3}} p_if_{0,\mathfrak{c}}(x, p) \sqrt{J_{\mathfrak{c}}(p)}dp d x=\int_{\mathbb{T}^3} \int_{\mathbb{R}^{3}} f_{0,\mathfrak{c}}(x, p) p^0\sqrt{J_{\mathfrak{c}}(p)}dp d x=0
%    	\end{align}
%        }
%    	for $i=1,2,3$, 
        and
    	\begin{align}\label{1.29-R}
    		\left\|w_{\beta} f_{0,\mathfrak{c}}\right\|_{L^{\infty}} \le \varepsilon_{0},
    	\end{align}
    	then the relativistic Boltzmann equation \eqref{1.3-0}-\eqref{1.4-0} admits a unique global solution $F_{\mathfrak{c}}(t, x, p)=J_{\mathfrak{c}}(p)+\sqrt{J_{\mathfrak{c}}(p)}f_{\mathfrak{c}}(t, x, p) \ge 0$ satisfying \eqref{1.8-10}
%    	{\small
%    		\begin{align}\label{1.9-0}
%    		\int_{\mathbb{T}^3} \int_{\mathbb{R}^{3}} f_{\mathfrak{c}}(t, x, p) \sqrt{J_{\mathfrak{c}}(p)}dp d x=\int_{\mathbb{T}^3} \int_{\mathbb{R}^{3}} p_if_{\mathfrak{c}}(t, x, p) \sqrt{J_{\mathfrak{c}}(p)}dp d x=\int_{\mathbb{T}^3} \int_{\mathbb{R}^{3}} f_{\mathfrak{c}}(t, x, p) p^0\sqrt{J_{\mathfrak{c}}(p)}dp d x=0
%    	\end{align}
%        }
    	and
    	\begin{align}\label{1.31}
    		\sup_{t\ge 0}\big\{e^{\lambda_{0} t}\|w_{\beta} f_{\mathfrak{c}}(t)\|_{L^{\infty}}\big\}\le C_{0} \left\|w_{\beta} f_{0,\mathfrak{c}}\right\|_{L^{\infty}}.
    	\end{align}
    \end{Theorem}
    
Our next results are concerned with the Newtonian limit. More specifically, we prove that, in the limit $\mathfrak{c}\to \infty$, the solutions of the relativistic Boltzmann equation converge, in a suitable norm, to the solutions of the Newtonian Boltzmann equation. To achieve this, the solutions of \eqref{1.3-0} will be directly compared to the solutions of the Newtonian Boltzmann equation with hard-sphere collision kernel, which is
\begin{align}\label{1.14-0}
	\partial_t F+p\cdot \nabla_x F=Q(F,F).
\end{align}
Here the Newtonian collision term $Q(h_1,h_2)$ in \eqref{1.14-0} has the form
\begin{align*}%\label{1.15-0}
	Q(h_1,h_2)=&\int_{\mathbb R^3}\int_{\mathbb S^2}\mathcal{K}_{\infty}(p,q,\omega)\Big[h_1(\bar{p}')h_2(\bar{q}')-h_1(p)h_2(q)\Big]d\omega dq
\end{align*}
with hard-sphere collision kernel
\begin{align*}
	\mathcal{K}_{\infty}(p,q,\omega):=|\omega\cdot (p-q)|.
\end{align*}
The post-collision velocities $(\bar{p}',\bar{q}')$ of two particles with the pre-collision velocity $(p,q)$ are given by 
\begin{align}\label{1.16-0}
	\bar{p}'=p+[\omega\cdot(q-p)]\omega,\quad \bar{q}'=q-[\omega\cdot(q-p)]\omega.
\end{align}

We impose equation \eqref{1.14-0} with initial data
\begin{align}\label{1.17-0}
	F(t,x,p)\mid_{t=0}=F_{0}(x,p).
\end{align}

To look for a solution $F(t,x,p)$  to the problem \eqref{1.14-0}, \eqref{1.17-0}, we introduce the Newtonian global Maxwellian 
\begin{align*}
	\mu(p):=\frac{1}{(2\pi)^{\frac{3}{2}}}e^{-\frac{|p|^2}{2}}.
\end{align*}

Let $F$ be a solution of the Newtonian Boltzmann equation \eqref{1.14-0}, \eqref{1.17-0}. Then $F$ satisfies the conversations of mass, momentum, and energy
\begin{align}\label{1.18-10}
	\int_{\mathbb{T}^3} \int_{\mathbb{R}^3}[F(t, x, p)-\mu(p)]
	\begin{pmatrix}
		1  \\  p \vspace{1.0ex}\\  |p|^2
	\end{pmatrix}
	d p d x&=\int_{\mathbb{T}^3} \int_{\mathbb{R}^3}[F_0(x, p)-\mu(p)]
	\begin{pmatrix}
		1  \\  p \vspace{1.0ex}\\  |p|^2
	\end{pmatrix}
	d p d x
	=\begin{pmatrix}
		M_{0} \vspace{1.0ex} \\  \mathbf{J}_{0} \vspace{1.0ex}\\  E_{0}
	\end{pmatrix}.
\end{align}
%\begin{align}
%	\int_{\mathbb{T}^3} \int_{\mathbb{R}^3}[F(t, x, p)-\mu(p)]
%	 d p d x&=\int_{\mathbb{T}^3} \int_{\mathbb{R}^3}\left[F_0-\mu(p)\right] d p d x:=M_0, \label{1.18-10}\\
%	\int_{\mathbb{T}^3} \int_{\mathbb{R}^3} p[F(t, x, p)-\mu(p)] d p d x &=\int_{\mathbb{T}^3} \int_{\mathbb{R}^3} p\left[F_0-\mu(p)\right] d p d x:=\mathbf{J}_0,  \label{1.19-10}\\
%	\int_{\mathbb{T}^3}\int_{\mathbb{R}^3} |p|^2[F(t, x, p)-\mu(p)] d p d x & =\int_{\mathbb{T}^3} \int_{\mathbb{R}^3} |p|^2\left[F_0-\mu(p)\right] d p d x:=E_0. \label{1.20-10}
%\end{align}

In the framework of perturbation
\begin{align}\label{1.19-0}
f(t, x, p):=\frac{F(t, x, p)-\mu(p)}{\sqrt{\mu(p)}},
\end{align}
by applying Theorems 1.1 $\&$ 1.6 in \cite{Duan} to the case $\|w_{\beta}f_0\|_{L^{\infty}}\le \bar{\epsilon}$ with $\bar{\epsilon}$ sufficiently small, we can obtain the following lemma, see also \cite{Kim} for the case with a large amplitude external potential. 

\begin{Lemma}\emph{(\cite{Duan,Kim})}\label{lem1.2}
		Let $\beta>4$ and $(M_{0},\mathbf{J}_{0},E_{0})=(0,\mathbf{0},0)$. There exist constants $\bar{\epsilon}>0$, $\bar{C}>0$ and $\bar{\sigma}>0$, such that if $F_{0}(x, p)=\mu(p)+\sqrt{\mu(p)}f_{0}(x, p) \ge 0$ 
%		satisfies
%		{\small
%			\begin{align}\label{1.19-01}
%				\int_{\mathbb{T}^3} \int_{\mathbb{R}^{3}} f_{0}(x, p) \sqrt{\mu(p)}dp d x=\int_{\mathbb{T}^3} \int_{\mathbb{R}^{3}} p_if_{0}(x, p) \sqrt{\mu(p)}dp d x=\int_{\mathbb{T}^3} \int_{\mathbb{R}^{3}} f_{0}(x, p) |p|^2\sqrt{\mu(p)}dp d x=0
%			\end{align}
%		}
%		for $i=1,2,3$, 
		and
		\begin{align}\label{1.29-N}
			\left\|w_{\beta} f_{0}\right\|_{L^{\infty}} \le \bar{\epsilon},
		\end{align}
		then the Newtonian Boltzmann equation \eqref{1.14-0}, \eqref{1.17-0} admits a unique global solution $F(t, x, p)=\mu(p)+\sqrt{\mu(p)}f(t, x, p) \ge 0$ satisfying \eqref{1.18-10}
%		{\small
%			\begin{align}\label{1.20-00}
%				\int_{\mathbb{T}^3} \int_{\mathbb{R}^{3}} f(t,x, p) \sqrt{\mu(p)}dp d x=\int_{\mathbb{T}^3} \int_{\mathbb{R}^{3}} p_i f(t,x, p)\sqrt{\mu(p)}dp d x=\int_{\mathbb{T}^3} \int_{\mathbb{R}^{3}} f(t,x, p)|p|^2\sqrt{\mu(p)}dp d x=0
%			\end{align}
%		}
		and
		\begin{align*} 
			\sup_{t\ge 0}\big\{e^{\bar{\sigma} t}\|w_{\beta} f(t)\|_{L^{\infty}}\big\}\le \bar{C} \left\|w_{\beta} f_{0}\right\|_{L^{\infty}}.
		\end{align*}
\end{Lemma}

%We now concern the Newtonian limit from solutions of relativistic Boltzmann equation to solutions of Newtonian Boltzmann equation measured in the $L_p^1 L_x^{\infty}$ norm 
%\begin{align*}
%	\left\|\xi\right\|_{L_p^1 L_x^{\infty}}:=\int_{\mathbb{R}^3} \left\|\xi(p)\right\|_{L_x^{\infty}}dp.
%\end{align*}
For later use, we define the translation operator $\tau_h^x$ and $\tau_h^p$  as follows
\begin{align*}
	\tau_h^x \xi(x, p):=\xi(x+h, p),\quad \tau_h^p \xi(x, p):=\xi(x, p+h),\quad h \in \mathbb{R}^3.
\end{align*}

\begin{Theorem}[Newtonian limit in $L_p^1 L_x^{\infty}$] \label{thm1.3}
Let the initial data $F_0$, $\left\{F_{0, \mathfrak{c}}\right\}\in L^1_pL^{\infty}_x$ for all $\mathfrak{c} \geq 1$. Assume  $F_0$, $\left\{F_{0, \mathfrak{c}}\right\}$ initially converge for some $\hat{C}_1>0$ as
	\begin{align}\label{1.5}
		\left\|F_{0, \mathfrak{c}}-F_0\right\|_{L_p^1 L_x^{\infty}} \leq \hat{C}_1 / \mathfrak{c}^k, \quad \exists k \in(0,2] .
	\end{align}
	Suppose each of these initial data $\left\{F_{0, \mathfrak{c}}\right\}$ leads to global-in-time unique mild solutions $F_{\mathfrak{c}}(t)$ of \eqref{1.3-0}-\eqref{1.4-0} with
	\begin{align*}
	F_{\mathfrak{c}}(t,x,p)=J_{\mathfrak{c}}(p)+\sqrt{J_{\mathfrak{c}}(p)}f_{\mathfrak{c}}(t,x,p)\geq 0,\quad \|f_{\mathfrak{c}}(t)\|_{L^{\infty}_{x,p}} \leq \hat{C}_2e^{-\tilde{\sigma}t},
	\end{align*}
	 for some $\tilde{\sigma}>0$ independent of $\mathfrak{c}$. Similarly, suppose $F_0$ leads to a global-in-time mild solution $F(t)$ of \eqref{1.14-0}, \eqref{1.17-0} with 
	 \begin{align*}
	 F(t,x,p)=\mu(p)+\sqrt{\mu(p)}f(t,x,p)\geq 0,\quad \|f(t)\|_{L^{\infty}_{x,p}} \leq \hat{C}_2e^{-\tilde{\sigma}t}.
	 \end{align*}
	
Assume
	\begin{align}\label{1.6}
		\left\|\tau_h^x F_0-F_0\right\|_{L_p^1 L_x^{\infty}}+\left\|\tau_h^p F_0-F_0\right\|_{L_p^1 L_x^{\infty}} \leq \hat{C}_3|h|, \quad|h|<1,
	\end{align}
then for any $ \delta \in(0, k)$ and $0<\alpha_0<\frac{1}{2}$, $\beta_0>1$, %the solutions $F_{\mathfrak{c}}$ to \eqref{1.3-0} and $F$ to \eqref{1.14-0} corresponding to these initial data converge as follows:\\
	we have the following Newtonian limit
	\begin{align}\label{1.7}
		\begin{split}
		&\left\|F_{\mathfrak{c}}(t)-F(t)\right\|_{L_p^1 L_x^{\infty}} \leq C(\delta) / \mathfrak{c}^{k-\delta},\quad \mbox{for}\quad 0\le t\le (\log \mathfrak{c})^{\alpha_0},\\
		&\left\|F_{\mathfrak{c}}(t)-F(t)\right\|_{L_p^1 L_x^{\infty}} \le C(\delta)/\mathfrak{c}^{2-\delta},\quad \mbox{for}\quad t\ge (\log \mathfrak{c})^{\beta_0},\\
		&\left\|F_{\mathfrak{c}}(t)-F(t)\right\|_{L_p^1 L_x^{\infty}} \le \hat{C}_4 e^{-\tilde{\sigma}(\log \mathfrak{c})^{\alpha_0}},\quad \mbox{for}\quad t\in \big[(\log \mathfrak{c})^{\alpha_0}, (\log \mathfrak{c})^{\beta_0}\big].
		\end{split}
	\end{align}
 %   For $$, one has
%    \begin{align}\label{1.7-30}
   % 	\left\|F_{\mathfrak{c}}(t)-F(t)\right\|_{L_p^1 L_x^{\infty}} \le C(\delta)/\mathfrak{c}^{2-\delta}.
 %   \end{align}
 %   For $$, one has
 %   \begin{align}\label{1.7-40}
%    \end{align}
All above positive constants $\hat{C}_i\ (i=1,\cdots,4)$ and $C(\delta)$ are independent of $\mathfrak{c}$.
\end{Theorem}

\begin{remark}
	Due to Theorem \ref{thm1.1} and Lemma \ref{lem1.2}, the Newtonian limit in Theorem \ref{thm1.3} is self-consistent under the smallness assumption \eqref{1.29-R} and \eqref{1.29-N} of the initial data. Indeed, the Newtonian limit in Theorem \ref{thm1.3}  is valid as long as the uniform-in-$\mathfrak{c}$ estimates $\|f_\fc\|_{L^\infty_{x,p}}\lesssim e^{-\tilde{\sigma}t}$ and $\|f\|_{L^\infty_{x,p}}\lesssim e^{-\tilde{\sigma}t}$ hold true, but it is hard to establish such existence results for large initial data.
\end{remark}

%\begin{remark}
%	Due to Theorem \ref{thm1.1}, the Newtonian limit in Theorem \ref{thm1.3} is self-consistent under the smallness condition of the initial data. Indeed, the Newtonian limit \eqref{1.7} holds for solutions as long as the uniform-in-$\mathfrak{c}$ estimate $\|f_{\mathfrak{c}}\|_{L^{\infty}_{x,p}} \leq \hat{C}_2<\infty$ is available, but the existence of  such solution is still open for large initial data.
%	It can be seen from the proof that we do not need the smallness of the solutions or the decay rates with respect to time. In order to make the results self-consistent, we prove Theorem \ref{thm1.1} for the existence, uniqueness, and asymptotic behavior of solutions for small initial data. 
%	Actually, for large initial data, there are currently no results available that can be directly referenced regarding the existence, uniqueness, and uniform estimates of such solutions.
%\end{remark}
%%%%%%%%%%%%%%%%%%%%%%%%%%%%%%%%%%%%%%%%%%%%%%%%%%%%%%%%%%%%%%%
%Compared with the results in Theorem \ref{thm1.3} where the Newtonian limit is measured in a slightly weaker norm $L_p^1 L_x^{\infty}$, when 
If initial data $f_0\in W^{1,\infty}_{x,p}(\mathbb{T}^3\times\mathbb{R}^3)$, we can further obtain the following Newtonian limit  in stronger norm  $L^{\infty}_{x,p}$, and also with better convergence rate.

\begin{Theorem}[Newtonian limit in $L^{\infty}_{x,p}$]\label{thm1.4}
	Let $\beta>8$ and $(M_\fc,\mathbf{J}_\fc,E_\fc)=(M_0,\mathbf{J}_0,E_0)=(0,\mathbf{0},0)$. Assume that $\{F_{\mathfrak{c}}(t)\}$ and $F(t)$ are the solutions for relativistic Boltzmann equation and Newtonian Boltzmann equation, respectively, such that
	\begin{align}
		F_{\mathfrak{c}}(t,x, p)&=J_{\mathfrak{c}}(p)+\sqrt{J_{\mathfrak{c}}(p)}f_{\mathfrak{c}}(x, p) \ge 0, \quad \|w_{\beta}f_\fc(t)\|_{L^\infty_{x,p}}\lesssim e^{-\sigma_0 t},\label{1.29}\\
		F(t,x, p)&=\mu(p)+\sqrt{\mu(p)}f(x, p) \ge 0,\qquad \|w_{\beta}f(t)\|_{L^\infty_{x,p}}\lesssim e^{-\sigma_0t},\label{1.30}
	\end{align}
    where $\sigma_0>0$ is independent of $\mathfrak{c}$. Suppose further that $f_0\in W^{1,\infty}_{x,p}(\mathbb{T}^3\times\mathbb{R}^3)$ with
	\begin{align}
		\|w_{\beta-1}\nabla_xf_0\|_{L^\infty_{x,p}}+\|w_{\beta-2}\nabla_p f_0\|_{L^\infty_{x,p}}\leq \hat{C}_5,\label{1.25}\\
		\|w_{\beta-6}(f_{0,\fc}-f_0)\|_{L^\infty_{x,p}}\leq \f{\hat{C}_6}{\fc^k},\quad \exists k\in(0,2].\label{1.26}
	\end{align}
    Then for any $\delta\in(0,k)$, there exists a constant $\hat{C}(\delta )>0$ such that
	\begin{align}\label{1.27}
		\sup_{t\ge 0}\|w_{\beta-6}(f_\fc-f)(t)\|_{L^\infty_{x,p}}\leq \f{\hat{C}(\delta)}{\fc^{k-\delta}}.
	\end{align}
	The positive constants $\hat{C}_5, \hat{C}_6$ and  $\hat{C}(\delta)$ are all independent of the light speed $\fc$.
\end{Theorem}
\begin{remark}
	Compared with the results in \cite{Strain} where the Newtonian limit holds for any finite time interval $[0, T]$, our results in Theorems \ref{thm1.3} $\&$ \ref{thm1.4}  hold  for all $t\in [0,\infty)$.
\end{remark}

%	\begin{remark}
%	It is noted that we do not require the smallness assumption of $\|w_{\beta}f(t)\|_{L^\infty_{x,p}}$ or $\|w_{\beta}f_\fc(t)\|_{L^\infty_{x,p}}$. Actually, let $\bar{\mathcal E}(F_0)$ be the relative entropy of $F_0$ (see the definition in \cite[p.377]{Duan}). For any positive constant $\bar{M}>0$, if $\|w_{\beta}f_0\|_{L^{\infty}_{x,p}}\le \bar{M}$ and $\bar{\mathcal E}(F_0)+\|f_0\|_{L^1_xL^{\infty}_p}$  suitably small, then there indeed exist such solutions $f(t)$ which satisfies \eqref{1.30} and $\|w_{\beta}f(t)\|_{L^{\infty}_{x,p}}\le C\bar{M}^2$, see Theorems 1.1 $\&$ 1.6 in \cite{Duan} for details. Similar results are also obtained for $f_{\mathfrak{c}}$  which satisfies \eqref{1.29} in the case of $\mathfrak{c}=1$. We refer to Theorems 1.1 $\&$ 1.4 in \cite{Wang} for details. While, for large initial data $f_{0,\mathfrak{c}}$, there are currently no results available that can be directly referenced regarding the existence, uniqueness, and uniform estimates of such solutions $f_{\mathfrak{c}}$ which satisfy \eqref{1.29}.
%\end{remark}

\begin{remark} 
	It is worth noting that when \eqref{1.30} and \eqref{1.25} are both valid, the solution  $f(t)$ of Newtonian Boltzmann equation obeys the following estimates
	\begin{align*}
		\|w_{\beta-1}\nabla_xf(t)\|_{L^\infty_{x,p}}&\lesssim e^{-\hat{\lambda}t},\text{ for }t\geq0,\\
			\|w_{\beta-2}\nabla_pf(t)\|_{L^\infty_{x,p}}&\lesssim 1,\text{ for }t\geq0,
	\end{align*}
  %   which remarkably expand the relevant results on the Newtonian Boltzmann theory, 
  see Lemmas \ref{lem5.1} $\&$ \ref{lem5.2} for details. The above estimates will play important roles in the Newtonian limit in $L^\infty_{x,p}$.
\end{remark}

\subsection{Key points of the proof}
%%%%%%%%%%%%%%%%%%%%%%%%%%%%%%%%%%%%%%%%%%%%%%%%%%%%%%%%%%%%%%%
We now make some comments on the analysis of this paper.

$\bullet$ {\it Uniform-in-$\mathfrak{c}$ estimates on hydrodynamic part $\mathbf{P}_{\mathfrak{c}} f_{\mathfrak{c}}$.} In the small perturbation framework, for Newtonian Boltzmann equation, only the dissipation  $\|(\mathbf{I-P})f\|_{\nu}^2$ is controlled. How to estimate the missing hydrodynamic part $\mathbf{P}f$ is a well-known basic question in the Boltzmann theory. Guo \cite{Guo4,Guo6} first developed a new nonlinear energy method in high Sobolev norms  to estimate $\mathbf{P}f$ in terms of $\|(\mathbf{I-P})f\|_{\nu}^2$ in the case of periodic box for classical Vlasov-Poisson(Maxwell)-Boltzmann equations. 

In the present paper, we follow the ideas in \cite{Esposito} to control the hydrodynamic part $\mathbf{P}_\fc f_\fc$ as
\begin{align*}%\label{2.21-1}
	\mathbf{P}_\fc f_\fc=\Big\{a+b\cdot p+c\frac{p^0-A_3}{\sqrt{A_2-A^2_3}}\Big\}\sqrt{J_\fc}
\end{align*}
for our relativistic Boltzmann equation, where $A_2, A_3$ are coefficients involving Bessel functions, see Lemma \ref{lem2.2} for more details. The situation becomes slightly more complex due to the appearance of factor $\displaystyle{\frac{\mathfrak{c}}{p^0}}$ which indeed comes from the effect of special relativity.
We point out that the test functions are very similar to the ones in \cite{Esposito}, but with complex constants $\beta_{c}$, $\beta_{b}$ and $\beta_{a}$ which involve Bessel functions, see Proposition \ref{prop3.2}. However, to control the term $\int_{\mathbb{R}^3}\partial_t\psi f_\fc dp$, one should be more careful to calculate the coefficients.

We briefly explain some key points for the estimates in the following. To control $c$  first, we choose the test function as $\psi=\psi_c=(|p|^2-\beta_{c})\sqrt{J(p)}p\cdot \nabla_x \phi_{c}(x)$ with $\phi_{c}$ defined in \eqref{3.27-0}, and $\beta_{c}$ such that
\begin{align}\label{1.34}    	\int_{\mathbb{R}^3}\frac{\mathfrak{c}}{p^0}\left(|p|^2-\beta_{c}\right)p^2_{i}J_\fc(p)dp=0.
\end{align}
Since
\begin{align*} 
	&\int_0^t\iint_{\mathbb{T}^3 \times \mathbb{R}^3}\partial_t\psi_c \mathbf{P}_\fc f_\fc dpdxd\tau\nonumber\\
	&=\sum_{i=1}^3\int_0^t\iint_{\mathbb{T}^3 \times \mathbb{R}^3}(|p|^2-\beta_{c})p_iJ_\fc(p)  \Big\{a+b\cdot p+\frac{p^0-A_3}{\sqrt{A_2-A_3^2}}c\Big\}\partial_t\partial_i \phi_{c}dpdxd\tau,
\end{align*}
we notice that the $a$ and $c$ contributions vanish due to the oddness in $p$. While, the $b$ contribution does not vanish, which is different from \cite{Esposito}. In fact, by using \eqref{1.34},  we have
\begin{align*} 
	&\Big| \int_0^t\iint_{\mathbb{T}^3\times \mathbb{R}^3}(|p|^2-\beta_{c})p^2_i J_\fc(p) \partial_t\partial_i \phi_{c}b_idpdxd\tau\Big|%\nonumber\\
%	&=\Big| \int_0^t\iint_{\mathbb{T}^3 \times \mathbb{R}^3}(1-\frac{\mathfrak{c}}{p^0})(|p|^2-\beta_{c})p^2_i J_\fc(p) \partial_t\partial_i \phi_{c}b_idpdxd\tau\Big|\nonumber\\
%	&\le \int_0^t \|\partial_t \nabla_x \phi_{c}\|_{L^2}\cdot \|b(\tau)\|_{L^2}d\tau\cdot   \int_{\mathbb{R}^3}\big|(1-\frac{\mathfrak{c}}{p^0})\left(|p|^2-\beta_{c}\right)\big|p^2_{i}J_\fc(p)dp\nonumber\\
\le C\frac{1}{\mathfrak{c}^2}\int_0^t \|\partial_t \nabla_x \phi_{c}\|_{L^2}\cdot \|b(\tau)\|_{L^2}d\tau.
\end{align*}
Since the coefficient $\frac{1}{\mathfrak{c}^2}$ is small for suitably large $\mathfrak{c}$, the term  $\frac{1}{\mathfrak{c}^2}\int_0^t \|b(\tau)\|^2_{L^2}d\tau$ does not cause difficulties for us to close the estimate of $c$, see \eqref{3.38-0}--\eqref{3.41-1} for details.

\smallskip
To control $b$, we choose the test function as $\psi=\psi^{i,j}_b=(p_i^2-\beta_{b})\sqrt{J_\fc(p)}\partial_{j} \phi^j_{b}$ with $\phi^j_{b}$ defined in \eqref{3.26-0}, and $\beta_{b}$ such that
\begin{align*}	\int_{\mathbb{R}^3}\frac{\mathfrak{c}}{p^0}\left(p_i^2-\beta_{b}\right)p_k^2J_\fc(p)dp=0, \quad k\neq i.
\end{align*}
Using Lemma \ref{lem2.2}, we obtain $\beta_{b}=\frac{K_3(\mathfrak{c}^2)}{K_2(\mathfrak{c}^2)}$. Then one has
\begin{align}\label{1.38}
	&\int_{0}^{t} \iint_{\mathbb{T}^3 \times \mathbb{R}^3}\partial_t\psi \mathbf{P}_\fc f_\fc dpdxd\tau
	\nonumber\\
	&=\int_{0}^{t} \iint_{\mathbb{T}^3 \times  \mathbb{R}^3}\left(p_{i}^{2}-\beta_{b}\right) J_\fc(p) \partial_{t} \partial_{j} \phi_{b}^{j} \Big\{a+b \cdot p+c\frac{p^0-A_3}{\sqrt{A_2-A^2_3}}\Big\}dpdxd\tau.
\end{align}
For the terms on the RHS of \eqref{1.38}, the $b$ contribution vanishes due to the oddness in $p$. For the $a$ contribution, fortunately, with the help of Lemma \ref{lem2.2}, it is direct to check that
\begin{align*}
	\int_{\mathbb{R}^3} \left(p_i^2-\beta_{b}\right) J_\fc(p)dp=0,
\end{align*}
which yields that the $a$ contribution also vanishes. It is indeed a crucial point in the estimate of $b$. Without the zero contribution of $a$, it is hard for us to close the estimate of $b$. 
The estimate of $a$ is similar to the one of $c$. 

Combining all the above estimates, we obtain
\begin{align*}
	\int_{s}^{t}\|\mathbf{P}_{\mathfrak{c}} f_{\mathfrak{c}}(\tau)\|_{\nu_{\mathfrak{c}}}^{2}d\tau &\lesssim
	\int_{s}^{t}\Big(\|(\mathbf{I}-\mathbf{P}_{\mathfrak{c}}) f_{\mathfrak{c}}(\tau)\|_{\nu_{\mathfrak{c}}}^{2}+\|\mathcal{S}(\tau)\|_{L^2}^{2}\Big)d\tau+|G(t)-G(s)|.
\end{align*}
Here we emphasize that the above constants are independent of $\mathfrak{\mathfrak{c}}$, see Proposition \ref{prop3.2} for details.

%Therefore, we can close the estimate of $c$ first, then $b$ and finally $a$ to finish the proof of Proposition \ref{prop3.2}.

$\bullet${\it Global-in-time Newtonian limit.} %\textcolor{blue}{In the present paper, we establish the global-in-time existence for solutions to the initial value problem of the relativistic Boltzmann equation, the uniform-in-$\mathfrak{c}$ estimates and time decay rate  are obtained in $L^{\infty}_{x,p}$. We then rigorously justify the global-in-time Newtonian limit from the relativistic Boltzmann solutions to solutions of Newtonian Boltzmann  in $L^1_pL^{\infty}_x$ with rate of convergence. Moreover, if the initial data  of Newtonian Boltzmann equation is in  $W^{1,\infty}(\mathbb{T}^3\times\mathbb{R}^3)$, we first establish the global $W^{1,\infty}(\mathbb{T}^3\times\mathbb{R}^3)$ regularities  without smallness on initial  $W^{1,\infty}(\mathbb{T}^3\times\mathbb{R}^3)$-norm. With above regularities, we can derive the global-in-time Newtonian limit in  $L^{\infty}_{x,p}$ with rate of convergence,  which is based on a decomposition and a $L^2-L^\infty$ argument. ??}
%To obtain the  global-in-time convergence of the relativistic Boltzmann solutions to the Newtonian Boltzmann solutions, one needs first to compare the relativistic Maxwellian $J_\fc(p)$ and the Newtonian Maxwellian $\mu(p)$ and obtain the convergence rate. This is done in Lemma \ref{lem2.6}. 
%To measure the convergence in the $L_p^1L^{\infty}_x$ norm, 
Following the arguments in \cite{Calogero,Strain}, and applying the uniform-in-$\mathfrak{c}$ estimate $\|f_{\mathfrak{c}}(t)\|_{L^{\infty}_{x,p}} \lesssim e^{-\tilde{\sigma}t}$,  we can obtain the desired Newtonian limit in $L^1_pL^\infty_x$ as well as the decay rates by performing a difference on the mild forms of two solutions.
%While the proof of Newtonian limit  in   $L^{\infty}_{x,p}$ is much more complicated, \textcolor{blue}{and no result has been obtained in this case.} 
To obtain the global-in-time Newtonian limit in $L^{\infty}_{x,p}$, we assume that the initial data of Newtonian Boltzmann equation $f_0\in W^{1,\infty}(\mathbb{T}^3\times\mathbb{R}^3)$.  Consider the equation for $\mathfrak{g}:=f_{\mathfrak{c}}-f$
	\begin{align*}%\label{1.34-00}
		\partial_t \mathfrak{g}+\hat{p}\cdot\nabla_x \mathfrak{g}+\FL_\fc \mathfrak{g}=\mathbf{\Gamma}_\fc(f_\fc,f_\fc)-\mathbf{\Gamma}_\fc(f,f)+\mathcal{R},
	\end{align*}
	where 
	\begin{align*}
		\mathcal{R}:=(p-\hat{p})\cdot\nabla_xf+(\FL-\FL_\fc)f+(\mathbf{\Gamma}_\fc-\mathbf{\Gamma})(f,f).
	\end{align*}
Unfortunately, due to the term $\mathbf{L}f-\mathbf{\Gamma}(f,f)$, we note that 
\begin{align*}
	\int_{\mathbb{T}^3}\int_{\mathbb{R}^3}\Big(\mathbf{\Gamma}_\fc(f_\fc,f_\fc)-\mathbf{\Gamma}_\fc(f,f)+\mathcal{R}\Big)
	\begin{pmatrix}
		1  \\  p \vspace{1.0ex}\\  \f{p^0-A_3}{\sqrt{A_2-A_3^2}}
	\end{pmatrix}
	\sqrt{J_\fc(p)}dpdx\neq\mathbf{0},
\end{align*}
so it is hard to  use similar arguments as in Proposition \ref{prop3.2} and Lemma \ref{lem4.5} to obtain a key $L^2$ decay estimate on $\mathfrak{g}$. To overcome this difficulty, we introduce a decompose $\mathfrak{g}$ as $\mathfrak{g}=\mathfrak{g}_1+\mathfrak{g}_2$ and thus the equation for $\mathfrak{g}_2$ takes the form
\begin{align*}%\label{1.36-01}
	\partial_t\mathfrak{g}_2+\hat{p}\cdot\nabla_x\mathfrak{g}_2+\FL_\fc \mathfrak{g}_2=\mathbf{\Gamma}_\fc(\mathfrak{g}_2,f_\fc)+\mathbf{\Gamma}_\fc(f,\mathfrak{g}_2)+\tilde{\mathcal{R}},
\end{align*}
where
\begin{align*}
	\tilde{\mathcal{R}}:=\mathbf{\Gamma}_\fc(\mathfrak{g}_1,f_\fc)+\mathbf{\Gamma}_\fc(f,\mathfrak{g}_1)+\mathcal{R}-\{\partial_t\mathfrak{g}_1+\hat{p}\cdot\nabla_x\mathfrak{g}_1+\FL_\fc \mathfrak{g}_1\}.
\end{align*}
We choose $\mathfrak{g}_1=\Big(\mathfrak{e}(t,x)+\mathfrak{p}(t,x)\cdot p+\mathfrak{l}(t,x)\f{p^0-A_3}{\sqrt{A_2-A_3^2}}\Big)\sqrt{J_\fc(p)}$ with
\begin{align*}
	&\partial_t \mathfrak{e}=\int_{\R^3}\{\FL f-\mathbf{\Gamma}(f,f)\}\sqrt{J_\fc(p)}dp,\\
	&\partial_t \mathfrak{p}=\f{1}{A_1}\int_{\R^3}\{\FL f-\mathbf{\Gamma}(f,f)\}p\sqrt{J_\fc(p)}dp,\\
	&\partial_t \mathfrak{l}=\int_{\R^3}\{\FL f-\mathbf{\Gamma}(f,f)\}\f{p^0-A_3}{\sqrt{A_2-A_3^2}}\sqrt{J_\fc(p)}dp,
\end{align*}
and a careful choosing of the initial data on $(\mathfrak{e},\mathfrak{p},\mathfrak{l})(0,x)$, so that 
\begin{align*} 
	\int_{\mathbb{T}^3}\int_{ \mathbb{R}^3}\mathfrak{g}_{2}(0,x,p)
	\begin{pmatrix}
		1  \\  p \vspace{1.0ex}\\  \f{p^0-A_3}{\sqrt{A_2-A_3^2}}
	\end{pmatrix}
	\sqrt{J_\fc(p)}dpdx=\mathbf{0}.
\end{align*}
Based on the choice of $\mathfrak{g}_{1}$, it is direct to check that the following condition holds
\begin{align*}
	\int_{\mathbb{T}^3}\int_{\mathbb{R}^3}\Big(\mathbf{\Gamma}_\fc(\mathfrak{g}_2,f_\fc)+\mathbf{\Gamma}_\fc(f,\mathfrak{g}_2)+\tilde{\mathcal{R}}\Big)
	\begin{pmatrix}
		1  \\  p \vspace{1.0ex}\\  \f{p^0-A_3}{\sqrt{A_2-A_3^2}}
	\end{pmatrix}
	\sqrt{J_\fc(p)}dpdx=\mathbf{0},
\end{align*}
then we can apply similar arguments as in Proposition \ref{prop3.2} and Lemma \ref{lem4.5} for  $\mathfrak{g}_2$  to obtain a key $L^2$ estimate of $\mathfrak{g}_2$. 

We point that the estimates on $\mathfrak{g}_1$ and $\mathcal{R}$ are not trivial, which are involved in  $\|w_{\beta-1}\nabla_xf\|_{L^{\infty}_{x,p}}$ and $\|w_{\beta-2}\nabla_pf\|_{L^{\infty}_{x,p}}$, hence we put a lot of effort into obtaining 
	\begin{align*}
\begin{split}
		\|w_{\beta-1}\nabla_xf(t)\|_{L^\infty_{x,p}}&\lesssim e^{-\hat{\lambda}t},\\
	\|w_{\beta-2}\nabla_pf(t)\|_{L^\infty_{x,p}}&\lesssim 1,
\end{split}
\quad \text{ for }t\geq0,
\end{align*}
 to close the whole proof. Here it is hard to obtain the time decay  estimate for $\|w_{\beta-2}\nabla_pf(t)\|_{L^\infty_{x,p}}$ due to the appearance of $\f{p_i}{2}\f{1}{\sqrt{\mu}}Q(F,F)$ on RHS of \eqref{5.9-a}--\eqref{5.9-b}, see Lemma \ref{lem5.2} for more details. 
 Based on above preparations, we can obtain the global-in-time estimate $$\|w_{\beta-6}\mathfrak{g}\|_{L^\infty_{x,p}}\lesssim \frac{1}{\mathfrak{c}^{k-\v}},$$ 
 by a $L^2-L^\infty$ argument.
%It should be noted that, due to the exponential time decay of $f_{\mathfrak{c}}$ and $f$, the Newtonian limit in the $L^{\infty}_{x,p}$ norm holds for all $t\in [0,\infty)$, which is different from \cite{Strain} where the Newtonian limit is valid for arbitrary finite time intervals $[0,T]$.

%%%%%%%%%%%%%%%%%%%%%%%%%%%%%%%%%%%%%%%%%%%%%%%%%%%%%%%%%%%%%%%
\subsection{Organization of the paper}
%%%%%%%%%%%%%%%%%%%%%%%%%%%%%%%%%%%%%%%%%%%%%%%%%%%%%%%%%%%%%%%

In Section 2, we present some basic facts about Bessel functions, Newtonian Boltzmann collision operators and relativistic Boltzmann collision operators.
Section 3 is dedicated to the global existence for the solutions of the relativistic Boltzmann equation in the torus $\mathbb{T}^3$, and the uniform-in-$\mathfrak{c}$ estimates on the relativistic Boltzmann solutions are also obtained.
In Section 4,  the Newtonian limit is obtained in $L^1_pL^{\infty}_x$.
The Newtonian limit in $L^{\infty}_{x,p}$ is proved in Section 5.

%%%%%%%%%%%%%%%%%%%%%%%%%%%%%%%%%%%%%%%%%%%%%%%%%%%%%%%%%%%%%%%
\subsection{Notations}
%%%%%%%%%%%%%%%%%%%%%%%%%%%%%%%%%%%%%%%%%%%%%%%%%%%%%%%%%%%%%%%

Throughout this paper, $C$ denotes a generic positive constant which is independent of $\mathfrak{c}$ and $C_{a}, C_{b}, \ldots$ denote the generic positive constants depending on $a, b, \ldots$, respectively, which may vary from line to line. $A \lesssim B$ means that there exists a constant $C>0$, which is independent of $\mathfrak{c}$, such that $A \le C B$.  $A\cong B$ means that both $A \lesssim B$ and $B \lesssim A$ hold.  $\|\cdot\|_{L^2}$ denotes either the standard  $L^{2}\left( \mathbb{T}_{x}^{3}\right)$-norm or $L^{2}\left( \mathbb{R}_{p}^{3}\right)$-norm or $L^{2}\left(\mathbb{T}_x^3 \times \mathbb{R}_{p}^{3}\right)$-norm. Similarly, $\|\cdot\|_{L^{\infty}}$ denotes either the $L^{\infty}\left( \mathbb{T}_{x}^{3}\right)$-norm or $L^{\infty}\left( \mathbb{R}_{p}^{3}\right)$-norm or  $L^{\infty}\left(\mathbb{T}_x^3 \times \mathbb{R}_{p}^{3}\right)$-norm. We denote $\langle\cdot, \cdot\rangle$ as either the $L^{2}\left(\mathbb{T}_{x}^{3}\right)$ inner product or $L^{2}\left(\mathbb{R}_{p}^{3}\right)$ inner product or $L^{2}\left(\mathbb{T}_x^3 \times \mathbb{R}_{p}^{3}\right)$ inner product. Moreover, we denote $\|\cdot\|_{\nu_{\mathfrak{c}}}:=\|\sqrt{\nu_{\mathfrak{c}}} \cdot\|_{L^2}$.

%%%%%%%%%%%%%%%%%%%%%%%%%%%%%%%%%%%%%%%%%%%%%%%%%%%%%%%%%%%%%%%
\section{Preliminaries}
%%%%%%%%%%%%%%%%%%%%%%%%%%%%%%%%%%%%%%%%%%%%%%%%%%%%%%%%%%%%%%%

%%%%%%%%%%%%%%%%%%%%%%%%%%%%%%%%%%%%%%%%%%%%%%%%%%%%%%%%%%%%%%%
\subsection{Bessel function}
%%%%%%%%%%%%%%%%%%%%%%%%%%%%%%%%%%%%%%%%%%%%%%%%%%%%%%%%%%%%%%%

The modified Bessel function of the second kind is given by
\begin{align}\label{e2.15}
	K_{j}(z)=\Big(\frac{z}{2}\Big)^j\frac{\Gamma(\frac{1}{2})}{\Gamma(j+\frac{1}{2})}\int_{1}^{\infty} e^{-z s}\left(s^{2}-1\right)^{j-\frac{1}{2}} ds,\quad j\ge 0,\ z>0.
\end{align}

Some of the properties of the modified Bessel function $K_j$ are needed.
\begin{Lemma}\emph{(\cite{Olver, Watson})} \label{lem2.1}
	It holds that
	\begin{align*}
		K_{j+1}(z) & =\frac{2j}{z}K_j(z)+K_{j-1}(z), \quad j\ge 1,\ z>0.
	\end{align*}
	The asymptotic expansion for $K_j(z)$ takes the form
	\begin{align*}
		K_{j}(z) &=\sqrt{\frac{\pi}{2z}} \frac{1}{e^{z}}\Big[\sum_{m=0}^{n-1}\mathcal{A}_{j,m}z^{-m}+\gamma_{j,n}(z)z^{-n}\Big],\quad j\ge 0,\ z>0,\ n\ge 1,
	\end{align*}
	where the following additional identities and inequalities also hold:
	\begin{align*}
		\mathcal{A}_{j,0}&=1,\nonumber\\
		\mathcal{A}_{j,m}&=\frac{1}{m!8^m}(4j^2-1)(4j^2-3^2)\cdots (4j^2-(2m-1)^2),\quad j\ge 0,\ m\ge 1,\nonumber\\
		|\gamma_{j,n}(z)|&\le 2|\mathcal{A}_{j,n}|\exp{\Big(\big[j^2-\frac{1}{4}\big]z^{-1}\Big)},\quad j\ge 0,\ n\ge 1,\nonumber\\
		K_{j}(z)&<K_{j+1}(z),\quad j\ge 0.
	\end{align*} 
	Furthermore, for $j\le n+\frac{1}{2}$, one has a more exact estimate
	\begin{align*}
		|\gamma_{j,n}(z)|&\le |\mathcal{A}_{j,n}|.
	\end{align*}
\end{Lemma}

\

Using the asymptotic expansions for $K_2(z)$ and $K_3(z)$ in Lemma \ref{lem2.1}, one can obtain
\begin{align}\label{2.2-30}
	\frac{K_3(\mathfrak{c}^2)}{K_2(\mathfrak{c}^2)}-1&
	%		=\frac{\frac{5}{2\gamma}+\frac{105}{16\gamma^2}+\frac{945}{256\gamma^3}+O(\gamma^{-4})}{1+\frac{15}{8\gamma}+\frac{105}{128\gamma^2}+O(\gamma^{-3})}
	=\frac{5}{2}\mathfrak{c}^{-2}+\frac{15}{8}\mathfrak{c}^{-4}-\frac{15}{8}\mathfrak{c}^{-6}+O(\mathfrak{c}^{-8})
\end{align}
and 
\begin{align}\label{2.3-30}
	\frac{K^2_3(\mathfrak{c}^2)}{K^2_2(\mathfrak{c}^2)}-1&
	%	=\frac{\frac{5}{\gamma}+\frac{115}{4\gamma^2}+\frac{2205}{32\gamma^3}+\frac{10395}{128\gamma^4}+O(\gamma^{-5})}{1+\frac{15}{4\gamma}+\frac{165}{32\gamma^2}+\frac{315}{128\gamma^3}+O(\gamma^{-4})}\nonumber\\
	=5\mathfrak{c}^{-2}+10\mathfrak{c}^{-4}+\frac{45}{8}\mathfrak{c}^{-6}-\frac{15}{4}\mathfrak{c}^{-8}+O(\mathfrak{c}^{-10}).
\end{align}

\

Using the definition of $K_j$ and integration by parts, one can prove the following integration identities directly. We omit the details here for brevity.

\begin{Lemma} \label{lem2.2}
	For the normalized global Maxwellian $J_{\mathfrak{c}}(p)$, for $i=1,2,3$, $i\neq k$, it holds that
	\begin{enumerate}[(1)]
	\item
	$\displaystyle  \int_{\mathbb{R}^3}p^2_iJ_{\mathfrak{c}}(p)dp=\frac{K_3(\mathfrak{c}^2)}{K_2(\mathfrak{c}^2)}:=A_1$.\\
	\item
	$\displaystyle 
	\int_{\mathbb{R}^3}(p^0)^2J_{\mathfrak{c}}(p)dp=\mathfrak{c}^2+3\frac{K_3(\mathfrak{c}^2)}{K_2(\mathfrak{c}^2)}:=A_2$.\\
	\item
	$\displaystyle 
	\int_{\mathbb{R}^3}p^0J_{\mathfrak{c}}(p)dp=\mathfrak{c}\frac{K_3(\mathfrak{c}^2)}{K_2(\mathfrak{c}^2)}-\frac{1}{\mathfrak{c}}:=A_3$.\\
	\item
	$\displaystyle 
	\int_{\mathbb{R}^3}\frac{p^2_i}{p^0}J_{\mathfrak{c}}(p)dp=\frac{1}{\mathfrak{c}}:=A_4$.\\
	\item
	$\displaystyle 
	\int_{\mathbb{R}^3}{p^2_i}p^0J_{\mathfrak{c}}(p)dp=\mathfrak{c}+\frac{5}{\mathfrak{c}}\frac{K_3(\mathfrak{c}^2)}{K_2(\mathfrak{c}^2)}:=A_5$.\\
	\item
	$\displaystyle 
	\int_{\mathbb{R}^3}\f{p_i^2|p|^2}{p^0}J_{\mathfrak{c}}(p)dp=\f{5}{\fc}\f{K_3(\fc^2)}{K_2(\fc^2)}:=A_6$.\\
	\item
	$\displaystyle  \int_{\R^3}\f{p_i^2p_k^2}{p^0}J_{\mathfrak{c}}(p)dp=\f{1}{\mathfrak{c}}\f{K_3(\fc^2)}{K_2(\fc^2)}:=A_7$.\\
	\item $\displaystyle \int_{\R^3}|p|^2\f{p_i^2p_k^2}{p^0}J_{\mathfrak{c}}(p)dp=\f{7}{\fc}+\f{42}{\fc^3}\f{K_3(\fc^2)}{K_2(\fc^2)}:=A_8$.\\
	\item
	$\displaystyle  \int_{\R^3}\f{p_i^4}{p^0}J_{\mathfrak{c}}(p)dp=\f{3}{\fc}\f{K_3(\fc^2)}{K_2(\fc^2)}:=A_9.$\\
	\item $\displaystyle \int_{\R^3}p_i^2|p|^2J_{\mathfrak{c}}(p)dp=5+\f{30}{\fc^2}\f{K_3(\fc^2)}{K_2(\fc^2)}:=A_{10}.$
	\end{enumerate}
\end{Lemma}

%Using \eqref{a2.4}, we have
%\begin{Lemma} \label{alem2.2}
%	Denote $\alpha:=\frac{K_3(\mathfrak{c}^2)}{K_2(\mathfrak{c}^2)}$. Then it holds that
%	\begin{enumerate}
%		%			\item $\displaystyle  K_3(z)=\sqrt{\frac{\pi}{2z}}\frac{1}{e^z}[1+\frac{35}{8z}+\frac{945}{2!(8z)^2}+O(z^{-3})]$.\\
%		%			\item $\displaystyle  K_2(z)=\sqrt{\frac{\pi}{2z}}\frac{1}{e^z}[1+\frac{15}{8z}+\frac{105}{2!(8z)^2}+O(z^{-3})]$.\\
%		%            \item $\displaystyle \frac{K_3(z)}{K_2(z)}-1=\frac{\frac{5}{2z}+O(z^{-2})}{1+O(z^{-1})}$.\\
%		%            \item $\displaystyle  \Big(\frac{K_3(z)}{K_2(z)}\Big)^2-1=\frac{\frac{5}{z}+\frac{115}{4z^2}+O(z^{-3})}{1+\frac{15}{4z}+O(z^{-2})}$.\\
%		\item $\displaystyle \alpha-1=\frac{\frac{5}{2\mathfrak{c}^2}+O(\mathfrak{c}^{-4})}{1+O(\mathfrak{c}^{-2})}$.\\
%		
%		\item $\displaystyle \alpha^2-1=\frac{\frac{5}{\mathfrak{c}^2}+\frac{115}{4\mathfrak{c}^4}+O(\mathfrak{c}^{-6})}{1+\frac{15}{4\mathfrak{c}^2}+O(\mathfrak{c}^{-4})}$.\\
%		
%		\item $\displaystyle  
%		\mathfrak{c}^2(A_2-A_3^2)=\frac{3}{2}+O(\mathfrak{c}^{-2})$.
%		%			\item $\displaystyle  \lim\limits_{\mathfrak{c}\rightarrow\infty}\frac{p^0-A_3}{\sqrt{A_2-A_3^2}}=\frac{1}{\sqrt{6}}(|p|^2-3)$.
%	\end{enumerate}
%\end{Lemma}
%Lemmas \ref{lem2.2}-\ref{alem2.2} can be proved by direct calculation.

%%%%%%%%%%%%%%%%%%%%%%%%%%%%%%%%%%%%%%%%%%%%%%%%%%%%%%%%%%%%%%%
\subsection{Newtonian Boltzmann collision operators}
%%%%%%%%%%%%%%%%%%%%%%%%%%%%%%%%%%%%%%%%%%%%%%%%%%%%%%%%%%%%%%%

Recall \eqref{1.19-0}, we can rewrite the Newtonian Boltzmann equation \eqref{1.14-0} as
\begin{align}\label{2.4-10}
f_t+p \cdot \nabla_x f+\mathbf{L} f=\mathbf{\Gamma}(f, f),
\end{align}
where the linearized term is given by
$$
\mathbf{L} f=\nu(p) f-\mathbf{K} f=-\frac{1}{\sqrt{\mu}}\{Q(\mu, \sqrt{\mu} f)+Q(\sqrt{\mu} f, \mu)\}
$$
Here $\nu$ and $\mathbf{K}:=\mathbf{K}_2-\mathbf{K}_1$ are defined as
\begin{equation}\label{2.5}
	\begin{split}
		\left(\mathbf{K}_1 f\right)(p):= & \int_{\mathbb{R}^3} \int_{\mathbb{S}^2} \mathcal{K}_\infty(p,q,\omega) \sqrt{\mu(p) \mu(q)} f(q) \mathrm{d} \omega \mathrm{d}q, \\
		\left(\mathbf{K}_2 f\right)(p):= & \int_{\mathbb{R}^3} \int_{\mathbb{S}^2} \mathcal{K}_\infty(p,q,\omega) \sqrt{\mu(q) \mu\left(\bar{q}^{\prime}\right)} f\left(\bar{p}^{\prime}\right) \mathrm{d} \omega \mathrm{d} q, \\
		& +\int_{\mathbb{R}^3} \int_{\mathbb{S}^2} \mathcal{K}_\infty(p,q,\omega) \sqrt{\mu(q) \mu\left(\bar{p}^{\prime}\right)} f\left(\bar{q}^{\prime}\right) \mathrm{d} \omega \mathrm{d} q,  \\
		\nu(p):=& \int_{\mathbb{R}^3} \int_{\mathbb{S}^2} \mathcal{K}_\infty(p,q,\omega) \mu(q) \mathrm{d} \omega \mathrm{d} q.
	\end{split}
\end{equation}
%where
%\begin{align}\label{2.20-1}
%	\mathcal{K}_\infty(p,q,\omega):=|\omega\cdot(p-q)|.
%\end{align}
The Newtonian nonlinear collision operator is given by
\begin{align*}%\label{2.6}
	\mathbf{\Gamma}(h_1, h_2) & :=\frac{1}{\sqrt{\mu}} Q(\sqrt{\mu} h_1, \sqrt{\mu} h_2)=\frac{1}{\sqrt{\mu}} Q^{+}(\sqrt{\mu} h_1, \sqrt{\mu} h_2)-\frac{1}{\sqrt{\mu}} Q_{-}(\sqrt{\mu} h_1, \sqrt{\mu} h_2)\nonumber\\
%	&=\int_{\mathbb R^3}\int_{\mathbb S^2}\mathcal{K}_\infty(p,q,\omega) \sqrt{\mu(q)}\Big[h_1(\bar{p}')h_2(\bar{q}')-h_1(p)h_2(q)\Big]d\omega dq,
    &=\mathbf{\Gamma}^{+}(h_1, h_2)-\mathbf{\Gamma}^{-}(h_1, h_2),
\end{align*}
where the post-collision velocities are defined in \eqref{1.16-0}. It holds that
\begin{align*}
	\|\nu^{-1}w_{\beta}\mathbf{\Gamma}(h_1,h_2)\|_{L^\infty_p}\lesssim \|w_{\beta}h_1\|_{L^\infty_p}\|w_{\beta}h_2\|_{L^\infty_p}
\end{align*}
and 
\begin{align*}
	\|w_{\beta}\mathbf{\Gamma}^{+}(h_1,h_2)\|_{L^\infty_p}\lesssim \|w_{\beta}h_1\|_{L^\infty_p}\|w_{\beta}h_2\|_{L^\infty_p}.
\end{align*}

It follows from \cite{Glassey} that
\begin{align*}
	(\mathbf{K}_if)(p)=\int_{\mathbb{R}^3}k_i(p,q)f(q)dq,\quad i=1,2,
\end{align*}
where 
\begin{align*}
	k_1(p,q)&:=\frac{1}{\sqrt{2\pi}}|p-q|e^{-\frac{|p|^2}{4}-\frac{|q|^2}{4}},\\
	k_2(p,q)&:=\sqrt{\frac{2}{\pi}}\frac{1}{|p-q|}e^{-\frac{|p-q|^2}{8}-\frac{(|p|^2-|q|^2)^2}{8|p-q|^2}}.
\end{align*}
Denote $k(p,q):=k_2(p,q)-k_1(p,q)$. By similar calculations as in \cite[Lemma 3.3.1]{Glassey}, it holds that for any $\alpha\ge 0$,
\begin{align}\label{2.12-20}
	\int_{\mathbb{R}^3}|k(p,q)|\frac{(1+|p|^2)^{\frac{\alpha}{2}}}{(1+|q|^2)^{\frac{\alpha}{2}}}dq\lesssim \frac{1}{1+|p|}.
\end{align}

%%%%%%%%%%%%%%%%%%%%%%%%%%%%%%%%%%%%%%%%%%%%%%%%%%%%%%%%%%%%%%%
\subsection{Hilbert-Schmidt formulation}
%%%%%%%%%%%%%%%%%%%%%%%%%%%%%%%%%%%%%%%%%%%%%%%%%%%%%%%%%%%%%%%

Using Lorentz transformations as described in \cite{Groot, Strain2}, in the \textit{center of momentum system}, $\mathcal{Q}_{\mathfrak{c}}(F,F)$ can be written as
\begin{align}\label{a1.6}
	\mathcal{Q}_{\mathfrak{c}}(F,F)=&\int_{\mathbb R^3}\int_{\mathbb S^2}v_{\o} \Big[F(p')F(q')-F(p)F(q)\Big]d\omega dq 
	:=\mathcal{Q}_{\mathfrak{c}}^{+}(F,F)-\mathcal{Q}_{\mathfrak{c}}^{-}(F,F),
\end{align}
where $v_{\o}=v_{\o}(p,q)$ is the M{\o}ller velocity
\begin{align*}%\label{1.12}
	v_{\o}(p,q):=\frac{\mathfrak{c}}{2}\sqrt{\left|\frac{p}{p^0}-\frac{q}{q^0} \right|^2-\left|\frac{p}{p^0}\times \frac{q}{q^0}\right|^2}=\frac{\mathfrak{c}}{4}\frac{g\sqrt{\mathfrak{s}}}{p^0q^0}.
\end{align*}
The pre-post collisional momentum in \eqref{a1.6} satisfies
\begin{equation}\label{2.14-20}
	\left\{
	\begin{aligned}
		&p'=\frac12(p+q)+\frac12 g\Big(\omega+(\gamma_0-1)(p+q)\frac{(p+q)\cdot \omega}{|p+q|^2}\Big),\\
		&q'=\frac12(p+q)-\frac12 g\Big(\omega+(\gamma_0-1)(p+q)\frac{(p+q)\cdot \omega}{|p+q|^2}\Big),
	\end{aligned}
	\right.
\end{equation}
where $\gamma_0:=(p^0+q^0)/\sqrt{\mathfrak{s}}$. The pre-post collisional energy is given by
\begin{equation}\label{2.14-21}
	\left\{
	\begin{aligned}
		&p^{\prime 0}=\frac12(p^0+q^0)+\frac{1}{2}\frac{g}{\sqrt{\mathfrak{s}}}(p+q)\cdot \omega,\\
		&q^{\prime 0}=\frac12(p^0+q^0)-\frac{1}{2}\frac{g}{\sqrt{\mathfrak{s}}}(p+q)\cdot \omega.
	\end{aligned}
	\right.
\end{equation}

Under the perturbation form \eqref{1.10-0}, we can rewrite the relativistic Boltzmann equation \eqref{1.3-0}--\eqref{1.4-0} as
\begin{align}\label{5.8}
	\partial_tf_{\fc}+\hat{p} \cdot \nabla_x f_\fc+\mathbf{L}_{\mathfrak{c}} f_{\mathfrak{c}}=\mathbf{\Gamma}_{\mathfrak{c}}(f_\fc, f_\fc)
\end{align}
and 
\begin{align}\label{5.9}
	f_\fc(t,x,p)|_{t=0}=f_{0,\fc},
\end{align}
where the linearized term $\mathbf{L}_{\mathfrak{c}}f_{\mathfrak{c}}$ is
\begin{align*}
	\mathbf{L}_{\mathfrak{c}} f_{\mathfrak{c}}:=-\frac{1}{\sqrt{J_{\mathfrak{c}}}}\big[\mathcal{Q}_{\mathfrak{c}}(J_{\mathfrak{c}}, \sqrt{J_{\mathfrak{c}}} f_{\mathfrak{c}})+\mathcal{Q}_{\mathfrak{c}}(\sqrt{J_{\mathfrak{c}}} f_{\mathfrak{c}}, J_{\mathfrak{c}})\big]=\nu_{\mathfrak{c}}(p) f_{\mathfrak{c}}-\mathbf{K}_{\mathfrak{c}} f_{\mathfrak{c}}.
\end{align*}
Here $\nu_{\mathfrak{c}}$ and  $\mathbf{K}_{\mathfrak{c}}f_\fc:=\mathbf{K}_{\mathfrak{c}2}f_\fc-\mathbf{K}_{\mathfrak{c}1}f_\fc$ are defined as 
\begin{align*}%\label{2.13-40}
	\begin{split}
	\left(\mathbf{K}_{\mathfrak{c}1} f_{\mathfrak{c}}\right)(p):= & \int_{\mathbb{R}^3} \int_{\mathbb{S}^2} v_{\o}  \sqrt{J_{\mathfrak{c}}(p) J_{\mathfrak{c}}(q)} f_{\mathfrak{c}}(q) d \omega d q,  \\
	\left(\mathbf{K}_{\mathfrak{c}2} f_{\mathfrak{c}}\right)(p)
%	:= & \frac{1}{\sqrt{J_{\mathfrak{c}}}}\left\{\mathcal{Q}_{\mathfrak{c}}^{+}(J_{\mathfrak{c}}, \sqrt{J_{\mathfrak{c}}} f_{\mathfrak{c}})+\mathcal{Q}_{\mathfrak{c}}^{+}(\sqrt{J_{\mathfrak{c}}} f_{\mathfrak{c}}, J_{\mathfrak{c}})\right\} \nonumber\\
	:= & \int_{\mathbb{R}^3} \int_{\mathbb{S}^2} v_{\o}  \sqrt{J_{\mathfrak{c}}(q) J_{\mathfrak{c}}\left(q^{\prime}\right)} f_{\mathfrak{c}}\left(p^{\prime}\right) d \omega d q +\int_{\mathbb{R}^3} \int_{\mathbb{S}^2} v_{\o}   \sqrt{J_{\mathfrak{c}}(q) J_{\mathfrak{c}}\left(p^{\prime}\right)} f_{\mathfrak{c}}\left(q^{\prime}\right) d \omega d q, \\
	\nu_{\mathfrak{c}}(p):=&\int_{\mathbb{R}^3} \int_{\mathbb{S}^2} v_{\o}  J_{\mathfrak{c}}(q) d \omega d q.
	\end{split}
\end{align*}
The relativistic nonlinear collision operator takes the form
\begin{align*} 
	\mathbf{\Gamma}_{\mathfrak{c}}(h_1, h_2) &=\frac{1}{\sqrt{J_{\mathfrak{c}}}} \mathcal{Q}_{\mathfrak{c}}(\sqrt{J_{\mathfrak{c}}} h_1, \sqrt{J_{\mathfrak{c}}} h_2)=\frac{1}{\sqrt{J_{\mathfrak{c}}}} \mathcal{Q}^{+}_{\mathfrak{c}}(\sqrt{J_{\mathfrak{c}}} h_1, \sqrt{J_{\mathfrak{c}}} h_2)-\frac{1}{\sqrt{J_{\mathfrak{c}}}} \mathcal{Q}^{-}_{\mathfrak{c}}(\sqrt{J_{\mathfrak{c}}} h_1, \sqrt{J_{\mathfrak{c}}} h_2)\nonumber\\
	&:=\mathbf{\Gamma}_{\mathfrak{c}}^{+}(h_1, h_2)-\mathbf{\Gamma}_{\mathfrak{c}}^{-}(h_1, h_2).
\end{align*}

	Denote
\begin{align*}
	\boldsymbol{\ell}=\mathfrak{c}\frac{p^0+q^0}{2}, \quad  
	\boldsymbol{j}=\mathfrak{c}\frac{|p\times q|}{g}.
\end{align*}
From \cite{Groot, Strain2}, we know that
\begin{align*}
	\left(\mathbf{K}_{\mathfrak{c} i} f_{\mathfrak{c}}\right)(p)=\int_{\mathbb{R}^{3}} k_{\mathfrak{c}i}(p, q) f_{\mathfrak{c}}(q) d q, \quad i=1,2
\end{align*}
with the symmetric kernels
\begin{align*} 
	k_{\mathfrak{c}1}(p, q)=\frac{\pi}{2}\mathfrak{c}\frac{g\sqrt{\mathfrak{s}}}{p^0q^0}\frac{1}{4\pi \mathfrak{c}K_{2}(\mathfrak{c}^2)}e^{-\boldsymbol{\ell}}\int_{0}^{\pi}\sin \theta d\theta=\frac{1}{4K_2(\mathfrak{c}^2)}\frac{g\sqrt{\mathfrak{s}}}{p^0q^0}e^{-\boldsymbol{\boldsymbol{\ell}}}
\end{align*}
and 
\begin{align*} 
	k_{\mathfrak{c}2}(p,q)&=\frac{\pi}{4}\mathfrak{c}\frac{\mathfrak{s}^{3 / 2}}{g p^{0} q^{0}}\frac{1}{4\pi \mathfrak{c}K_{2}(\mathfrak{c}^2)}\int_{0}^{\infty} \frac{y\left(1+\sqrt{y^{2}+1}\right)}{\sqrt{y^{2}+1}}  e^{-\boldsymbol{\ell}\sqrt{y^{2}+1}}I_{0}\left(\boldsymbol{j}y\right)d y\nonumber\\
	&=\frac{1}{16K_2(\mathfrak{c}^2)}\frac{\mathfrak{s}^{3 / 2}}{g p^{0} q^{0}}\left[J_1(\boldsymbol{\ell},\boldsymbol{j})+J_2(\boldsymbol{\ell},\boldsymbol{j})\right].
\end{align*}
Here $I_{0}(u):=\frac{1}{2 \pi} \int_{0}^{2 \pi} e^{u \cos \varphi} d \varphi$ and
\begin{align*} 
		J_{1}(\boldsymbol{\ell}, \boldsymbol{j}):=\frac{\boldsymbol{\ell}}{\boldsymbol{\ell}^{2}-\boldsymbol{j}^{2}}\left[1+\frac{1}{\sqrt{\boldsymbol{\ell}^{2}-\boldsymbol{j}^{2}}}\right] e^{-\sqrt{\boldsymbol{\ell}^{2}-\boldsymbol{j}^{2}}}, \quad 
		J_{2}(\boldsymbol{\ell}, \boldsymbol{j}):=\frac{1}{\sqrt{\boldsymbol{\ell}^{2}-\boldsymbol{j}^{2}}} e^{-\sqrt{\boldsymbol{\ell}^{2}-\boldsymbol{j}^{2}}}.%\label{2.13}
\end{align*}

It is well-known that $\mathbf{L}_{\mathfrak{c}}$ is a self-adjoint non-negative definite operator in $L_{p}^{2}$ space with the kernel
\begin{align*}
	\mathcal{N}_{\mathfrak{c}}=\operatorname{span}\left\{\sqrt{J_{\mathfrak{c}}},\ p_i\sqrt{J_{\mathfrak{c}}} \ (i=1,2,3),\ p^0\sqrt{J_{\mathfrak{c}}} \right\}.
\end{align*}
Using Lemma \ref{lem2.2}, one can obtain the {\it orthonormal} basis for the null space $\mathcal{N}_{\mathfrak{c}}$ as
\begin{align*}
	\mathcal{N}_{\mathfrak{\mathfrak{c}}}=\operatorname{span}\left\{e_{0}, e_{1}, e_{2}, e_{3}, e_{4}\right\}
\end{align*}
with
\begin{align}\label{2.20}
	e_{0} =\sqrt{J_{\mathfrak{c}}},\ e_{i} =\frac{p_i}{\sqrt{A_1}}\sqrt{J_{\mathfrak{c}}}\ (i=1,2,3), \
	e_{4} =\frac{p^0-A_3}{\sqrt{A_2-A^2_3}}\sqrt{J_{\mathfrak{c}}},
\end{align}
where $A_1$, $A_2$ and $A_3$ are defined in Lemma \ref{lem2.2}. 

Let $\mathbf{P}_{\mathfrak{c}}$ be the orthogonal projection from $L_{p}^{2}$ onto $\mathcal{N}_{\mathfrak{c}}$. For given $f$,  we denote the relativistic macroscopic part $\mathbf{P}_{\mathfrak{c}} f$ as 
\begin{align*}
	\mathbf{P}_{\mathfrak{c}} f=\Big\{a+b\cdot p+c\frac{p^0-A_3}{\sqrt{A_2-A^2_3}}\Big\}\sqrt{J_{\mathfrak{c}}}
\end{align*}
and further denote $(\mathbf{I}-\mathbf{P}_{\mathfrak{c}})f$ to be the relativistic microscopic part of $f$. 

Using \cite[Propositon 10]{Ruggeri}, one has
\begin{align*}
	\left(\frac{K_1(z)}{K_2(z)}\right)^2+\frac{3}{z}\frac{K_1(z)}{K_2(z)}-\frac{3}{z^2}-1<0, \quad z>0,
\end{align*}
which is equivalent to
\begin{align*} 
	\left(\frac{K_3(z)}{K_2(z)}\right)^2-\frac{5}{z}\frac{K_3(z)}{K_2(z)}+\frac{1}{z^2}-1<0, \quad z>0.
\end{align*}
Then we have $A_2-A^2_3>0$ and thus \eqref{2.20} is well-defined.

Denote $k_{\mathfrak{c}}(p,q):=k_{\mathfrak{c}2}(p,q)-k_{\mathfrak{c}1}(p,q)$ and
\begin{align*}
	k_{\mathfrak{c}w}(p,q):=k_{\mathfrak{c}}(p,q)\frac{w_{\beta}(p)}{w_{\beta}(q)}.
\end{align*}
%As in \cite{Wang}, we denote
%\begin{align}
%	\tilde{k}_1(p,q)&=|p-q|e^{-\frac{|p|^2}{4}-\frac{|q|^2}{4}},\label{2.4-1}\\
%	\tilde{k}_2(p,q)&=\frac{2}{|p-q|}e^{-\frac{|p-q|^2}{8}-\frac{(|p|^2-|q|^2)^2}{8|p-q|^2}},\label{2.5-2}
%\end{align}
%which are indeed the corresponding kernels of Newtonian Boltzmann equation. Similarly, $\tilde{\nu},\tilde{K},\tilde{\mathbf{\Gamma}}$ represent the classical case.
%By similar arguments as in \cite{Wang-Xiao}, we can obtain the following key uniform estimates whose proofs are omitted for brevity.
Then the following uniform-in-$\mathfrak{c}$ estimates holds:
\begin{Lemma}\emph{(\cite{Wang-Xiao})}\label{lem2.3}
	Recall $w_{\beta}(p)$ in \eqref{1.9-01}. Then it holds that 
	\begin{enumerate}[(1)]
		\item $\displaystyle 
\nu_\fc(p)\simeq\begin{cases}
	1+|p|,\ |p|\leq \mathfrak{c},\\
	\mathfrak{c},\ |p|\geq\mathfrak{c}.
\end{cases}$ \\
		\item $\displaystyle \int_{\R^3}k^2_{\mathfrak{c}w}(p,q)dq\lesssim
		\left\{
		\begin{aligned}
			&\frac{1}{1+|p|}, \ |p|\le \mathfrak{c},\\
			&\frac{1}{\mathfrak{c}}, \  |p|\ge \mathfrak{c}.
		\end{aligned}
		\right.$\\
%		\item $\displaystyle \int_{\R^3}k_2(p,q)\frac{w(p)}{w(q)}dq\lesssim
%		\begin{cases}
%			\frac{1}{1+|p|},\ |p|\leq \mathfrak{c},\\
%			\frac{1}{\mathfrak{c}},\ |p|\geq\mathfrak{c}.
%		\end{cases}$\\
		\item $\displaystyle 
		\int_{\mathbb{R}^3}|k_{\mathfrak{c}w}(p,q)|e^{\frac{1}{32}|p-q|}dq\lesssim
		\left\{
		\begin{aligned}
			&\frac{1}{1+|p|}, \ |p|\le \mathfrak{c},\\
			&\frac{1}{\mathfrak{c}}, \  |p|\ge \mathfrak{c}.
		\end{aligned}
		\right.$\\
		\item $\displaystyle 
		\langle\mathbf{L}_\fc g,g\rangle\geq \zeta_0\|(\mathbf{I-P_\fc})g\|_{\nu_\fc}^2,\quad \zeta_0>0\, \text{and $\zeta_0$ is independent of $\mathfrak{c}$} $.
	\end{enumerate}
\end{Lemma}
%The mild form of the relativistic Boltzmann equation \eqref{5.8} is given by
%\begin{align}\label{5.11}
%	f_\fc(t, x, p)= & e^{-\nu(p) t} f_{0,\fc}(x-\hat{p} t, p)+\int_0^t e^{-\nu(p)(t-s)}(K f_\fc)(s, x-\hat{p}(t-s), p) d s \nonumber\\
%	& +\int_0^t e^{-\nu(p)(t-s)} \mathbf{\Gamma}(f_\fc, f_\fc)(s, x-\hat{p}(t-s), p) d s,
%\end{align}
%with initial condition
%\begin{align}\label{5.12}
%	f_{0,\fc}(x, p)=\frac{F_{0,\fc}(x, p)-J_{\mathfrak{c}}(p)}{\sqrt{J_{\mathfrak{c}}(p)}}.
%\end{align}

%%%%%%%%%%%%%%%%%%%%%%%%%%%%%%%%%%%%%%%%%%%%%%%%%%%%%%%%%%%%%%%
\subsection{Glassey-Strauss expression}
%%%%%%%%%%%%%%%%%%%%%%%%%%%%%%%%%%%%%%%%%%%%%%%%%%%%%%%%%%%%%%%
To establish the Newtonian limit, as pointed in \cite{Strain}, it is more convenient to use the Glassey-Strauss expression instead of the one under {\it center of momentum system}. Glassey and Strauss illustrated in \cite{Gla-Str} that a reduction of the collision integrals can be performed, without using Lorentz transformations, to obtain
\begin{align}\label{2.13-00}
	\mathcal{Q}_{\mathfrak{c}}(h_1,h_2)=\iint_{\mathbb{R}^3\times \mathbb{S}^2}\mathcal{K}_\fc(p,q,\omega)[h_1(p')h_2(q')-h_1(p)h_2(q)]d\omega dq,
\end{align}
where the kernel
\begin{align*}
	\mathcal{K}_\fc(p,q,\omega):=\f{\mathfrak{s}}{p^0q^0}B(p,q,\omega)
\end{align*}
with
\begin{align*}
	B(p,q,\omega):=\fc\f{(p^0+q^0)^2p^0q^0|\omega\cdot(\f{p}{p^0}-\f{q}{q^0})|}{[(p^0+q^0)^2-(\omega\cdot(p+q))^2]^2}.
\end{align*} 
The post-collision momentum in \eqref{2.13-00} is given by
\begin{align}\label{2.16-0}
	p'=p+a(p,q,\omega)\omega,\quad q'=q-a(p,q,\omega)\omega,
\end{align}
where
\begin{align*}
	a(p,q,\omega):=\f{2p^0q^0(p^0+q^0)\{\omega\cdot(\f{q}{q^0}-\f{p}{p^0})\}}{(p^0+q^0)^2-\{\omega\cdot(p+q)\}^2}.
\end{align*} 
The post-collision energies can be expressed as
\begin{align}\label{2.16-1}
	p'^{0}=p^0+N_0, q'^{0}=q^0-N_0
\end{align} 
with 
\begin{align}\label{2.30-20}
	N_0:=\f{2\omega\cdot(p+q)\{p^0(\omega\cdot q)-q^0(\omega\cdot p)\}}{(p^0+q^0)^2-\{\omega\cdot(p+q)\}^2}.
\end{align}

It follows from \cite{Glassey1} that
\begin{align*}
	 \frac{\partial(p',q')}{\partial(p,q)}=-\frac{p^{\prime 0}q^{\prime 0}}{p^0q^0},
\end{align*}
which, together with $B(p,q,\omega)=B(p',q',\omega)$, yields that 
\begin{align}\label{2.30-10}
	\mathcal{K}_\fc(p,q,\omega)dpdq=\f{\mathfrak{s}}{p^{\prime 0}q^{\prime 0}}B(p',q',\omega)dp'dq'=\mathcal{K}_\fc(p',q',\omega)dp'dq'.
\end{align}
It is direct to check that
\begin{align*}
	\lim_{\mathfrak{c}\to \infty}\mathcal{K}_{\mathfrak{c}}(p,q,\omega)=\mathcal{K}_\infty(p,q,\omega).
\end{align*}

It follows from \cite[Corollary 1.5]{Strain1} that, for any suitable integrable function $\mathcal{G}: \mathbb{R}^3\times \mathbb{R}^3\times \mathbb{R}^3\times \mathbb{R}^3\rightarrow \mathbb{R}$, it holds that
\begin{align*} 
	\int_{\mathbb{S}^2}  v_{\o}  \mathcal{G}\left(P, Q, P^{\prime}, Q^{\prime}\right)d \omega=\int_{\mathbb{S}^{2}} \mathcal{K}_{\mathfrak{c}}(p,q,\omega) \mathcal{G}\left(P, Q, P^{\prime}, Q^{\prime}\right)d \omega,
\end{align*}
where the post-collisional momentum on the left are defined as in \eqref{2.14-20}--\eqref{2.14-21} and the similar expressions on the right are the ones given  in \eqref{2.16-0}--\eqref{2.16-1}. 
Consequently, under the Glassey-Strauss expression with $\mathcal{Q}_{\mathfrak{c}}$ given by \eqref{2.13-00}, the collision operators take the following forms:
\begin{align*}
	\mathbf{L}_{\mathfrak{c}} f_{\mathfrak{c}}:=-\frac{1}{\sqrt{J_{\mathfrak{c}}}}\big[\mathcal{Q}_{\mathfrak{c}}(J_{\mathfrak{c}}, \sqrt{J_{\mathfrak{c}}} f_{\mathfrak{c}})+\mathcal{Q}_{\mathfrak{c}}(\sqrt{J_{\mathfrak{c}}} f_{\mathfrak{c}}, J_{\mathfrak{c}})\big]=\nu_{\mathfrak{c}}(p) f_{\mathfrak{c}}-\mathbf{K}_{\mathfrak{c}} f_{\mathfrak{c}}.
\end{align*}
Here $\nu_{\mathfrak{c}}$ and $\mathbf{K}_{\mathfrak{c}}f_\fc:=\mathbf{K}_{\mathfrak{c}2}f_\fc-\mathbf{K}_{\mathfrak{c}1}f_\fc$ are defined as 
\begin{equation}\label{2.21-03}
	\begin{split}
	\left(\mathbf{K}_{\mathfrak{c}1} f_{\mathfrak{c}}\right)(p)= & \int_{\mathbb{R}^3} \int_{\mathbb{S}^2} \mathcal{K}_{\mathfrak{c}}(p,q,\omega)  \sqrt{J_{\mathfrak{c}}(p) J_{\mathfrak{c}}(q)} f_{\mathfrak{c}}(q) d \omega d q, \\
	\left(\mathbf{K}_{\mathfrak{c}2} f_{\mathfrak{c}}\right)(p)
	= & \int_{\mathbb{R}^3} \int_{\mathbb{S}^2} \mathcal{K}_{\mathfrak{c}}(p,q,\omega)  \sqrt{J_{\mathfrak{c}}(q) J_{\mathfrak{c}}\left(q^{\prime}\right)} f_{\mathfrak{c}}\left(p^{\prime}\right) d \omega d q \\
	&\quad  +\int_{\mathbb{R}^3} \int_{\mathbb{S}^2} \mathcal{K}_{\mathfrak{c}}(p,q,\omega)   \sqrt{J_{\mathfrak{c}}(q) J_{\mathfrak{c}}\left(p^{\prime}\right)} f_{\mathfrak{c}}\left(q^{\prime}\right) d \omega d q, \\
	\nu_{\mathfrak{c}}(p)=&\int_{\mathbb{R}^3} \int_{\mathbb{S}^2} \mathcal{K}_{\mathfrak{c}}(p,q,\omega)  J_{\mathfrak{c}}(q) d \omega d q.
	\end{split}
\end{equation}
The relativistic nonlinear collision operator takes the form
\begin{align} \label{2.22-03}
	&\mathbf{\Gamma}_{\mathfrak{c}}(h_1, h_2)=\mathbf{\Gamma}_{\mathfrak{c}}^{+}(h_1, h_2)-\mathbf{\Gamma}_{\mathfrak{c}}^{-}(h_1, h_2)\nonumber\\ 
	&=\int_{\mathbb{R}^3} \int_{\mathbb{S}^2} \mathcal{K}_{\mathfrak{c}}(p,q,\omega)   \sqrt{J_{\mathfrak{c}}(q) } h_1\left(p^{\prime}\right)h_2\left(q^{\prime}\right) d \omega d q-\int_{\mathbb{R}^3} \int_{\mathbb{S}^2} \mathcal{K}_{\mathfrak{c}}(p,q,\omega)   \sqrt{J_{\mathfrak{c}}(q) } h_1\left(p\right)h_2\left(q\right) d \omega d q.
\end{align}
It is noted that the post-collisional momentum $(p',q')$ in \eqref{2.21-03}--\eqref{2.22-03} are the ones given in \eqref{2.16-0}--\eqref{2.16-1}.

As pointed out in \cite{Strain}, due to the complex expression of $N_0$ in \eqref{2.30-20} which involves too many  angles $\omega$, it is hard to use the Glassey-Strauss expression \eqref{2.13-00} for $\mathcal{Q}_{\mathfrak{c}}$ to prove the global existence theorem, i.e., Theorem \ref{thm1.1}. Alternatively, we adopt the expression \eqref{a1.6} for $\mathcal{Q}_{\mathfrak{c}}$ in the {\it center of momentum system}. As shown in \cite{Wang-Xiao}, using expression \eqref{a1.6}, one can obtain a series of uniform-in-$\mathfrak{c}$ estimates for the linear and nonlinear collision operators. While, to prove the Newtonian limit from the relativistic Boltzmann solutions to the Newtonian Boltzmann solutions, it is much more convenient to use the Glassey-Strauss expression \eqref{2.13-00} for $\mathcal{Q}_{\mathfrak{c}}$. The reason is that, in the {\it center of momentum system}, it is hard to compute the Jacobian determinant $\frac{\partial (p',q')}{\partial (p,q)}$, see \eqref{2.14-20} for the complex expressions of the post-collision momentum $(p',q')$. For the study of this determinant, we refer to \cite{Chapman} for more details. On the Glassey-Strauss expression \eqref{2.13-00}, it has been proved in \cite{Glassey1} that
\begin{align*}
	\frac{\partial (p',q')}{\partial (p,q)}=-\frac{p^{\prime 0}q^{\prime 0}}{p^0q^0},
\end{align*}
which is useful for substitution of variables when estimating integrals involving both $(p,q)$ and $(p',q')$. 
Hence we shall adopt the expression in  {\it center of momentum system} in Section 3, and use the Glassey-Strauss expression in Sections 4 $\&$ 5.

By similar arguments as in \cite{Calogero,Strain} with some minor adjustments, we have:
\begin{Lemma}\emph{(\cite{Calogero,Strain})} \label{lem2.4}
	Let $p'$, $q'$ be defined in \eqref{2.16-0} and recall $\bar{p}'$, $\bar{q}'$ in \eqref{1.16-0}. Then the following estimates hold:
	\begin{enumerate}
		\item
		$| p^{\prime}-\bar{p}'|+|q^{\prime}-\bar{q}'| \lesssim \frac{(|p|+|q|)^3}{\mathfrak{c}^2}$.
		\item
		$|p-\hat{p}|\le\frac{|p|^3}{2\mathfrak{c}^2}$.
		\item
		$\left|\mathcal{K}_{\mathfrak{c}}(p, q, \omega)-\mathcal{K}_{\infty}(p, q, \omega)\right|  \lesssim \frac{(1+|p|+|q|)^6}{\mathfrak{c}^2}$.
		\item
		$\left|\mathcal{K}_{\mathfrak{c}}(p, q, \omega)\right|\lesssim (1+|p|)\left(1+|q|\right)^{3}$.
		\item
		$\left|\mathcal{K}_{\mathfrak{c}}(p, q, \omega)\right|\lesssim (1+|q|)\left(1+|p|\right)^{3}$.
	\end{enumerate}
\end{Lemma}

%%%%%%%%%%%%%%%%%%%%%%%%%%%%%%%%%%%%%%%%%%%%%%%%%%%%%%%%%%%%%%%
\subsection{Some auxiliary estimates}
%%%%%%%%%%%%%%%%%%%%%%%%%%%%%%%%%%%%%%%%%%%%%%%%%%%%%%%%%%%%%%%

The following lemma collects some estimates which are stated in \cite[p.1583]{Strain} and will play an important role in the proof of Theorem \ref{thm1.3}. %We omit the details of the proofs for brevity since they are quite straightforward.
\begin{Lemma}\emph{(\cite{Strain})}\label{lem2.5}
	Let $r(z)=(\log z)^{\alpha_1}$ for $\frac{1}{2}<\alpha_1<1$. Then there hold
	\begin{enumerate}
		\item For given $B_0>0$, $\epsilon>0$, there is a uniform constant $\tilde{B}_0=\tilde{B}_0(\epsilon)>0$ such that
		$$
		e^{B_0r(z)} \leq \tilde{B}_0 z^\epsilon, \quad z \geq 1 .
		$$
		\item For given $B_1>0$, $\alpha>0$, there is a uniform constant $\tilde{B}_1=\tilde{B}_1(\alpha)>0$ such that
		$$
		e^{-B_1 r^2(z)} \leq \frac{\tilde{B}_1}{z^{1+\alpha}},\quad z\geq 1.
		$$
	\end{enumerate}
\end{Lemma}

For later use, we need the following lemma which 
 compares $J_{\mathfrak{c}}(p)$ and $\mu(p)$.
\begin{Lemma}\label{lem2.6}
	Recall that $J_{\mathfrak{c}}(p)=\frac{1}{4\pi\fc K_2(\fc^2)}e^{-\mathfrak{c}p^0}$ and $\mu(p)=\frac{1}{(2\pi)^{\f32}}e^{-\frac{|p|^2}{2}}$. Then for any $\epsilon\in (0,2)$, it holds that
	\begin{align}\label{1.8-00}
		|J_{\mathfrak{c}}(p)-\mu(p)|\le C(\epsilon) \frac{1}{\mathfrak{c}^{2-\epsilon}}e^{-\frac{|p|}{2}}
	\end{align}
    and
    \begin{align}\label{1.9-00}
    	|\sqrt{J_{\mathfrak{c}}(p)}-\sqrt{\mu(p)}|\le  C(\epsilon)\frac{1}{\mathfrak{c}^{2-\epsilon}}e^{-\frac{|p|}{4}}.
    \end{align}
\end{Lemma}
\begin{proof}
    For any $p\in \mathbb{R}^3$, it is clear that
    \begin{align*}
    	\mathfrak{c}^2-\mathfrak{c}p^0=\mathfrak{c}^2(1-\sqrt{1+\frac{|p|^2}{\mathfrak{c}^2}})=-\frac{|p|^2}{1+\sqrt{1+\frac{|p|^2}{\mathfrak{c}^2}}},
    \end{align*}
    which yields
    \begin{align}\label{2.36-00}
    	-\frac{|p|^2}{2}\le \mathfrak{c}^2-\mathfrak{c}p^0\le -\frac{|p|^2}{1+\sqrt{1+|p|^2}}=-\sqrt{1+|p|^2}+1\le -|p|+1.
    \end{align}

    Using Lemma \ref{lem2.1}, one has 
	\begin{align}\label{2.17-0}
		J_{\mathfrak{c}}(p)-\mu(p)&=\f{1}{(2\pi)^{\f32}}e^{\fc^2-\fc p^0}\big(1+O(\fc^{-2})\big)-\f{1}{(2\pi)^{\f32}}e^{-\f{|p|^2}{2}}\nonumber\\
		&= \f{1}{(2\pi)^{\f32}}e^{\fc^2-\fc p^0}O(\fc^{-2})+\f{1}{(2\pi)^{\f32}}\Big(e^{\fc^2-\fc p^0}-e^{-\f{|p|^2}{2}}\Big):= \mathcal{D}_1+\mathcal{D}_2.
	\end{align}
It follows from \eqref{2.36-00} that
\begin{align}\label{2.18-0}
	|\mathcal{D}_1|\lesssim \frac{1}{\mathfrak{c}^2}e^{-|p|}.
\end{align}

To estimate $\mathcal{D}_2$, for $|p|\ge \mathfrak{c}^{\frac{\epsilon}{4}}$, we have
\begin{align}\label{2.19-0}
	|\mathcal{D}_2|\lesssim e^{-|p|}+e^{-\frac{|p|^2}{2}}&\lesssim \exp{\Big(-\frac{\mathfrak{c}^{\frac{\epsilon}{4}}}{2}\Big)}e^{-\frac{|p|}{2}}+\exp{\Big(-\frac{\mathfrak{c}^{\frac{\epsilon}{2}}}{4}\Big)}e^{-\frac{|p|^2}{4}}\le C(\epsilon)\frac{1}{\mathfrak{c}^2}e^{-\frac{|p|}{2}}.
\end{align}
For $|p|\le \mathfrak{c}^{\frac{\epsilon}{4}}$, it holds that
\begin{align*}
	0\le \fc^2-\fc p^0+\f{|p|^2}{2}=-\f{\fc|p|^2}{\fc+p^0}+\f{|p|^2}{2}=\f{|p|^4}{2(\fc+p^0)^2}\le  \frac{1}{8\mathfrak{c}^{2-\epsilon}},
\end{align*}
which implies that
\begin{align}\label{2.20-0}
	|\mathcal{D}_2|\lesssim e^{-\frac{|p|^2}{2}}\Big[\exp{\Big(\fc^2-\fc p^0+\f{|p|^2}{2}\Big)}-1\Big]\lesssim \frac{1}{\mathfrak{c}^{2-\epsilon}}e^{-\frac{|p|}{2}},\quad |p|\le \mathfrak{c}^{\frac{\epsilon}{4}}.
\end{align}
Thus \eqref{1.8-00} follows from \eqref{2.17-0}--\eqref{2.20-0} and \eqref{1.9-00} can be proved similarly. Therefore the proof is completed.
\end{proof}

%%%%%%%%%%%%%%%%%%%%%%%%%%%%%%%%%%%%%%%%%%%%%%%%%%%%%%%%%%%%%%%
\subsection{Nonlinear estimate} 
%%%%%%%%%%%%%%%%%%%%%%%%%%%%%%%%%%%%%%%%%%%%%%%%%%%%%%%%%%%%%%%

One also needs the following nonlinear estimate to obtain the local-in-time existence for the relativistic Boltzmann equation. 
\begin{Lemma}\label{lem2.7}
	Assume $\beta>\f72$, then it holds that
	\begin{align}\label{2.17-10}
		\left\|\nu_{\mathfrak{c}}^{-1}w_{\beta} \mathbf{\Gamma}_{\mathfrak{c}}(f_1,f_2)\right\|_{L^\infty} \le C\left\| w_{\beta} f_1\right\|_{L^\infty} \left\| w_{\beta} f_2\right\|_{L^\infty} .
	\end{align}
	Especially, we have 
	\begin{align}\label{2.18-10}
		\left\|w_{\beta} \mathbf{\Gamma}_{\mathfrak{c}}^{+}(f_1,f_2)\right\|_{L^\infty} \le C \left\| w_{\beta} f_1\right\|_{L^\infty} \left\| w_{\beta} f_2\right\|_{L^\infty} .
	\end{align}
    We point out that the above constants $C$ are independent of $\mathfrak{c}$.
\end{Lemma}
\begin{proof}
		It follows from \cite[Lemma 6.1]{Wang-Xiao} that \eqref{2.17-10} holds and $|p|\lesssim |p'|+|q'|$. Hence we mainly focus on the proof of \eqref{2.18-10}. Using the fact that  $$d\omega=\frac{2\sqrt{\mathfrak{s}}}{g}\delta^{(4)}(p^{\mu}+q^{\mu}-p'^{\mu}-q'^{\mu})\frac{dp'}{p'^{0}}\frac{dq'}{q'^{0}},$$ 
		one has
	\begin{align}\label{2.19-10}
		&\left|w_{\beta}(p)\mathbf{\Gamma}_{\mathfrak{c}}^{+}(f_1, f_2)(p)\right|\nonumber\\
		&=\left|w_{\beta}(p) \int_{\mathbb{R}^{3}}\int_{\mathbb{S}^{2}} v_{\o} \sqrt{J_{\mathfrak{c}}(q)} f_1(p') f_2(q') d \omega d q\right| \nonumber\\
		&=\left|\frac{w_{\beta}(p)}{p^0} \int_{\mathbb{R}^{3}}\int_{\mathbb{R}^{3}}\int_{\mathbb{R}^{3}}\frac{\mathfrak{c}}{2}\mathfrak{s}\delta^{(4)}(p^{\mu}+q^{\mu}-p'^{\mu}-q'^{\mu})\sqrt{J_{\mathfrak{c}}(q)} f_1(p') f_2(q') \frac{dq}{q^{0}}\frac{dp'}{p'^{0}}\frac{dq'}{q'^{0}}\right| \nonumber\\
		&\lesssim  \frac{\mathfrak{c}}{p^0} \int_{\mathbb{R}^{3}}\int_{\mathbb{R}^{3}}\int_{\mathbb{R}^{3}}\mathfrak{s}\delta^{(4)}(p^{\mu}+q^{\mu}-p'^{\mu}-q'^{\mu})\sqrt{J_{\mathfrak{c}}(q)}|w_{\beta}(q')f_2(q')| \cdot |f_1(p')| \frac{dq}{q^{0}}\frac{dp'}{p'^{0}}\frac{dq'}{q'^{0}} \nonumber\\
		&\qquad +\frac{\mathfrak{c}}{p^0} \int_{\mathbb{R}^{3}}\int_{\mathbb{R}^{3}}\int_{\mathbb{R}^{3}}\mathfrak{s}\delta^{(4)}(p^{\mu}+q^{\mu}-p'^{\mu}-q'^{\mu})\sqrt{J_{\mathfrak{c}}(q)}|w_{\beta}(p')f_1(p')| \cdot |f_2(q')| \frac{dq}{q^{0}}\frac{dp'}{p'^{0}}\frac{dq'}{q'^{0}}.
	\end{align}

	For the first term on the RHS of \eqref{2.19-10}, we exchange $p'$ and $q$ to get that
	\begin{align}\label{2.20-10}
		\frac{\mathfrak{c}}{p^0} \int_{\mathbb{R}^{3}}\int_{\mathbb{R}^{3}}\int_{\mathbb{R}^{3}}\bar{\mathfrak{s}}\delta^{(4)}(p^{\mu}+p'^{\mu}-q^{\mu}-q'^{\mu})\sqrt{J_{\mathfrak{c}}(p')}|w_{\beta}(q')f_2(q')| \cdot |f_1(q)| \frac{dq}{q^{0}}\frac{dp'}{p'^{0}}\frac{dq'}{q'^{0}},
	\end{align}
	where
	\begin{equation*}
     \bar{\mathfrak{s}}=\bar{g}^{2}+4 \mathfrak{c}^{2},\quad \bar{g}^{2}=g^{2}+\frac{1}{2}(p^{\mu}+q^{\mu}) \left(p_{\mu}+q_{\mu}-p'_{\mu}-q'_{\mu}\right).
	\end{equation*}

	For the second term on the RHS of \eqref{2.19-10}, we first exchange $p'$ and $q'$ and then exchange $p'$ and $q$ to get that
	\begin{align}\label{2.21-10}
		\frac{\mathfrak{c}}{p^0} \int_{\mathbb{R}^{3}}\int_{\mathbb{R}^{3}}\int_{\mathbb{R}^{3}}\bar{\mathfrak{s}}\delta^{(4)}(p^{\mu}+p'^{\mu}-q^{\mu}-q'^{\mu})\sqrt{J_{\mathfrak{c}}(p')}|w_{\beta}(q')f_1(q')| \cdot |f_2(q)| \frac{dq}{q^{0}}\frac{dp'}{p'^{0}}\frac{dq'}{q'^{0}}.
	\end{align}
	It follows from \eqref{2.19-10}--\eqref{2.21-10} that
	\begin{align}\label{2.22-10}
		\left|w_{\beta}(p)\mathbf{\Gamma}_{\mathfrak{c}}^{+}(f_1, f_2)(p)\right|\lesssim \left\| w_{\beta} f_1\right\|_{L^\infty} \left\| w_{\beta} f_2\right\|_{L^\infty} \int_{\mathbb{R}^{3}}\mathcal{H}(p,q)(1+|q|^2)^{-\frac{\beta}{2}}dq,
	\end{align}
	where
	\begin{align*} 
		\mathcal{H}(p,q):=\frac{\mathfrak{c}}{p^0q^0}\int_{\mathbb{R}^{3}}\int_{\mathbb{R}^{3}}\bar{\mathfrak{s}}\delta^{(4)}(p^{\mu}+p'^{\mu}-q^{\mu}-q'^{\mu})\sqrt{J_{\mathfrak{c}}(p')} \frac{dp'}{p'^{0}}\frac{dq'}{q'^{0}}.
	\end{align*}

    By similar arguments as in \cite{Wang}, one can show that
    	\begin{align}\label{2.24-10}
    	\mathcal{H}(p,q)&=\frac{\mathfrak{c}\pi \mathfrak{s}^{\frac{3}{2}}}{4gp^{0} q^{0}}\Big(\frac{1}{4\pi \mathfrak{c}K_2(\mathfrak{c}^2)}\Big)^{\frac{1}{2}}e^{-\frac{1}{4}\mathfrak{c}(q^0-p^0)}\int_{0}^{\infty} \frac{y\left[1+\sqrt{y^{2}+1}\right]}{\sqrt{1+y^{2}}} e^{-\frac{1}{2}\boldsymbol{\ell}\sqrt{1+y^2}} I_{0}\left(\frac{\boldsymbol{j}}{2} y\right)dy\nonumber\\
    	&\lesssim \frac{\mathfrak{c} \mathfrak{s}^{\frac{3}{2}}}{gp^{0} q^{0}}e^{\frac{\mathfrak{c}^2}{2}}e^{-\frac{\mathfrak{c}}{4}(q^0-p^0)}\int_{0}^{\infty} \frac{y\left[1+\sqrt{y^{2}+1}\right]}{\sqrt{1+y^{2}}} e^{-\frac{1}{2}\boldsymbol{\ell}\sqrt{1+y^2}} I_{0}\left(\frac{\boldsymbol{j}}{2} y\right)dy\nonumber\\
    	&\lesssim \frac{\mathfrak{c} \mathfrak{s}^{\frac{3}{2}}}{gp^{0} q^{0}}\frac{\boldsymbol{\ell}}{\boldsymbol{\ell}^2-\boldsymbol{j}^2}\exp{\Big(\frac{\mathfrak{c}^2}{2}-\frac{\mathfrak{c}}{4}(q^0-p^0)-\frac{1}{2}\sqrt{\boldsymbol{\ell}^2-\boldsymbol{j}^2}\Big)}\nonumber\\
    	&\lesssim \frac{\sqrt{\mathfrak{s}}(p^0+q^0)}{p^{0} q^{0}}\frac{1}{|p-q|}\exp{\Big(\frac{\mathfrak{c}^2}{2}-\frac{\mathfrak{c}}{4}(q^0-p^0)-\frac{1}{2}\sqrt{\boldsymbol{\ell}^2-\boldsymbol{j}^2}\Big)}.
    \end{align}
    A direct calculation shows that 
    $$
    \frac{\mathfrak{c}^2}{2}-\frac{\mathfrak{c}}{4}(q^0-p^0)-\frac{1}{2}\sqrt{\boldsymbol{\ell}^2-\boldsymbol{j}^2}\leq \frac{\mathfrak{c}^2}{2}+\frac{\mathfrak{c}}{4} \sqrt{|p-q|^2-g^2}-\frac{|p-q|}{2} \sqrt{\frac{\mathfrak{c}^2}{4}+\frac{\mathfrak{c}^4}{g^2}}:=\mathcal{A}.
    $$
    Set $\mathfrak{a}=|p-q|$, then $\mathfrak{a} \geq g$. We regard $\mathcal{A}$ as a function of $\mathfrak{a}$ with $g$ fixed, then
    $$
    \frac{d \mathcal{A}}{d \mathfrak{a}}=\frac{\mathfrak{c}}{4} \frac{\mathfrak{a}}{\sqrt{\mathfrak{a}^2-g^2}}-\frac{1}{2} \sqrt{\frac{\mathfrak{c}^2}{4}+\frac{\mathfrak{c}^4}{g^2}} .
    $$
    It is clear that $\mathcal{A}$ has a maximum point $\mathfrak{a}_0$ with
    $$
    \mathfrak{a}_0^2=\frac{g^4}{4 \mathfrak{c}^2}+g^2.
    $$
    Then one has
    $$
    \frac{\mathfrak{c}^2}{2}-\frac{\mathfrak{c}}{4}(q^0-p^0)-\frac{1}{2}\sqrt{\boldsymbol{\ell}^2-\boldsymbol{j}^2}\le \mathcal{A} \leq \frac{\mathfrak{c}^2}{2}+\frac{1}{8} g^2-\frac{1}{2}\left(\mathfrak{c}^2+\frac{g^2}{4}\right) =0,
    $$
    which, together with \eqref{2.24-10}, yields that
    \begin{align*}
    	\mathcal{H}(p,q)&\lesssim \frac{\sqrt{\mathfrak{s}}(p^0+q^0)}{p^{0} q^{0}}\frac{1}{|p-q|}\lesssim \frac{p^0+q^0}{\sqrt{p^{0} q^{0}}}\frac{1}{|p-q|}\nonumber\\
    	&\lesssim \frac{\sqrt{1+|p|}}{|p-q|}+\frac{\sqrt{1+|q|}}{|p-q|}.
    \end{align*}
    Hence, for $\beta>\frac{7}{2}$, one has 
    \begin{align}\label{2.25-10}
    	\int_{\mathbb{R}^{3}}\mathcal{H}(p,q)(1+|q|^2)^{-\frac{\beta}{2}}dq&\lesssim \sqrt{1+|p|}\int_{\mathbb{R}^{3}}\frac{1}{|p-q|}(1+|q|)^{-\beta}dq+\int_{\mathbb{R}^{3}}\frac{1}{|p-q|}(1+|q|)^{-\beta+\frac{1}{2}}dq\nonumber\\
    	&\lesssim 1.
    \end{align}
    Hence \eqref{2.18-10} follows from \eqref{2.22-10} and \eqref{2.25-10}. Therefore the proof is completed.
\end{proof}

\section{Global existence for the relativistic Boltzmann equation}
%%%%%%%%%%%%%%%%%%%%%%%%%%%%%%%%%%%%%%%%%%%%%%%%%%%%%%%%%%%%%%%

In this section, we shall establish the existence of global solutions for the relativistic Boltzmann equation \eqref{1.3-0}--\eqref{1.4-0}.

%To prove the global existence for the problem \eqref{1.3-0}--\eqref{1.4-0}, we shall start with the linear problem.

%%%%%%%%%%%%%%%%%%%%%%%%%%%%%%%%%%%%%%%%%%%%%%%%%%%%%%%%%%%%%%%
\subsection{Linear problem}
%%%%%%%%%%%%%%%%%%%%%%%%%%%%%%%%%%%%%%%%%%%%%%%%%%%%%%%%%%%%%%%

To solve the nonlinear problem, we first consider a linear problem. Recall \eqref{1.9-01} and denote
\begin{align*} 
\mathfrak{h}_{\mathfrak{c}}(t, x, p):=w_{\beta}(p) f_{\mathfrak{c}}(t, x, p).
\end{align*}
Multiplying \eqref{5.8} by $w_{\beta}$, one gets
\begin{align*} 
\partial_t\mathfrak{h}_{\mathfrak{c}}+\hat{p} \cdot \nabla_x \mathfrak{h}_{\mathfrak{c}}+\nu_{\mathfrak{c}}(p) \mathfrak{h}_{\mathfrak{c}}-\mathbf{K}_{\mathfrak{c}w} \mathfrak{h}_{\mathfrak{c}}=w_{\beta}\mathbf{\Gamma}_{\mathfrak{c}}\Big(\frac{\mathfrak{h}_{\mathfrak{c}}}{w_{\beta}}, \frac{\mathfrak{h}_{\mathfrak{c}}}{w_{\beta}}\Big),
\end{align*}
where
$$
\left(\mathbf{K}_{\mathfrak{c}w} \mathfrak{h}_{\mathfrak{c}}\right)(p):=w_{\beta}\Big(\mathbf{K}_{\mathfrak{c}} \frac{\mathfrak{h}_{\mathfrak{c}}}{w_{\beta}}\Big)(p).
$$
%and
%$$
%\begin{aligned}
%	\mathbf{\Gamma}_{w}(h, h) & :=w \mathbf{\Gamma}\left(\frac{h}{w}, \frac{h}{w}\right)=w \mathbf{\Gamma}^{+}\left(\frac{h}{w}, \frac{h}{w}\right)-w \mathbf{\Gamma}_{-}\left(\frac{h}{w}, \frac{h}{w}\right) \\
%	& \equiv w \mathbf{\Gamma}(f, f)=w \mathbf{\Gamma}^{+}(f, f)-w \mathbf{\Gamma}_{-}(f, f) .
%\end{aligned}
%$$

We first study the following linear inhomogeneous problem:
\begin{equation}\label{4.6}
	\left\{\begin{aligned}
		&\partial_{t} f_{\mathfrak{c}}+\hat{p} \cdot \nabla_{x} f_{\mathfrak{c}}+\mathbf{L}_{\mathfrak{c}} f_{\mathfrak{c}}=\mathcal{S}, \\
		&f_{\mathfrak{c}}(t, x, p)|_{t=0}=f_{0,\mathfrak{c}}(x, p),
	\end{aligned}\right.
\end{equation}
where $\mathcal{S}$ is a given function. Then the equation of $\mathfrak{h}_{\mathfrak{c}}=w_{\beta}f_{\mathfrak{c}}$ becomes
\begin{equation}\label{4.7}
	\left\{\begin{aligned}
		&\partial_{t} \mathfrak{h}_{\mathfrak{c}}+\hat{p} \cdot \nabla_{x} \mathfrak{h}_{\mathfrak{c}}+\nu_{\mathfrak{c}}(p) \mathfrak{h}_{\mathfrak{c}}=\mathbf{K}_{\mathfrak{c}w} \mathfrak{h}_{\mathfrak{c}}+w_{\beta} \mathcal{S}, \\
		&\mathfrak{h}_{\mathfrak{c}}(t,x,p)\mid_{t=0}=\mathfrak{h}_{0,\mathfrak{c}}(x,p).
	\end{aligned}\right.
\end{equation}
The mild form of \eqref{4.7} can be written as
\begin{align}\label{6.3}
	\mathfrak{h}_{\mathfrak{c}}(t, x, p)&= e^{-\nu_{\mathfrak{c}}(p) t} \mathfrak{h}_{0,\mathfrak{c}}(x-\hat{p} t, p)+\int_0^t e^{-\nu_{\mathfrak{c}}(p)(t-s)}\left(\mathbf{K}_{\mathfrak{c}w} \mathfrak{h}_{\mathfrak{c}}\right)(s, x-\hat{p}(t-s), p) d s \nonumber\\
	&\qquad +\int_0^t e^{-\nu_{\mathfrak{c}}(p)(t-s)}\left(w_{\beta}\mathcal{S}\right)(s, x-\hat{p}(t-s), p) d s\nonumber\\
	&:=\mathcal{I}_1+\mathcal{I}_2+\mathcal{I}_3.
\end{align}

\begin{Lemma}\label{lem4.3}
	Let $\mathfrak{h}_{\mathfrak{c}}(t, x, p)$ be the $L^{\infty}$ mild solution of the linear problem \eqref{4.7}. It holds that
	\begin{align}\label{4.10}
		|\mathfrak{h}_{\mathfrak{c}}(t, x, p)| &\le 
		 Ce^{-\frac{\nu_0}{2}t}\|\mathfrak{h}_{0,\mathfrak{c}}\|_{L^{\infty}}+Ce^{-\lambda_1t}\sup_{0 \le \tau \le t}\|e^{\lambda_1\tau}\nu_{\mathfrak{c}}^{-1}w_{\beta}\mathcal{S}(\tau)\|_{L^{\infty}}\nonumber\\
		&\quad +C_{N}e^{-\lambda_1t}\sup_{0 \le \tau \le t}\|e^{\lambda_1\tau}f_{\mathfrak{c}}(\tau)\|_{L^{2}}+C\max\Big\{\frac{1}{\mathfrak{c}},\ \frac{1}{N}\Big\}e^{-\lambda_1t}\sup_{0 \le \tau \le t}
		\|e^{\lambda_1\tau}\mathfrak{h}_{\mathfrak{c}}(\tau)\|_{L^{\infty}},
	\end{align}
	where $\nu_0:=\inf_{p\in \mathbb{R}^3}\nu_{\mathfrak{c}}(p)>0$ and $0<\lambda_1<\frac{\nu_0}{2}$, and $N\geq 1$ is a large constant chosen later. Here all above constants are independent of $\mathfrak{\mathfrak{c}}$.
\end{Lemma}
\begin{proof}
	Recall the mild form \eqref{6.3} of $\mathfrak{h}_{\mathfrak{c}}(t, x, p)$.  For $\mathcal{I}_1$, due to $\nu_{\mathfrak{c}}(p)\geq \nu_0$, one has
	\begin{align}\label{3.7-0}
		|\mathcal{I}_1|\leq Ce^{-\nu_0t}\|\mathfrak{h}_{0,\mathfrak{c}}\|_{L^\infty}.
	\end{align}
For $\mathcal{I}_3$, we have
\begin{align}\label{3.8-0}
	|\mathcal{I}_3|&\leq C\sup\limits_{0\leq\tau\leq t}\|e^{\lambda_1\tau}\nu_{\mathfrak{c}}^{-1}w_{\beta}\mathcal{S}(\tau)\|_{L^\infty}\cdot\int_0^te^{-\nu_{\mathfrak{c}}(p)(t-s)}e^{-\lambda_1s}\nu_{\mathfrak{c}}(p)ds\nonumber\\
	&\leq Ce^{-\lambda_1t}\sup\limits_{0\leq\tau\leq t}\|e^{\lambda_1\tau}\nu_{\mathfrak{c}}^{-1}w_{\beta}\mathcal{S}(\tau)\|_{L^\infty}.
\end{align}

  To estimate $\mathcal{I}_2$, we set $y_{\mathfrak{c}}:=x-\hat{p}(t-s)$ and substitute \eqref{6.3} into $\mathcal{I}_2$ to get
\begin{align}\label{3.9-0}
	\mathcal{I}_2&=\int_0^te^{-\nu_{\mathfrak{c}}(p)(t-s)}ds\int_{\R^3}k_{\mathfrak{c}w}(p,\tilde{p})\mathfrak{h}_{\mathfrak{c}}(s,y_{\mathfrak{c}},\tilde{p})d\tilde{p}\nonumber\\
	&=\int_0^{t}e^{-\nu_{\mathfrak{c}}(p)(t-s)}ds\int_{\mathbb{R}^3}k_{\mathfrak{c}w}(p,\tilde{p})e^{-\nu_{\mathfrak{c}}(\tilde{p})s}\mathfrak{h}_{0,\mathfrak{c}}(y_{\mathfrak{c}}-\hat{\tilde{p}}s,\tilde{p})d\tilde{p}\nonumber\\
	&\quad +\int_0^te^{-\nu_{\mathfrak{c}}(p)(t-s)}ds\int_{\R^3}k_{\mathfrak{c}w}(p,\tilde{p})d\tilde{p}\int_0^s e^{-\nu_{\mathfrak{c}}(\tilde{p})(s-s_1)}w_{\beta}\mathcal{S}(s_1,y_{\mathfrak{c}}-\hat{\tilde{p}}(s-s_1),\tilde{p})ds_1\nonumber\\
	&\quad +\int_0^te^{-\nu_{\mathfrak{c}}(p)(t-s)}ds\int_{\R^3}k_{\mathfrak{c}w}(p,\tilde{p})d\tilde{p} \int_0^se^{-\nu_{\mathfrak{c}}(\tilde{p})(s-s_1)}(\mathbf{K}_{\mathfrak{c}w}\mathfrak{h}_{\mathfrak{c}})(s_1,y_{\mathfrak{c}}-\hat{\tilde{p}}(s-s_1),\tilde{p})ds_1\nonumber\\
	&:=\mathcal{I}_{2,1}+\mathcal{I}_{2,2}+\mathcal{I}_{2,3}.
\end{align}
It follows from Lemma \ref{lem2.3} that 
\begin{align}\label{3.10-0}
	\int_{\R^3}|k_{\mathfrak{c}w}(p,\tilde{p})|d\tilde{p}\leq C\max\Big\{\f{1}{1+|p|},\f{1}{\fc}\Big\},
\end{align}
then one has
\begin{align}\label{3.11-0}
	|\mathcal{I}_{2,1}|&\leq C\|\mathfrak{h}_{0,\mathfrak{c}}\|_{L^\infty}\cdot\int_0^te^{-\nu_0(t-s)}ds\int_{\mathbb{R}^3}|k_{\mathfrak{c}w}(p,\tilde{p})|e^{-\nu_0s}d\tilde{p}\nonumber\\
	&\leq Ce^{-\f{\nu_0}{2}t}\|\mathfrak{h}_{0,\mathfrak{c}}\|_{L^{\infty}}.
\end{align}

Using \eqref{3.10-0} again, one has
\begin{align}\label{3.12-0}
	|\mathcal{I}_{2,2}|&\leq C\sup\limits_{0\leq\tau\leq t}\|e^{\la_1\tau}\nu_{\mathfrak{c}}^{-1}w_{\beta}\mathcal{S}(\tau)\|_{L^\infty}\cdot\int_0^{t}e^{-\nu_{\mathfrak{c}}(p)(t-s)}e^{-\lambda_1 s}ds\int_{\mathbb{R}^3}|k_{\mathfrak{c}w}(p,\tilde{p})|d\tilde{p}\nonumber\\
	&\qquad \times \int_0^se^{-\nu_{\mathfrak{c}}(\tilde{p})(s-s_1)}e^{\lambda_1(s-s_1)}\nu_{\mathfrak{c}}(\tilde{p})ds_1\nonumber\\
	&\leq Ce^{-\la_1t}\sup\limits_{0\leq\tau\leq t}\|e^{\la_1\tau}\nu_{\mathfrak{c}}^{-1}w_{\beta}\mathcal{S}(\tau)\|_{L^\infty}.
\end{align}
%\begin{align}\label{eq2.11}
%	|I_{2,1}|&\leq C\sup\limits_{0\leq\tau\leq t}\|e^{\la_1\tau}\mathfrak{h}_{\mathfrak{c}}\|_{L^\infty}\cdot\int_0^te^{-\nu_{\mathfrak{c}}(p)(t-s)}e^{-\la_1s}ds\cdot\max\Big\{\f{1}{\fc},\f{1}{N}\Big\}\nonumber\\
%	&\leq C\max\Big\{\f{1}{\fc},\f{1}{N}\Big\}e^{-\la_1t}\sup\limits_{0\leq\tau\leq t}\|e^{\la_1\tau}\mathfrak{h}_{\mathfrak{c}}\|_{L^\infty}.
%\end{align}
%For $I_{2,2}$, it is clear that
%\begin{align}\label{eq2.12}
%	|I_{2,2}|&\leq C\sup\limits_{0\leq\tau\leq t}\|e^{\la_1\tau}\mathfrak{h}_{\mathfrak{c}}\|_{L^\infty}\cdot\int_{t-\f{1}{N}}^te^{-\nu_{\mathfrak{c}}(p)(t-s)}e^{-\la_1s}ds\nonumber\\
%	&\leq \f{C}{N}e^{-\la_1t}\sup\limits_{0\leq\tau\leq t}\|e^{\la_1\tau}\mathfrak{h}_{\mathfrak{c}}\|_{L^\infty}.
%\end{align}
%
%\begin{align}\label{eq2.13}
%	|I_{2,3}|&\leq C\|\mathfrak{h}_{0,\mathfrak{c}}\|_{L^\infty}\cdot\int_0^te^{-\nu_0(t-s)}\int_{|p|\leq N}|k_{\mathfrak{c}w}(p,p')|e^{-\nu_0s}dp'ds\nonumber\\
%	&\leq Ce^{-\f{\nu_0}{2}t}\|\mathfrak{h}_{0,\mathfrak{c}}\|_{\infty}
%\end{align}
%and

Now we focus on the estimate of $\mathcal{I}_{2,3}$. Denote $\tilde{y}_{\mathfrak{c}}:=y_{\mathfrak{c}}-\hat{\tilde{p}}(s-s_1)$, then one gets
\begin{align}\label{3.13-0}
	|\mathcal{I}_{2,3}|&\le \int_0^{t}e^{-\nu_{\mathfrak{c}}(p)(t-s)}ds\int_0^se^{-\nu_{\mathfrak{c}}(\tilde{p})(s-s_1)}ds_1\nonumber\\
	&\qquad \times\iint_{\mathbb{R}^3\times\mathbb{R}^3 }|k_{\mathfrak{c}w}(p,\tilde{p})k_{\mathfrak{c}w}(\tilde{p},\tilde{p}')|\cdot|\mathfrak{h}_{\mathfrak{c}}(s_1,\tilde{y}_{\mathfrak{c}},\tilde{p}')|d\tilde{p}'d\tilde{p}.
\end{align}
We divide the estimate of \eqref{3.13-0} into four cases.\\
{\it Case 1:} $|p|\ge N$. Using \eqref{3.10-0} and \eqref{3.13-0}, we have
\begin{align}\label{3.14-0} 
	|\mathcal{I}_{2,3}|&\leq C\max\Big\{\f{1}{\fc},\f{1}{N}\Big\}\sup\limits_{0\leq\tau\leq t}\|e^{\la_1\tau}\mathfrak{h}_{\mathfrak{c}}(\tau)\|_{L^\infty}\cdot\int_0^te^{-\nu_{\mathfrak{c}}(p)(t-s)}ds\int_0^se^{-\nu_{\mathfrak{c}}(\tilde{p})(s-s_1)}e^{-\la_1s_1}ds_1\nonumber\\
	&\leq C\max\Big\{\f{1}{\fc},\f{1}{N}\Big\}e^{-\la_1t}\sup\limits_{0\leq\tau\leq t}\|e^{\la_1\tau}\mathfrak{h}_{\mathfrak{c}}(\tau)\|_{L^\infty}.
\end{align}
{\it Case 2:} $|p|\le N$, $|\tilde{p}|\ge 2N$, $\tilde{p}'\in \mathbb{R}^3$, or $|p|\le N$, $|\tilde{p}|\le 2N$, $\tilde{p}'\ge 3N$. Using Lemma \ref{lem2.3}, one of the following holds
\begin{align*}
	&\int_{|\tilde{p}|\ge 2N}|k_{\mathfrak{c}w}(p,\tilde{p})|d\tilde{p}\leq \int_{|\tilde{p}|\ge 2N}|k_{\mathfrak{c}w}(p,\tilde{p})|e^{\f{1}{32}|p-\tilde{p}|}d\tilde{p}\cdot e^{-\f{1}{32}N}\leq C\cdot e^{-\f{1}{32}N} ,
\end{align*}
or
\begin{align*}
	&\int_{|\tilde{p}'|\ge 3N}|k_{\mathfrak{c}w}(\tilde{p},\tilde{p}')|d\tilde{p}'\leq \int_{|\tilde{p}'|\ge 3N}|k_{\mathfrak{c}w}(\tilde{p},\tilde{p}')|e^{\f{1}{32}|\tilde{p}'-\tilde{p}|}d\tilde{p}'\cdot e^{-\f{1}{32}N}\leq C\cdot e^{-\f{1}{32}N}.
\end{align*}
Then it is clear that
\begin{align}\label{3.15-0}
	&\int_0^{t}e^{-\nu_{\mathfrak{c}}(p)(t-s)}ds\int_0^se^{-\nu_{\mathfrak{c}}(\tilde{p})(s-s_1)}ds_1\Big\{\int_{|\tilde{p}|\ge 2N}\int_{\mathbb{R}^3}+\int_{|\tilde{p}|\le 2N}\int_{|\tilde{p}'|\ge 3N}\Big\}\Big\{\cdot\Big\}d\tilde{p}'d\tilde{p}\nonumber\\
	&\le Ce^{-\f{1}{32}N}e^{-\la_1t}\sup\limits_{0\leq\tau\leq t}\|e^{\la_1\tau}\mathfrak{h}_{\mathfrak{c}}(\tau)\|_{L^\infty}.
\end{align}
{\it Case 3:} $s-\frac{1}{N}\le s_1\le s$ and $|p|\le N$, $|\tilde{p}|\le 2N$, $\tilde{p}'\le 3N$. It is easy to see
\begin{align}\label{3.16-0}
	&\int_0^{t}e^{-\nu_{\mathfrak{c}}(p)(t-s)}ds\int_{s-\frac{1}{N}}^se^{-\nu_{\mathfrak{c}}(\tilde{p})(s-s_1)}ds_1\int_{|\tilde{p}|\le 2N}\int_{|\tilde{p}'|\le 3N}\Big\{\cdot\Big\}d\tilde{p}'d\tilde{p}\nonumber\\
	&\le  \frac{C}{N}e^{-\lambda_1 t} \sup\limits_{0\leq\tau\leq t}\|e^{\la_1\tau}\mathfrak{h}_{\mathfrak{c}}(\tau)\|_{L^{\infty}}.
\end{align}
{\it Case 4:} $s-s_1\ge \frac{1}{N}$ and $|p|\le N$, $|\tilde{p}|\le 2N$, $\tilde{p}'\le 3N$. It holds 
\begin{align}\label{3.17-0}
	&\int_{|\tilde{p}|\le 2N}\int_{|\tilde{p}'|\le 3N} |k_{\mathfrak{c}w}(p,\tilde{p})k_{\mathfrak{c}w}(\tilde{p},\tilde{p}')|\cdot|\mathfrak{h}_{\mathfrak{c}}(s_1,\tilde{y}_{\mathfrak{c}},\tilde{p}')|d\tilde{p}'d\tilde{p}\nonumber\\
	&\le C_Ne^{-\lambda_1 s_1}\int_{|\tilde{p}|\le 2N}\int_{|\tilde{p}'|\le 3N} |k_{\mathfrak{c}w}(p,\tilde{p})k_{\mathfrak{c}w}(\tilde{p},\tilde{p}')|\cdot|e^{\lambda_1 s_1}f_{\mathfrak{c}}(s_1,\tilde{y}_{\mathfrak{c}},\tilde{p}')|d\tilde{p}'d\tilde{p}\nonumber\\
	&\le C_Ne^{-\lambda_1 s_1}\Big(\int_{|\tilde{p}|\le 2N}\int_{|\tilde{p}'|\le 3N} k^2_{\mathfrak{c}w}(p,\tilde{p})k^2_{\mathfrak{c}w}(\tilde{p},\tilde{p}')d\tilde{p}'d\tilde{p}\Big)^{\frac{1}{2}}\nonumber\\
	&\quad \times \Big(\int_{|\tilde{p}|\le 2N}\int_{|\tilde{p}'|\le 3N}|e^{\lambda_1 s_1}f_{\mathfrak{c}}(s_1,\tilde{y}_{\mathfrak{c}},\tilde{p}')|^2d\tilde{p}'d\tilde{p}\Big)^{\frac{1}{2}}\nonumber\\
	&\le C_N\big[(s-s_1)^{\frac{3}{2}}+1\big]e^{-\lambda_1 s_1}\Big(\int_{\mathbb{T}^3}\int_{\mathbb{R}^3} |e^{\lambda_1 s_1}f_{\mathfrak{c}}(s_1,\tilde{y}_{\mathfrak{c}},\tilde{p}')|^2d\tilde{y}_{\mathfrak{c}}d\tilde{p}'\Big)^{\frac{1}{2}}\nonumber\\
	&\le  C_N\big[(s-s_1)^{\frac{3}{2}}+1\big]e^{-\lambda_1 s_1} \sup\limits_{0\leq\tau\leq t}\|e^{\la_1\tau}f_{\mathfrak{c}}(\tau)\|_{L^2},
\end{align}
where we have made a change of variables $\tilde{p}\mapsto \tilde{y}_{\mathfrak{c}}$ with
\begin{align*}
	\Big|\text{det}\Big(\f{\partial \tilde{y}_{\mathfrak{c}}}{\partial \tilde{p}}\Big)\Big|=\f{\fc^5(s-s_1)^3}{(p^{\prime 0})^5}\geq \f{\fc^5N^{-3}}{(\fc^2+4N^2)^{\f52}}\geq \f{N^{-3}}{(1+4N^2)^{\f52}}.
\end{align*}
It follows that
\begin{align}\label{3.18-0}
	&\int_0^{t}e^{-\nu_{\mathfrak{c}}(p)(t-s)}ds\int_{0}^{s-\frac{1}{N}}e^{-\nu_{\mathfrak{c}}(\tilde{p})(s-s_1)}ds_1\int_{|\tilde{p}|\le 2N}\int_{|\tilde{p}'|\le 3N}\Big\{\cdot\Big\}d\tilde{p}'d\tilde{p}\nonumber\\
	&\le C_N\sup\limits_{0\leq\tau\leq t}\|e^{\la_1\tau}f_{\mathfrak{c}}(\tau)\|_{L^2}\int_0^{t}e^{-\nu_{\mathfrak{c}}(p)(t-s)}ds\int_{0}^{s-\frac{1}{N}}e^{-\nu_{\mathfrak{c}}(\tilde{p})(s-s_1)}e^{-\lambda_1 s_1}\big[(s-s_1)^{\frac{3}{2}}+1\big]ds_1\nonumber\\
	&\le C_Ne^{-\lambda_1 t}\sup\limits_{0\leq\tau\leq t}\|e^{\la_1\tau}f_{\mathfrak{c}}(\tau)\|_{L^2}.
\end{align}
Combining \eqref{3.13-0}--\eqref{3.18-0}, we obtain
\begin{align*} 
	|\mathcal{I}_{2,3}|\le C\max\Big\{\f{1}{\fc},\f{1}{N}\Big\}e^{-\la_1t}\sup\limits_{0\leq\tau\leq t}\|e^{\la_1\tau}\mathfrak{h}_{\mathfrak{c}}(\tau)\|_{L^\infty}+ C_Ne^{-\lambda_1 t}\sup\limits_{0\leq\tau\leq t}\|e^{\la_1\tau}f_{\mathfrak{c}}(\tau)\|_{L^2},
\end{align*}
which, together with \eqref{3.9-0}, \eqref{3.11-0} and \eqref{3.12-0}, yields that
\begin{align}\label{3.20-0}
	|\mathcal{I}_2|&\le Ce^{-\la_1t}\sup\limits_{0\leq\tau\leq t}\|e^{\la_1\tau}\nu_{\mathfrak{c}}^{-1}w_{\beta}\mathcal{S}(\tau)\|_{L^\infty}+C\max\Big\{\f{1}{\fc},\f{1}{N}\Big\}e^{-\la_1t}\sup\limits_{0\leq\tau\leq t}\|e^{\la_1\tau}\mathfrak{h}_{\mathfrak{c}}(\tau)\|_{L^\infty}\nonumber\\
	&\quad +Ce^{-\f{\nu_0}{2}t}\|\mathfrak{h}_{0,\mathfrak{c}}\|_{L^{\infty}}+ C_Ne^{-\lambda_1 t}\sup\limits_{0\leq\tau\leq t}\|e^{\la_1\tau}f_{\mathfrak{c}}(\tau)\|_{L^2}.
\end{align}
Hence \eqref{4.10} follows from \eqref{6.3}, \eqref{3.7-0}, \eqref{3.8-0} and \eqref{3.20-0}. Therefore the proof of Lemma \ref{lem4.3} is completed.
\end{proof}

\begin{Proposition}\label{prop3.2}
	Let $f_{\mathfrak{c}}$ be a solution of \eqref{4.6} under the condition that $f_{0,\mathfrak{c}}$ satisfies \eqref{1.8-10} with $(M_{\mathfrak{c}},\mathbf{J}_{\mathfrak{c}},E_{\mathfrak{c}})=(0,\mathbf{0},0)$ and 
	\begin{align}\label{3.21-0}
		\iint_{\mathbb{T}^3\times\mathbb{R}^3} \mathcal{S}(t,x,p) \sqrt{J_{\mathfrak{c}}(p)}
		\begin{pmatrix}
		1  \\  p \\ p^0
		\end{pmatrix}
		dpdx=0.
%		=\iint_{\mathbb{T}^3\times\mathbb{R}^3} \mathcal{S}(t,x,p) p_i\sqrt{J_{\mathfrak{c}}(p)}dpdx=\iint_{\mathbb{T}^3\times \mathbb{R}^3}\mathcal{S}(t,x,p)p^0\sqrt{J_{\mathfrak{c}}(p)}dpdx=0
	\end{align}
    Then there exists a function $G(t)$ such that, for all $0 \leq \tau \leq t$, $G(\tau) \lesssim\|f_{\mathfrak{c}}(\tau)\|_{L^2}^{2}$ and
	\begin{align}\label{4.28}
		\int_{s}^{t}\|\mathbf{P}_{\mathfrak{c}} f_{\mathfrak{c}}(\tau)\|_{\nu_{\mathfrak{c}}}^{2}d\tau &\lesssim
		\int_{s}^{t}\Big(\|(\mathbf{I}-\mathbf{P}_{\mathfrak{c}}) f_{\mathfrak{c}}(\tau)\|_{\nu_{\mathfrak{c}}}^{2}+\|\mathcal{S}(\tau)\|_{L^2}^{2}\Big)d\tau+|G(t)-G(s)|.
	\end{align}
     Here all the constants are independent of $\mathfrak{\mathfrak{c}}$.
\end{Proposition}
\begin{proof}
	By Green's identity, the weak formulation of \eqref{4.6} is 
	\begin{align}\label{3.23}
		&-\int_s^t\iint_{\mathbb{T}^3\times\R^3}(\hat{p}\cdot\nabla_x\psi)f_{\mathfrak{c}}dpdxd\tau-\int_s^t\iint_{\mathbb{T}^3\times\R^3}(\partial_{\tau} \psi)f_{\mathfrak{c}}dpdxd\tau\nonumber\\
		&=-\iint_{\mathbb{T}^3\times\R^3}\psi f_{\mathfrak{c}}(t)dpdx+\iint_{\mathbb{T}^3\times\R^3}\psi f_{\mathfrak{c}}(s)dpdx+\int_s^t\iint_{\mathbb{T}^3\times\R^3}(-\psi \mathbf{L}_{\mathfrak{c}}(\mathbf{I}-\mathbf{P}_{\mathfrak{c}})f_{\mathfrak{c}}+\psi \mathcal{S})dpdxd\tau.
	\end{align}
    Without loss of generality we only prove the case of $s=0$. We note that \eqref{1.8-10}, \eqref{4.6} and \eqref{3.21-0} are all invariant under a standard $t$-mollification for all $t>0$. The estimates in \textit{ Step 1} to \textit{Step 3} below are obtained via a $t$-mollification so that all the functions are smooth in $t$. For notational simplicity we do not write explicitly of the  regularization. 
    
    Denote $\mathbf{P}_{\mathfrak{c}} f_{\mathfrak{c}}= \Big\{a+b\cdot  p+c\f{p^0-A_3}{\sqrt{A_2-A_3^2}}\Big\}\sqrt{J_{\mathfrak{c}}(p)}$ on $\mathbb{T}^3\times \mathbb{R}^3$. 
    It follows from \eqref{1.8-10}, \eqref{4.6} and \eqref{3.21-0} that
    \begin{align*} 
    	\iint_{\mathbb{T}^3\times\mathbb{R}^3} f_{\mathfrak{c}}(t,x,p) \sqrt{J_{\mathfrak{c}}(p)}
    	\begin{pmatrix}
    		1  \\  p \\ p^0
    	\end{pmatrix}
    	dpdx=0,
    \end{align*}
%    	\begin{align*} 
%    		&\iint_{\mathbb{T}^3\times\mathbb{R}^3} f_{\mathfrak{c}}(t,x,p) \sqrt{J_{\mathfrak{c}}(p)}dpdx=\iint_{\mathbb{T}^3\times\mathbb{R}^3} f_{\mathfrak{c}}(t,x,p) p_i\sqrt{J_{\mathfrak{c}}(p)}dpdx\nonumber\\
%    		&=\iint_{\mathbb{T}^3\times \mathbb{R}^3}f_{\mathfrak{c}}(t,x,p)p^0\sqrt{J_{\mathfrak{c}}(p)}dpdx=0,
%    	\end{align*}
    which implies that
    \begin{align*}
	\int_{\mathbb{T}^3}a(t,x)dx=\int_{\mathbb{T}^3}b_i(t,x)dx=\int_{\mathbb{T}^3}c(t,x)dx=0,\quad t\ge 0.
    \end{align*}
    We define $\phi_a(t,x)$, $\phi^i_b(t,x)$ and $\phi_c(t,x)$ as the solutions of 
    \begin{align}
	    &-\Delta\phi_a(t,x)=a(t,x),\quad \int_{\mathbb{T}^3}\phi_a(t,x)dx=0,\label{3.25-0}\\ &-\Delta\phi_b^i(t,x)=b_i(t,x),\quad  \int_{\mathbb{T}^3}\phi_b^i(t,x)dx=0,\label{3.26-0}\\
	    &-\Delta\phi_c(t,x)=c(t,x),\quad \int_{\mathbb{T}^3}\phi_c(t,x)dx=0.\label{3.27-0}
    \end{align}   
It follows from the standard elliptic estimate that
\begin{align*}
	\|\phi_a\|_{H^2}\leq C\|a\|_{L^2},\quad \|\phi_b^i\|_{H^2}\leq C\|b_i\|_{L^2},\quad \|\phi_c\|_{H^2}\leq C\|c\|_{L^2}.
\end{align*}
Now we divide the proof into four steps.\\
\noindent{\it Step 1. Estimate on $\nabla_x\partial_t\phi_a$.} Choosing the test function $\psi=\phi(x)\sqrt{J_{\mathfrak{c}}(p)}$ with $\phi(x)$ depending only on $x$ and substituting it into \eqref{3.23} (with time integration over $[t,t+\varepsilon]$), we have from Lemma \ref{lem2.2} that 
\begin{align*}
	&\int_{\mathbb{T}^3}[a(t+\v)-a(t)]\phi(x)dx=\int_{t}^{t+\v}\iint_{\mathbb{T}^3\times\R^3}\hat{p}\cdot\nabla_x\phi\sqrt{J_{\mathfrak{c}}(p)}\cdot (\FI-\mathbf{P}_{\mathfrak{c}})f_{\mathfrak{c}}dpdxd\tau\nonumber\\
	&\quad +\int_t^{t+\v}\int_{\mathbb{T}^3}b\cdot\nabla_x\phi dxd\tau  +\int_{t}^{t+\v}\iint_{\mathbb{T}^3\times\R^3}\phi\sqrt{J_{\mathfrak{c}}(p)}\mathcal{S}dpdxd\tau.
\end{align*}
Taking the difference quotient, we have that for all $t>0$,
\begin{align*}
	\int_{\mathbb{T}^3}\phi\partial_ta(t)dx&=\iint_{\mathbb{T}^3\times\R^3}\hat{p}\cdot\nabla_x\phi\sqrt{J_{\mathfrak{c}}(p)}\cdot (\FI-\mathbf{P}_{\mathfrak{c}})f_{\mathfrak{c}}(t)dpdx\nonumber\\
	&\quad +\int_{\mathbb{T}^3}b(t)\cdot\nabla_x\phi dx+\iint_{\mathbb{T}^3\times\R^3}\phi\sqrt{J_{\mathfrak{c}}(p)}\mathcal{S}(t)dpdx,
\end{align*}
which yields that
\begin{align*}
	\Big|\int_{\mathbb{T}^3}\phi \partial_ta(t)dx\Big|\lesssim\{\|b(t)\|_{L^2}+\|(\FI-\mathbf{P}_{\mathfrak{c}})f_{\mathfrak{c}}(t)\|_{L^2}+\|\mathcal{S}(t)\|_{L^2}\}\cdot \|\phi\|_{H^1}.
\end{align*}
Hence, for all $t>0$, one has
\begin{align*}
	\|\partial_ta(t)\|_{(H^1)^*}\lesssim\|b(t)\|_{L^2}+\|(\FI-\mathbf{P}_{\mathfrak{c}})f_{\mathfrak{c}}(t)\|_{L^2}+\|\mathcal{S}(t)\|_{L^2},
\end{align*}
where $(H^1)^*=\big(H^1(\mathbb{T}^3)\big)^{*}$ is the dual space of $H^1(\mathbb{T}^3)$ with respect to the dual pair $\langle A, B\rangle=\int_{\mathbb{T}^3} A(x) B(x) d x$, for $A \in H^{1}$ and $B \in\left(H^{1}\right)^{*}$. 

By the standard elliptic theory, we can solve $-\Delta \Phi_{a}=\partial_{t} a(t)$, $\int_{\mathbb{T}^3}\Phi_{a}dx=0$. Noting that $\Phi_{a}=-\Delta^{-1} \partial_{t} a=\partial_{t} \phi_{a}$ with $\phi_{a}$ defined in \eqref{3.25-0}, thus we have
\begin{align}\label{3.31}
	\left\|\nabla_{x} \partial_{t} \phi_{a}\right\|_{L^2}&\le \left\|\Delta^{-1} \partial_{t} a(t)\right\|_{H^{1}}=\left\|\Phi_{a}\right\|_{H^{1}} \lesssim\left\|\partial_{t} a(t)\right\|_{\left(H^{1}\right)^{*}}\nonumber\\ &\le C\Big(\|b(t)\|_{L^2}+\|(\mathbf{I}-\mathbf{P}_{\mathfrak{c}}) f_{\mathfrak{c}}(t)\|_{L^2}+\|\mathcal{S}(t)\|_{L^2}\Big),
\end{align}
where the constant $C>0$ is independent of the light speed $\mathfrak{c}$.\\

\noindent \textit{Step 2. Estimate on $\nabla_{x} \partial_{t} \phi_{b}^{i}$.} Choosing the test function $\psi=\phi(x) p_{i} \sqrt{J_{\mathfrak{c}}}$ with $\phi(x)$ depending only on $x$ and substituting it into \eqref{3.23}, we have from Lemma \ref{lem2.2} that
\begin{align*}
	&A_1\int_{\mathbb{T}^3}\left[b_{i}(t+\varepsilon)-b_{i}(t)\right] \phi dx=\int_{t}^{t+\varepsilon}\int_{\mathbb{T}^3}\partial_i\phi\Big[a+\frac{1}{\mathfrak{c}\sqrt{A_2-A^2_3}}c\Big]dxd\tau \nonumber\\
	&\qquad +\int_{t}^{t+\varepsilon} \iint_{\mathbb{T}^3 \times \mathbb{R}^3} p_i\sqrt{J_{\mathfrak{c}}} \hat{p}\cdot \nabla_x \phi(\mathbf{I}-\mathbf{P}_{\mathfrak{c}}) f_{\mathfrak{c}} dpdxd\tau
	+\int_{t}^{t+\varepsilon} \iint_{\mathbb{T}^3 \times \mathbb{R}^3} \phi p_{i} \mathcal{S} \sqrt{J_{\mathfrak{c}}}dpdxd\tau.
\end{align*}
Taking the difference quotient, we can obtain
\begin{align} \label{3.25-000}
	&A_1\int_{\mathbb{T}^3} \partial_{t} b_{i}(t) \phi dx=\int_{\mathbb{T}^3} \partial_{i} \phi\Big[a(t)+\frac{1}{\mathfrak{c}\sqrt{A_2-A^2_3}}c(t)\Big]dx\nonumber\\
	&\qquad +\iint_{\mathbb{T}^3 \times \mathbb{R}^3}p_i\sqrt{J_{\mathfrak{c}}}\hat{p}\cdot \nabla_x \phi(\mathbf{I}-\mathbf{P}_{\mathfrak{c}}) f_{\mathfrak{c}}(t)dpdx
	+\iint_{\mathbb{T}^3 \times \mathbb{R}^3} \phi p_{i} \mathcal{S}(t) \sqrt{J_{\mathfrak{c}}}dpdx,
 \end{align}
 which, together with $A_1=1+O(\mathfrak{c}^{-2})$ and $\mathfrak{c}\sqrt{A_2-A_3^2}=\frac{\sqrt{6}}{2}+O(\mathfrak{c}^{-2})$, yields that 
\begin{align*}
		\|\partial_t b_i(t)\|_{(H^1)^*} \lesssim \|(a,c)(t)\|_{L^2}+\|(\FI-\mathbf{P}_{\mathfrak{c}})f_{\mathfrak{c}}(t)\|_{L^2}+\|\mathcal{S}(t)\|_{L^2}.
\end{align*}

For fixed $t>0$, we choose $\phi=\Phi_{b}^{i}$ with $-\Delta \Phi_{b}^{i}=\partial_{t} b_{i}(t)$, $\int_{\mathbb{T}^3}\Phi_{b}^{i}dx=0$. It is clear that $\Phi_{b}^{i}=-\Delta^{-1} \partial_{t} b_{i}=\partial_{t} \phi_{b}^{i}$, where $\phi_{b}^{i}$ is defined in \eqref{3.26-0}. By similar arguments as in \textit{Step 1}, we have
\begin{align} \label{4.30}
	\left\|\nabla_{x} \partial_{t} \phi^i_{b}\right\|_{L^2}&\le \left\|\Delta^{-1} \partial_{t} b_i(t)\right\|_{H^{1}}=\left\|\Phi^i_{b}\right\|_{H^{1}} \lesssim\left\|\partial_{t} b_i(t)\right\|_{\left(H^{1}\right)^{*}}\nonumber\\ &\le C\Big(\|(a,c)(t)\|_{L^2}+\|(\FI-\mathbf{P}_{\mathfrak{c}})f_{\mathfrak{c}}(t)\|_{L^2}+\|\mathcal{S}(t)\|_{L^2}\Big),
\end{align}
%\begin{align*}
%	&\left\|\nabla_{x} \partial_{t} \phi_{b}^{i}(t)\right\|^2_{L^2}
%	=\int_{\mathbb{T}^3}\left|\nabla_{x} \Phi_{b}^{i}\right|^{2} d x=-\int_{\mathbb{T}^3} \partial_{t} b_{i}(t) \, \Phi_{b}^{i} d x \\
%	&\lesssim  \varepsilon\left\{\left\|\nabla_{x} \Phi_{b}^{i}\right\|_{L^2}^{2}+\left\|\Phi_{b}^{i}\right\|_{L^2}^{2}\right\}+C_{\varepsilon}\Big(\|(a,c)(t)\|_{L^2}^{2}+\|(\mathbf{I}-\mathbf{P}_{\mathfrak{c}}) f_{\mathfrak{c}}(t)\|_{L^2}^{2}+\|\mathcal{S}(t)\|_{L^2}^{2}\Big) \\
%	&\lesssim \varepsilon\left\|\nabla_{x} \Phi_{b}^{i}\right\|_{L^2}^{2}+C_{\varepsilon}\Big(\|(a,c)(t)\|_{L^2}^{2}+\|(\mathbf{I}-\mathbf{P}_{\mathfrak{c}}) f_{\mathfrak{c}}(t)\|_{L^2}^{2}+\|\mathcal{S}(t)\|_{L^2}^{2}\Big).
%\end{align*}
%Taking $\varepsilon$ suitably small, for all $t>0$, we have
%\begin{align}\label{4.30}
%	\left\|\nabla_{x} \partial_{t} \phi_{b}^{i}(t)\right\|_{L^2} \le C\Big(\|(a,c)(t)\|_{L^2}+\|(\mathbf{I}-\mathbf{P}_{\mathfrak{c}}) f_{\mathfrak{c}}(t)\|_{L^2}+\|\mathcal{S}(t)\|_{L^2}\Big),
%\end{align}
where the constant $C>0$ is independent of the light speed $\mathfrak{c}$.\\

\noindent \textit {Step 3. Estimate on $\nabla_{x} \partial_{t} \phi_{c}$}. We choose the test function $\psi=\phi(x)\frac{p^0-A_3}{\sqrt{A_2-A_3^2}} \sqrt{J_{\mathfrak{c}}}$ in \eqref{3.23}, then it follows from Lemma \ref{lem2.2} that
\begin{align*}
	\int_{\mathbb{T}^3} \phi(x) [c(t+\varepsilon)-c(t)]dx&=
	\sum_{i=1}^3\int_{t}^{t+\varepsilon} \iint_{\mathbb{T}^3\times \mathbb{R}^3} \frac{\mathfrak{c}}{p^0}p_i^2 \frac{p^0-A_3}{\sqrt{A_2-A_3^2}}J_{\mathfrak{c}}(p)b_i \cdot \partial_{i}\phi dpdxd\tau\nonumber\\
	&\quad + \int_{t}^{t+\varepsilon} \iint_{\mathbb{T}^3 \times \mathbb{R}^3} \phi \mathcal{S}\frac{p^0-A_3}{\sqrt{A_2-A_3^2}}\sqrt{J_{\mathfrak{c}}}dpdxd\tau\nonumber\\
	&\quad +\int_{t}^{t+\varepsilon} \iint_{\mathbb{T}^3 \times \mathbb{R}^3}(\mathbf{I}-\mathbf{P}_{\mathfrak{c}}) f_{\mathfrak{c}} \nabla_x\phi\cdot p \sqrt{J_{\mathfrak{c}}}\frac{\mathfrak{c}}{p^0}\frac{p^0-A_3}{\sqrt{A_2-A_3^2}}dpdxd\tau.
\end{align*}
Taking the difference quotient, we obtain
\begin{align*}
	\int_{\mathbb{T}^3} \phi(x)\partial_t{c(t)}d x&=
	\sum_{i=1}^3\iint_{\mathbb{T}^3\times \mathbb{R}^3}\frac{\mathfrak{c}}{p^0}p_i^2 \frac{p^0-A_3}{\sqrt{A_2-A_3^2}}J_{\mathfrak{c}}(p)b_i(t) \cdot \partial_{i}\phi dpdx\nonumber\\
	&\quad +\iint_{\mathbb{T}^3 \times \mathbb{R}^3}(\mathbf{I}-\mathbf{P}_{\mathfrak{c}}) f_{\mathfrak{c}}(t) \nabla_x\phi\cdot p \sqrt{J_{\mathfrak{c}}}\frac{\mathfrak{c}}{p^0}\frac{p^0-A_3}{\sqrt{A_2-A_3^2}}dpdx\nonumber\\
	&\quad +\iint_{\mathbb{T}^3 \times \mathbb{R}^3} \phi \mathcal{S}(t)\frac{p^0-A_3}{\sqrt{A_2-A_3^2}}\sqrt{J_{\mathfrak{c}}}dpdx,
\end{align*}
which yields that 
\begin{align*}
	\|\partial_t c(t)\|_{(H^1)^*} \lesssim \|b(t)\|_{L^2}+\|(\FI-\mathbf{P}_{\mathfrak{c}})f_{\mathfrak{c}}(t)\|_{L^2}+\|\mathcal{S}(t)\|_{L^2}.
\end{align*}

For fixed $t>0$, we choose $\phi=\Phi_{c}$ with $-\Delta \Phi_{c}=\partial_{t} c(t)$, $\int_{\mathbb{T}^3}\Phi_{c}dx=0$. It is clear that $\Phi_{c}=-\Delta^{-1} \partial_{t} c(t)=\partial_{t} \phi_{c}$, where $\phi_{c}$ is defined in \eqref{3.27-0}. By similar arguments as in \textit{Step 1}, it holds that
\begin{align}\label{4.33}
	\left\|\nabla_{x} \partial_{t} \phi_{c}\right\|_{L^2}&\le \left\|\Delta^{-1} \partial_{t} c(t)\right\|_{H^{1}}=\left\|\Phi_{c}\right\|_{H^{1}} \lesssim\left\|\partial_{t} c(t)\right\|_{\left(H^{1}\right)^{*}}\nonumber\\ 
	&\le C\Big(\|b(t)\|_{L^2}+\|(\FI-\mathbf{P}_{\mathfrak{c}})f_{\mathfrak{c}}(t)\|_{L^2}+\|\mathcal{S}(t)\|_{L^2}\Big),
\end{align}
%\begin{align*}
%	\left\|\nabla_{x} \partial_{t} \phi_{c}\right\|_{L^2}^{2}
%	&=\int_{\mathbb{T}^3}\left|\nabla_{x} \Phi_{c}(x)\right|^{2} d x=\int_{\mathbb{T}^3} \Phi_{c}(x) \partial_{t} c(t) d x \\
%	& \le  \varepsilon\left\|\nabla_{x} \Phi_{c}\right\|_{L^2}^{2}+C_{\varepsilon}\Big(\|b(t)\|_{L^2}^{2}+\|(\mathbf{I}-\mathbf{P}_{\mathfrak{c}}) f_{\mathfrak{c}}(t)\|_{L^2}^{2}+\|\mathcal{S}(t)\|_{L^2}^{2}\Big).
%\end{align*}
%By taking $\varepsilon$ suitably small, we have, for all $t>0$,
%\begin{align}\label{4.33}
%	\left\|\nabla_{x} \partial_{t} \phi_{c}\right\|_{L^2}  \le C\Big(\|b(t)\|_{L^2}+\|(\mathbf{I}-\mathbf{P}_{\mathfrak{c}}) f_{\mathfrak{c}}(t)\|_{L^2}+\|\mathcal{S}(t)\|_{L^2}\Big),
%\end{align}
where the constant $C>0$ is independent of the light speed $\mathfrak{c}$.\\ 

\noindent{\it Step 4. Estimates on $a,b,c$.} Choose $G(\tau):=-\iint_{\mathbb{T}^3\times\R^3}\psi f_{\mathfrak{c}}(\tau)dpdx$. Due to the choices of $\psi$ in the following cases, we have $|G(\tau)|\lesssim\|f_{\mathfrak{c}}(\tau)\|_{L^2}^2$.

\noindent{\it Step 4.1. Estimate on $c$.} To estimate $c$, we  choose the test function
\begin{align*}
	\psi_c=(|p|^2-\beta_c)\sqrt{J_{\mathfrak{c}}(p)}p\cdot\nabla_x\phi_c,
\end{align*}
where $\phi_c$ is defined in \eqref{3.27-0} and $\beta_c$ satisfies
\begin{align}\label{3.36-0}
	\int_{\R^3}\frac{\mathfrak{c}}{p^0}(|p|^2-\beta_c)p_i^2J_{\mathfrak{c}}(p)dp=0.
\end{align}
A direct calculation shows that $\beta_c=\f{A_6}{A_4}=5\f{K_3(\fc^2)}{K_2(\fc^2)}=5+O(\fc^{-2})$. Since 
$$
\hat{p} \cdot \nabla_{x} \psi_{c}=\sum_{i, j=1}^{3}\frac{\mathfrak{c}}{p^0}\left(|p|^2-\beta_{c}\right)  \sqrt{J_{\mathfrak{c}}(p)} p_{i}p_{j}\partial_{i j} \phi_{c},
$$ 
and, in view of Lemma \ref{lem2.2} and \eqref{2.2-30}--\eqref{2.3-30},
\begin{align*}
	\int_{\mathbb{R}^3}\frac{\mathfrak{c}}{p^0}\left(|p|^2-\beta_{c}\right)\frac{p^0-A_3}{\sqrt{A_2-A^2_3}}p^2_{i}J_{\mathfrak{c}}(p)dp=\frac{5\sqrt{6}}{3}
	+O(\mathfrak{c}^{-2}),
\end{align*}
then the first term on the LHS of \eqref{3.23} is
\begin{align}\label{3.157}
	&-\int_0^t\iint_{\mathbb{T}^3\times\R^3}(\hat{p}\cdot\nabla_x\psi_c)f_{\mathfrak{c}} dpdxd\tau\nonumber\\
	&=-\int_0^t\iint_{\mathbb{T}^3 \times \mathbb{R}^3}\frac{\mathfrak{c}}{p^0}\left(|p|^2-\beta_{c}\right) J_{\mathfrak{c}}(p)\Big\{\sum_{i, j=1}^{3} p_{i} p_{j} \partial_{i j} \phi_{c}\Big\}\cdot \Big\{a+b\cdot p+c\f{p^0-A_3}{\sqrt{A_2-A_3^2}}\Big\}dpdx\nonumber\\
	&\quad -\int_0^t\iint_{\mathbb{T}^3 \times \mathbb{R}^3}\frac{\mathfrak{c}}{p^0}\left(|p|^2-\beta_{c}\right) \sqrt{J_{\mathfrak{c}}}\Big\{\sum_{i, j=1}^{3} p_{i} p_{j} \partial_{i j} \phi_{c}\Big\}\cdot (\FI-\mathbf{P}_{\mathfrak{c}})f_{\mathfrak{c}}dpdx\nonumber\\
	&=-\sum_{i=1}^{3} \int_0^t\iint_{\mathbb{T}^3 \times \mathbb{R}^3}\frac{\mathfrak{c}}{p^0}\left(|p|^2-\beta_{c}\right) J_{\mathfrak{c}}(p)\Big\{p^2_{i} \partial_{ii} \phi_{c}\Big\}\cdot c\f{p^0-A_3}{\sqrt{A_2-A_3^2}}dpdx\nonumber\\
	&\quad -\int_0^t\iint_{\mathbb{T}^3 \times \mathbb{R}^3}\frac{\mathfrak{c}}{p^0}\left(|p|^2-\beta_{c}\right) \sqrt{J_{\mathfrak{c}}}\Big\{\sum_{i, j=1}^{3} p_{i} p_{j} \partial_{i j} \phi_{c}\Big\}\cdot (\FI-\mathbf{P}_{\mathfrak{c}})f_\fc dpdx\nonumber\\
    &\ge  \Big(\f53\sqrt{6}+O(\fc^{-2})\Big)\int_0^t\|c\|_{L^2}^2d\tau-C\int_0^t\|c\|_{L^2}\|(\FI-\mathbf{P}_{\mathfrak{c}})f_{\mathfrak{c}}\|_{L^2}d\tau,
\end{align}
where the $b$ contribution vanishes due to the oddness in $p$ and the $a$ contribution vanishes because of the choice of $\beta_c$. The second term on the LHS of \eqref{3.23} has the form
\begin{align}\label{3.38-0}
	&-\int_0^t\iint_{\mathbb{T}^3\times\R^3}\partial_{\tau} \psi_c f_{\mathfrak{c}}dpdxd\tau\nonumber\\
	&=-\sum_{i=1}^3\int_0^t\iint_{\mathbb{T}^3\times\R^3}(|p|^2-\beta_c)p_iJ_{\mathfrak{c}}(p)\partial_{\tau} \partial_i\phi_c\Big\{a+b\cdot p+c\f{p^0-A_3}{\sqrt{A_2-A_3^2}}\Big\}dpdxd\tau\nonumber\\
	&\quad -\sum_{i=1}^3\int_0^t\iint_{\mathbb{T}^3\times\R^3}(|p|^2-\beta_c)p_i\sqrt{J_{\mathfrak{c}}(p)}\partial_{\tau} \partial_i\phi_c\cdot(\FI-\mathbf{P}_{\mathfrak{c}})f_{\mathfrak{c}}dpdxd\tau:=\mathcal{P}_1+\mathcal{P}_2.
\end{align}
 The $a$, $c$ contributions in $\mathcal{P}_1$ vanish due to the oddness in $p$. Using \eqref{3.36-0}, one has
\begin{align*}
	\left|\int_{\mathbb{R}^3}\left(|p|^2-\beta_{c}\right)p^2_i J_{\mathfrak{c}}(p)dp\right|&=\left|\int_{\mathbb{R}^3}(1-\frac{\mathfrak{c}}{p^0})\left(|p|^2-\beta_{c}\right)p^2_i J_{\mathfrak{c}}(p)dp\right|
	\lesssim \frac{1}{\mathfrak{c}^2},
\end{align*}
which yields that
\begin{align}\label{3.40-0}
	|\mathcal{P}_1|\lesssim \f{1}{\fc^2}\int_0^t\|\partial_{\tau} \nabla_x\phi_c\|_{L^2}\cdot\|b\|_{L^2}d\tau.
\end{align}
It is clear that
\begin{align}\label{3.41-0}
	|\mathcal{P}_2|\lesssim \int_0^t\|\partial_{\tau} \nabla_x\phi_c\|_{L^2}\cdot \|(\FI-\mathbf{P}_{\mathfrak{c}})f_{\mathfrak{c}}\|_{L^2}d\tau.
\end{align}
Combing \eqref{4.33},\eqref{3.38-0}--\eqref{3.41-0}, one obtains that
\begin{align}\label{3.41-1}
	&\Big|\int_0^t\iint_{\mathbb{T}^3\times\R^3}\partial_{\tau} \psi_cf_{\mathfrak{c}}dpdxd\tau\Big|\nonumber\\
	&\lesssim\int_0^t\{\|b(\tau)\|_{L^2}+\|(\FI-\mathbf{P}_{\mathfrak{c}})f_{\mathfrak{c}}(\tau)\|_{L^2}+\|\mathcal{S}(\tau)\|_{L^2}\}\cdot\{\|(\FI-\mathbf{P}_{\mathfrak{c}})f_{\mathfrak{c}}(\tau)\|_{L^2}+\f{1}{\fc^2}\|b(\tau)\|_{L^2}\}d\tau\nonumber\\
	&\leq C_\v\int_0^t\{\|(\FI-\mathbf{P}_{\mathfrak{c}})f_{\mathfrak{c}}(\tau)\|_{L^2}^2+\|\mathcal{S}(\tau)\|_{L^2}^2\}d\tau+C\Big(\v+\f{1}{\fc^2}\Big)\int_0^t\|b(\tau)\|_{L^2}^2d\tau.
\end{align}
   It is easy to see that the RHS of \eqref{3.23} is bounded uniformly by 
   \begin{align}\label{3.42}
   	|G_c(t)-G_c(0)|+\int_0^t	\|c(\tau)\|_{L^2}\left\{\|(\mathbf{I}-\mathbf{P}_{\mathfrak{c}}) f_{\mathfrak{c}}(\tau)\|_{L^2}+\|\mathcal{S}(\tau)\|_{L^2}\right\}d\tau.
   \end{align}
It follows from \eqref{3.23}, \eqref{3.157}, \eqref{3.41-1} and \eqref{3.42} that
\begin{align}\label{3.43}
	\int_0^t\|c(\tau)\|_{L^2}^2d\tau&\leq C_\v \int_0^t\{\|(\FI-\mathbf{P}_{\mathfrak{c}})f_{\mathfrak{c}}(\tau)\|_{L^2}^2+\|\mathcal{S}(\tau)\|_{L^2}^2\}d\tau+|G_c(t)-G_c(0)|\nonumber\\
	&\quad +C\Big(\v+\f{1}{\fc^2}\Big)\int_0^t\|b(\tau)\|_{L^2}^2d\tau.
\end{align}

\noindent{\it Step 4.2. Estimate on $b$.} To estimate $b$, we choose the test function 
\begin{align}\label{3.44-0}
\psi_b^{ij}=(p_i^2-\beta_b)\sqrt{J_{\mathfrak{c}}(p)}\partial_j\phi_b^j,
\end{align}
where $\phi_b^j$ is defined in \eqref{3.26-0} and $\beta_b$ satisfies
\begin{align*}
	\int_{\R^3}\frac{\mathfrak{c}}{p^0}(p_i^2-\beta_b)p_k^2J_{\mathfrak{c}}(p)dp=0,\quad k\neq i.
\end{align*}
Using Lemma \ref{lem2.2}, one has $\beta_b=\f{K_3(\fc^2)}{K_2(\fc^2)}=1+O(\mathfrak{c}^{-2})$. Substituting \eqref{3.44-0} into \eqref{3.23} and using the fact that
\begin{equation*} 
	\int_{\mathbb{R}^3}\frac{\mathfrak{c}}{p^0}\left(p_{i}^{2}-\beta_{b}\right) p_{i}^{2} J_{\mathfrak{c}}(p) d p=2\frac{K_3(\mathfrak{c}^2)}{K_2(\mathfrak{c}^2)}=2+O(\mathfrak{c}^{-2}),
\end{equation*}
one has
\begin{align}\label{3.46}
	&-\int_0^t\iint_{\mathbb{T}^3\times\R^3}(\hat{p}\cdot\nabla_x\psi_b^{ij})f_{\mathfrak{c}}dpdxd\tau\nonumber\\
	&=-\sum_{k=1}^3\int_0^t\iint_{\mathbb{T}^3\times\R^3}\frac{\mathfrak{c}}{p^0}(p_i^2-\beta_b)p_k\partial_{jk}\phi_b^jJ_{\mathfrak{c}}(p)\cdot \Big\{a+b\cdot p+c\f{p^0-A_3}{\sqrt{A_2-A_3^2}}\Big\}dpdxd\tau\nonumber\\
	&\quad -\int_0^t\iint_{\mathbb{T}^3\times\R^3}(p_i^2-\beta_b)\sqrt{J_{\mathfrak{c}}(p)}(\hat{p}\cdot\nabla_x)\partial_j\phi_b^j\cdot (\FI-\mathbf{P}_{\mathfrak{c}})f_{\mathfrak{c}}dpdxd\tau\nonumber\\
	&=-\sum_{l=1}^3\int_0^t\iint_{\mathbb{T}^3\times\R^3}\frac{\mathfrak{c}}{p^0}(p_i^2-\beta_b)p^2_l\partial_{lj}\phi_b^jJ_{\mathfrak{c}}(p)\cdot b_l(x)dpdxd\tau\nonumber\\
	&\quad -\int_0^t\iint_{\mathbb{T}^3\times\R^3}(p_i^2-\beta_b)\sqrt{J_{\mathfrak{c}}(p)}(\hat{p}\cdot\nabla_x)\partial_j\phi_b^j\cdot (\FI-\mathbf{P}_{\mathfrak{c}})f_{\mathfrak{c}}dpdxd\tau\nonumber\\
	&=(-2+O(\mathfrak{c}^{-2}))\int_0^t\int_{\mathbb{T}^3}\partial_{ij}\phi_b^j\cdot b_i(x)dxd\tau\nonumber\\
	&\quad-\int_0^t\iint_{\mathbb{T}^3\times\R^3}(p_i^2-\beta_b)\sqrt{J_{\mathfrak{c}}(p)}(\hat{p}\cdot\nabla_x)\partial_j\phi_b^j\cdot (\FI-\mathbf{P}_{\mathfrak{c}})f_{\mathfrak{c}}dpdxd\tau,
\end{align}
    where the $a$, $c$ contributions vanish due to the oddness in $p$ and, by the choice of $\beta_b$, there is only one term $l=i$ left in the summation. 
    
For the second term on the LHS of \eqref{3.23}, it holds
\begin{align*}
	&-\int_0^t\iint_{\mathbb{T}^3\times\R^3}\partial_{\tau} \psi_b^{ij}f_{\mathfrak{c}}dpdxd\tau\nonumber\\
	&=-\int_0^t\iint_{\mathbb{T}^3\times\R^3}(p_i^2-\beta_b)J_{\mathfrak{c}}(p)\partial_{\tau} \partial_j\phi_b^j\Big\{a+b\cdot p+c\f{p^0-A_3}{\sqrt{A_2-A_3^2}}\Big\}dpdxd\tau\nonumber\\
	&\quad-\int_0^t\iint_{\mathbb{T}^3\times\R^3}(p_i^2-\beta_b)\sqrt{J_{\mathfrak{c}}(p)}\partial_{\tau} \partial_j\phi_b^j\cdot (\FI-\mathbf{P}_{\mathfrak{c}})f_{\mathfrak{c}}dpdxd\tau.
\end{align*}
The $b$ contribution vanishes due to the oddness in $p$. Since $\beta_b=\f{K_3(\fc^2)}{K_2(\fc^2)}$, it is direct to check that
\begin{align*}
	\int_{\R^3}(p_i^2-\beta_b)J_{\mathfrak{c}}(p)dp=0
\end{align*}
and
\begin{align*}
	\int_{\R^3}(p_i^2-\beta_b)J_{\mathfrak{c}}(p)\f{p^0-A_3}{\sqrt{A_2-A_3^2}}dp=\f{A_5-A_1A_3}{\sqrt{A_2-A_3^2}}=\frac{\sqrt{6}}{3}+O(\fc^{-2}),
\end{align*}
which, together with \eqref{4.30}, yields that
\begin{align}\label{3.48}
	&\Big|\int_0^t\iint_{\mathbb{T}^3\times\R^3}\partial_{\tau} \psi_b^{ij}f_{\mathfrak{c}}dpdxd\tau\Big|\nonumber\\
	&\lesssim \int_0^t\|\partial_{\tau} \nabla_x\phi_b^j\|_{L^2}\cdot\{\|c\|_{L^2}+\|(\FI-\mathbf{P}_{\mathfrak{c}})f_{\mathfrak{c}}\|_{L^2}\}d\tau\nonumber\\
	&\lesssim\int_0^t\{\|(a,c)\|_{L^2}+\|(\FI-\mathbf{P}_{\mathfrak{c}})f_{\mathfrak{c}}\|_{L^2}+\|\mathcal{S}\|_{L^2}\}\cdot\{\|(\FI-\mathbf{P}_{\mathfrak{c}})f_{\mathfrak{c}}\|_{L^2}+\|c\|_{L^2}\}d\tau\nonumber\\
	&\leq \frac{C}{\sqrt{\varepsilon}}\int_0^t\{\|(\FI-\mathbf{P}_{\mathfrak{c}})f_{\mathfrak{c}}\|_{L^2}^2+\|\mathcal{S}\|_{L^2}+\|c\|_{L^2}^2\}d\tau+C\sqrt{\v}\int_0^t\|a\|_{L^2}^2d\tau.
\end{align}
The RHS of \eqref{3.23} is  bounded  by 
\begin{align}\label{3.49}
	|G_b(t)-G_b(0)|+\int_0^t	\|b(\tau)\|_{L^2}\left\{\|(\mathbf{I}-\mathbf{P}_{\mathfrak{c}}) f_{\mathfrak{c}}(\tau)\|_{L^2}+\|\mathcal{S}(\tau)\|_{L^2}\right\}d\tau.
\end{align}
Combining \eqref{3.23}, \eqref{3.46} and \eqref{3.48}, for all $i$, $j$, one has that
\begin{align}\label{eq2.40}
	\Big|\int_0^t\int_{\mathbb{T}^3}(\partial_{ij}\Delta^{-1}b_j)b_idxd\tau\Big|&\leq |G_b(t)-G_b(0)|+C\sqrt{\v}\int_0^t \|(a,b)(\tau)\|_{L^2}^2d\tau\nonumber\\
	&\qquad +\frac{C}{\sqrt{\varepsilon}}\int_0^t\{\|(\FI-\mathbf{P}_{\mathfrak{c}})f\|_{L^2}^2+\|\mathcal{S}\|_{L^2}^2+\|c\|_{L^2}^2\}d\tau.
\end{align}

When $i\neq j$, we choose the test function as $\psi=|p|^2p_ip_j\sqrt{J_{\mathfrak{c}}(p)}\partial_j\phi_b^i$ and substitute it into \eqref{3.23}, then the first term on the LHS is
	\begin{align*}
	&-\int_0^t\iint_{\mathbb{T}^3\times\R^3}(\hat{p}\cdot\nabla_x\psi)f_{\mathfrak{c}}dpdxd\tau\nonumber\\
	&=-\int_0^t\iint_{\mathbb{T}^3 \times \mathbb{R}^3}\frac{\mathfrak{c}}{p^0}|p|^{2} p_{i} p_{j} \sqrt{J_{\mathfrak{c}}}\left\{\sum_{k=1}^3 p_{k} \partial_{k j} \phi_{b}^{i}\right\} f_{\mathfrak{c}}dpdxd\tau \nonumber\\
	&=-\int_0^t\iint_{\mathbb{T}^3 \times \mathbb{R}^3}\frac{\mathfrak{c}}{p^0}|p|^{2} p_{i}^{2} p_{j}^{2} J_{\mathfrak{c}}(p)\left[\partial_{i j} \phi_{b}^{i} b_{j}+\partial_{j j} \phi_{b}^{i}(x) b_{i}\right]dpdxd\tau \nonumber\\
	&\qquad -\sum_{k=1}^3\int_0^t\iint_{\mathbb{T}^3 \times \mathbb{R}^3}\frac{\mathfrak{c}}{p^0}|p|^{2} p_{i} p_{j} p_{k} \sqrt{J_{\mathfrak{c}}} \partial_{k j} \phi_{b}^{i}(x)\cdot (\mathbf{I}-\mathbf{P}_{\mathfrak{c}}) f_{\mathfrak{c}}dpdxd\tau\nonumber\\
	&=(7+O(\mathfrak{c}^{-2}))\int_0^t\int_{\mathbb{T}^3}\{(\partial_{ij}\Delta^{-1}b_i)b_j+(\partial_{jj}\Delta^{-1}b_i)b_i\}dxd\tau \nonumber\\
	&\qquad -\sum_{k=1}^3\int_0^t\iint_{\mathbb{T}^3 \times \mathbb{R}^3}\frac{\mathfrak{c}}{p^0}|p|^{2} p_{i} p_{j} p_{k} \sqrt{J_{\mathfrak{c}}} \partial_{k j} \phi_{b}^{i}(x)\cdot (\mathbf{I}-\mathbf{P}_{\mathfrak{c}}) f_{\mathfrak{c}}dpdxd\tau,
\end{align*}
where we have used the fact that
\begin{align*}
	\int_{\mathbb{R}^3}\frac{\mathfrak{c}}{p^0}|p|^{2} p_{i}^{2} p_{j}^{2} J_{\mathfrak{c}}(p)dp=\mathfrak{c}A_8=7\frac{K_4(\mathfrak{c}^2)}{K_2(\mathfrak{c}^2)}=7+O(\mathfrak{c}^{-2}),\quad i\neq j.
\end{align*}
Using \eqref{4.30}, we can bound the second term on the RHS of \eqref{3.23} as
\begin{align}\label{3.52}
	\Big|\int_0^t\iint_{\mathbb{T}^3\times\R^3}\partial_{\tau} \psi f_{\mathfrak{c}}dpdxd\tau\Big|&=\Big|-\int_0^t\iint_{\mathbb{T}^3\times\R^3}|p|^2p_ip_j\sqrt{J_{\mathfrak{c}}(p)}\partial_{\tau} \partial_j\phi_b^i\cdot (\FI-\mathbf{P}_{\mathfrak{c}})f_{\mathfrak{c}}dpdxd\tau\Big|\nonumber\\
	&\lesssim\int_0^t\{\|(a,c)\|_{L^2}+\|(\FI-\mathbf{P}_{\mathfrak{c}})f_{\mathfrak{c}}\|_{L^2}+\|\mathcal{S}\|_{L^2}\}\cdot\|(\FI-\mathbf{P}_{\mathfrak{c}})f_{\mathfrak{c}}\|_{L^2}d\tau\nonumber\\
	&\le \frac{C}{\sqrt{\varepsilon}}\int_0^t \Big(\|(\FI-\mathbf{P}_{\mathfrak{c}})f_{\mathfrak{c}}\|_{L^2}^2+\|\mathcal{S}\|_{L^2}^2\Big)d\tau+C\sqrt{\varepsilon} \int_0^t \|(a,c)(\tau)\|_{L^2}^2 d\tau.
\end{align}
The RHS of \eqref{3.23} is again bounded uniformly by \eqref{3.49}.  Hence
it follows from \eqref{3.23} and \eqref{3.49}--\eqref{3.52} that, for $i\neq j$,
\begin{align}\label{3.53}
	&\Big|\int_0^t\int_{\mathbb{T}^3}(\partial_{jj}\Delta^{-1}b_i)b_idxd\tau\Big|\nonumber\\
	&\leq C\Big|\int_0^t\int_{\mathbb{T}^3}(\partial_{ij}\Delta^{-1}b_i)b_jdxd\tau\Big|+ \frac{C}{\sqrt{\varepsilon}}\int_0^t \{\|(\FI-\mathbf{P}_{\mathfrak{c}})f_{\mathfrak{c}}\|_{L^2}^2+\|\mathcal{S}\|_{L^2}^2+\|c\|_{L^2}^2\}d\tau\nonumber\\
	&\quad +|G_b(t)-G_b(0)|+C\sqrt{\v}\int_0^t \|(a,b)\|_{L^2}^2 d\tau\nonumber\\
	&\le \frac{C}{\sqrt{\varepsilon}}\int_0^t \{\|(\FI-\mathbf{P}_{\mathfrak{c}})f_{\mathfrak{c}}\|_{L^2}^2+\|\mathcal{S}\|_{L^2}^2+\|c\|_{L^2}^2\}d\tau  +|G_b(t)-G_b(0)|+C\sqrt{\v}\int_0^t \|(a,b)\|_{L^2}^2 d\tau.
\end{align}
It follows from \eqref{eq2.40} that
\begin{align}\label{3.54}
	\Big|\int_0^t\int_{\mathbb{T}^3}(\partial_{ii}\Delta^{-1}b_i)b_idxd\tau\Big|&\leq |G_b(t)-G_b(0)|+C\sqrt{\v}\int_0^t \|(a,b)(\tau)\|_{L^2}^2 d\tau\nonumber\\
	&\qquad +\frac{C}{\sqrt{\varepsilon}}\int_0^t\{\|(\FI-\mathbf{P}_{\mathfrak{c}})f_{\mathfrak{c}}\|_{L^2}^2+\|\mathcal{S}\|_{L^2}^2+\|c\|_{L^2}^2\}d\tau.
\end{align}
Combining \eqref{3.53} and \eqref{3.54}, we obtain
\begin{align}\label{3.55}
	\int_0^t\|b(\tau)\|_{L^2}^2d\tau\leq&|G_b(t)-G_b(0)|+C\sqrt{\v}\int_0^t\|a\|_{L^2}^2d\tau\nonumber\\
	&+\frac{C}{\sqrt{\varepsilon}}\int_0^t\{\|(\FI-\mathbf{P}_{\mathfrak{c}})f_{\mathfrak{c}}\|_{L^2}^2+\|\mathcal{S}\|_{L^2}^2+\|c\|_{L^2}^2\}d\tau.
\end{align}

\noindent{\it Step 4.3. Estimate on $a$.}  We choose the test function 
\begin{align}\label{3.56}
	\psi_a=-(|p|^2-\beta_a)\sqrt{J_{\mathfrak{c}}(p)}p\cdot\nabla_x\phi_a,
\end{align} 
where $\phi_a$ is defined in \eqref{3.25-0} and $\beta_a$ satisfies
\begin{align*}
	\int_{\R^3}\f{\fc}{p^0}(|p|^2-\beta_a)p_i^2(p^0-A_3)J_{\mathfrak{c}}(p)dp=0.
\end{align*}
A direct calculation shows that $\beta_a=\f{A_{10}-A_3A_6}{A_1-A_3A_4}=10+O(\fc^{-2})$. Substituting \eqref{3.56} into \eqref{3.23} and using the fact that
\begin{align*}
	\int_{ \mathbb{R}^3}\frac{\mathfrak{c}}{p^0}\left(|p|^2-\beta_{a}\right)p_{i}^{2} J_{\mathfrak{c}}(p)dp =-30\frac{K_3(\mathfrak{c}^2)}{K_2(\mathfrak{c}^2)}+5\mathfrak{c}^2\Big\{\Big(\frac{K_3(\mathfrak{c}^2)}{K_2(\mathfrak{c}^2)}\Big)^2-1\Big\}=-5+O(\mathfrak{c}^{-2}),
\end{align*}
 one has
\begin{align}\label{3.58}
	&-\int_0^t\iint_{\mathbb{T}^3\times\R^3}(\hat{p}\cdot\nabla_x\psi_a)f_{\mathfrak{c}}dpdxd\tau\nonumber\\
	&=\sum_{i,j=1}^3\int_0^t\iint_{\mathbb{T}^3\times\R^3}\frac{\mathfrak{c}}{p^0}(|p|^2-\beta_a)J_{\mathfrak{c}}(p)p_ip_j\partial_{ij}\phi_a\cdot \Big\{a+b\cdot p+c\f{p^0-A_3}{\sqrt{A_2-A_3^2}}\Big\}dpdxd\tau\nonumber\\
	&\qquad +\sum_{i,j=1}^3\int_0^t\iint_{\mathbb{T}^3\times\R^3}\frac{\mathfrak{c}}{p^0}(|p|^2-\beta_a)\sqrt{J_{\mathfrak{c}}(p)}p_ip_j\partial_{ij}\phi_a\cdot (\FI-\mathbf{P}_{\mathfrak{c}})f_{\mathfrak{c}} dpdxd\tau\nonumber\\
	&=\sum_{i=1}^3\int_0^t\iint_{\mathbb{T}^3\times\R^3}\frac{\mathfrak{c}}{p^0}(|p|^2-\beta_a)J_{\mathfrak{c}}(p)p^2_i\partial_{ii}\phi_a\cdot a(x)dpdxd\tau\nonumber\\
	&\qquad +\sum_{i,j=1}^3\int_0^t\iint_{\mathbb{T}^3\times\R^3}\frac{\mathfrak{c}}{p^0}(|p|^2-\beta_a)\sqrt{J_{\mathfrak{c}}(p)}p_ip_j\partial_{ij}\phi_a\cdot (\FI-\mathbf{P}_{\mathfrak{c}})f_{\mathfrak{c}} dpdxd\tau\nonumber\\
	&\geq \{5+O(\fc^{-2})\}\int_0^t\|a\|_{L^2}^2d\tau-C\int_0^t\|a\|_{L^2}\cdot\|(\FI-\mathbf{P}_{\mathfrak{c}})f_{\mathfrak{c}}\|_{L^2}d\tau, 
\end{align}
where the $b$ contribution vanishes due to the oddness in $p$ and the $c$ contribution vanishes because of the choice of $\beta_a$.
The second term on the RHS of \eqref{3.23} is
\begin{align*}
	&-\int_0^t\iint_{\mathbb{T}^3\times\R^3}\partial_{\tau} \psi_a f_{\mathfrak{c}}dpdxd\tau\nonumber\\
	&=\sum_{i=1}^3\int_0^t\iint_{\mathbb{T}^3\times\R^3}(|p|^2-\beta_a)p_i\sqrt{J_{\mathfrak{c}}(p)}\partial_{\tau} \partial_i\phi_a f_{\mathfrak{c}}dpdxd\tau\nonumber\\
	&=\sum_{i=1}^3\int_0^t\iint_{\mathbb{T}^3\times\R^3}(|p|^2-\beta_a)p_i^2J_{\mathfrak{c}}(p)\partial_{\tau} \partial_i\phi_a\cdot b_i(x)dpdxd\tau\nonumber\\
	&\quad +\sum_{i=1}^3\int_0^t\iint_{\mathbb{T}^3\times\R^3}(|p|^2-\beta_a)p_i\sqrt{J_{\mathfrak{c}}(p)}\partial_{\tau} \partial_i\phi_a\cdot(\FI-\mathbf{P}_{\mathfrak{c}})f_{\mathfrak{c}}dpdxd\tau,
\end{align*}
where 
\begin{align*}
	\int_{\R^3}(|p|^2-\beta_a)p_i^2J_{\mathfrak{c}}(p)dp=A_{10}-\beta_aA_1=\f{A_3(A_1A_6-A_4A_{10})}{A_1-A_3A_4}=-5+O(\fc^{-2}),
\end{align*}
which, together with \eqref{3.31}, yields that
\begin{align}\label{3.60}
	\Big|\int_0^t\iint_{\mathbb{T}^3\times\R^3}\partial_{\tau} \psi_a f_{\mathfrak{c}}dpdxd\tau\Big|&\lesssim\int_0^t\{\|b\|_{L^2}+\|(\FI-\mathbf{P}_{\mathfrak{c}})f_{\mathfrak{c}}\|_{L^2}\}\cdot\{\|b\|_{L^2}+\|(\FI-\mathbf{P}_{\mathfrak{c}})f_{\mathfrak{c}}\|_{L^2}\}d\tau\nonumber\\
	&\lesssim\int_0^t\{\|b\|_{L^2}^2+\|(\FI-\mathbf{P}_{\mathfrak{c}})f_{\mathfrak{c}}\|_{L^2}^2\}d\tau.
\end{align}
The RHS of \eqref{3.23} is now bounded uniformly by 
\begin{align}\label{3.61}
	|G_a(t)-G_a(0)|+\int_0^t	\|a(\tau)\|_{L^2}\left\{\|(\mathbf{I}-\mathbf{P}_{\mathfrak{c}}) f_{\mathfrak{c}}(\tau)\|_{L^2}+\|\mathcal{S}(\tau)\|_{L^2}\right\}d\tau.
\end{align}

Combining \eqref{3.23}, \eqref{3.58} and \eqref{3.60}--\eqref{3.61}, we get 
\begin{align}\label{3.62}
	\int_0^t\|a(\tau)\|_{L^2}^2d\tau\leq |G_a(t)-G_a(0)|+C\int_0^t\{\|(\FI-\mathbf{P}_{\mathfrak{c}})f_{\mathfrak{c}}\|_{L^2}^2+\|\mathcal{S}\|_{L^2}^2+\|b\|_{L^2}^2\}d\tau.
\end{align}
Using \eqref{3.43}, \eqref{3.55} and \eqref{3.62}, we can first choose $\v$ suitably small, and then choose $\mathfrak{c}$ large enough. Thus one can first close the estimate of $c$, then $b$, and finally $a$. Hence \eqref{4.28} follows. Therefore the proof of Proposition \ref{prop3.2} is completed.
\end{proof}

\begin{Lemma}\label{lem4.5}
Assume \eqref{3.21-0} and $(M_{\mathfrak{c}},\mathbf{J}_{\mathfrak{c}},E_{\mathfrak{c}})=(0,\mathbf{0},0)$. Let $f_{\mathfrak{c}}(t)$ be the solution of the linearized relativistic Boltzmann equation  \eqref{4.6}. Then there exists a constant $\lambda_{2}>0$ such that for any $t \ge 0$,
	\begin{align*} 
		&\|f_{\mathfrak{c}}(t)\|_{L^{2}}^{2} \le Ce^{-\lambda_2 t}\left\{\left\|f_{0,\mathfrak{c}}\right\|_{L^{2}}^{2}+ \int_0^t e^{\lambda_2 s} \|\mathcal{S}(s)\|^2_{L^{2}}ds\right\}.
	\end{align*}
    \end{Lemma}
\begin{proof}
	Let $\mathbf{f}_{\mathfrak{c}}(t):=e^{\lambda t}f_{\mathfrak{c}}(t)$, $\lambda>0$, then
	\begin{align}\label{2.55}
		\begin{cases}
			\partial_t\mathbf{f}_{\mathfrak{c}}+\hat{p}\cdot\nabla_x\mathbf{f}_{\mathfrak{c}}+\mathbf{L}_{\mathfrak{c}} \mathbf{f}_{\mathfrak{c}}=\lambda \mathbf{f}_{\mathfrak{c}}+e^{\lambda t}\mathcal{S},\\
			\mathbf{f}_{\mathfrak{c}}(t,x,p)|_{t=0}=f_{0,\mathfrak{c}}.
		\end{cases}
	\end{align}
Multiplying both sides of the equation $\eqref{2.55}_1$ by $\mathbf{f}_{\mathfrak{c}}$ and integrating over $\mathbb{T}^3\times\R^3\times[0,t]$, one obtains
\begin{align*}
	\f12\|\mathbf{f}_{\mathfrak{c}}(t)\|_{L^2}^2+\int_0^t\langle \mathbf{f}_{\mathfrak{c}}, \mathbf{L}_{\mathfrak{c}} \mathbf{f}_{\mathfrak{c}}\rangle ds=\f12\|\mathbf{f}_{\mathfrak{c}}(0)\|_{L^2}^2+\lambda \int_0^t\|\mathbf{f}_{\mathfrak{c}}(s)\|_{L^2}^2ds+\int_0^t\iint_{\mathbb{T}^3\times \mathbb{R}^3}e^{\lambda s}\mathcal{S}(s)\mathbf{f}_{\mathfrak{c}}(s)dpdxds.
\end{align*}

Noting that 
\begin{align*}
	\int_0^t\langle \mathbf{f}_{\mathfrak{c}}, \mathbf{L}_{\mathfrak{c}} \mathbf{f}_{\mathfrak{c}}\rangle ds\geq \zeta_0\int_0^t\|(\FI-\mathbf{P}_{\mathfrak{c}})\mathbf{f}_{\mathfrak{c}}\|_{\nu_{\mathfrak{c}}}^2ds,
\end{align*}
we have
\begin{align}\label{2.56}
	\|\mathbf{f}_{\mathfrak{c}}(t)\|_{L^2}^2+2\zeta_0\int_0^t\|(\FI-\mathbf{P}_{\mathfrak{c}})\mathbf{f}_{\mathfrak{c}}\|_{\nu_{\mathfrak{c}}}^2 ds\leq\|\mathbf{f}_{\mathfrak{c}}(0)\|_{L^2}^2+C\lambda \int_0^t\|\mathbf{f}_{\mathfrak{c}}(s)\|_{L^2}^2ds+C_\lambda\int_0^te^{2\lambda s}\|\mathcal{S}(s)\|_{L^2}^2ds.
\end{align}

It is clear that
\begin{align*}
	\iint_{\mathbb{T}^3\times \mathbb{R}^3}(\lambda \mathbf{f}_{\mathfrak{c}}+e^{\lambda t}\mathcal{S})\sqrt{J_{\mathfrak{c}}(p)}
	\begin{pmatrix}
		1  \\ p \\ p^0
	\end{pmatrix}
	dpdx=0.
\end{align*}
Applying Proposition \ref{prop3.2} to \eqref{2.55}, one can get
\begin{align}\label{2.57}
	\int_0^t\|\mathbf{P}_{\mathfrak{c}} \mathbf{f}_{\mathfrak{c}}(s)\|_{\nu_{\mathfrak{c}}}^2ds&\leq C|G(t)-G(0)|+C\int_0^t\|(\FI-\mathbf{P}_{\mathfrak{c}})\mathbf{f}_{\mathfrak{c}}(s)\|_{\nu_{\mathfrak{c}}}^2ds\nonumber\\
	&\quad+C\int_0^te^{2\lambda s}\|\mathcal{S}(s)\|_{L^2}^2ds+C\lambda^2\int_0^t\|\mathbf{f}_{\mathfrak{c}}(s)\|_{L^2}^2ds,
\end{align}
where $G(t)\lesssim \|\mathbf{f}_{\mathfrak{c}}(t)\|_{L^2}^2$. Combining \eqref{2.56}--\eqref{2.57} and taking $\lambda$ suitably small, we have
\begin{align*}
	\|\mathbf{f}_{\mathfrak{c}}(t)\|_{L^2}^2+\int_0^t\|\mathbf{f}_{\mathfrak{c}}\|_{\nu_{\mathfrak{c}}}^2ds\leq C\|f_{0,\mathfrak{c}}\|_{L^2}^2+C\int_0^te^{2\lambda s}\|\mathcal{S}(s)\|_{L^2}^2ds,
\end{align*} 
which yields that
\begin{align*}
	\|f_{\mathfrak{c}}(t)\|_{L^2}^2\leq Ce^{-2\lambda t}\Big\{\|f_{0,\mathfrak{c}}\|_{L^2}^2+\int_0^t e^{2\lambda s}\|\mathcal{S}(s)\|_{L^2}^2ds\Big\}.
\end{align*}
Take $\lambda_2:=2\lambda$ and we complete the proof of Lemma \ref{lem4.5}.
\end{proof}

\begin{Proposition}\label{prop4.6}
Let $\beta>4$,  \eqref{3.21-0} and   \eqref{1.8-10} hold with $(M_{\mathfrak{c}},\mathbf{J}_{\mathfrak{c}},E_{\mathfrak{c}})=(0,\mathbf{0},0)$. Let $f_{\mathfrak{c}}(t)$ be the solution of  linearized relativistic Boltzmann equation \eqref{4.6} in $t\in [0,T]$.
%	Then  and \eqref{1.8-10}--\eqref{1.10-10}, \eqref{3.21-0} hold with $(M_{\mathfrak{c}},\mathbf{J}_{\mathfrak{c}},E_{\mathfrak{c}})=(0,\mathbf{0},0)$. 
%	Assume that
%	\begin{align*}
%		\left\|w_{\beta} f_{0,\mathfrak{c}}\right\|_{L^{\infty}}+\sup _{s\ge 0} \big\{e^{\lambda_{0} s}\left\|\nu_{\mathfrak{c}}^{-1} w_{\beta} \mathcal{S}(s)\right\|_{L^{\infty}}\big\}<\infty,
%	\end{align*}
%	where $\lambda_{0}>0$ is a small constant to be chosen in the proof. Then the linear initial value problem \eqref{4.6} admits a unique solution $f_{\mathfrak{c}}(t, x, p)$ satisfying
	Then it holds that
	\begin{align}\label{3.69}
		\sup _{0 \le s \le t} \big\{e^{\lambda_{0} s}\|w_{\beta} f_{\mathfrak{c}}(s)\|_{L^{\infty}}\big\} \le \hat{C}\left\|w_{\beta} f_{0,\mathfrak{c}}\right\|_{L^{\infty}}+\hat{C} \sup _{0 \le s \le t} \big\{e^{\lambda_{0} s}\left\|\nu_{\mathfrak{c}}^{-1} w_{\beta} \mathcal{S}(s)\right\|_{L^{\infty}}\big\}
	\end{align}
	for any $t \in [0,T]$, where $\hat{C}\geq1$ is independent of $\mathfrak{c}$.
    \end{Proposition}
    \begin{proof}
%	The local existence and uniqueness of solutions to the linear inhomogeneous problem \eqref{4.6}
%	can be proved by similar arguments as in \cite[Proposition 6.2]{DRWZ} and \cite[ Theorem 4.1]{DW}. We omit
%	the details for brevity.
	
	 Define $\lambda_0=\min\{\f{\nu_0}{4},\f{\lambda_2}{4}\}$ and take $\lambda_1=\lambda_0$. Using Lemma \ref{lem4.5}, one has
	\begin{align*}
		\|f_{\mathfrak{c}}(t)\|_{L^2}^2\leq& Ce^{-4\lambda_0t}\|w_{\beta}f_{0,\mathfrak{c}}\|_{L^\infty}^2+C\int_0^te^{-4\lambda_0(t-s)}\|\nu_{\mathfrak{c}}^{-1}w_{\beta}\mathcal{S}(s)\|_{L^\infty}^2ds\\
		\leq & Ce^{-4\lambda_0t}\|w_{\beta}f_{0,\mathfrak{c}}\|_{L^\infty}^2+Ce^{-2\lambda_0t}\sup\limits_{0\leq s\leq t}\|e^{\lambda_0s}\nu_{\mathfrak{c}}^{-1}w_{\beta}\mathcal{S}(s)\|_{L^\infty}^2,
	\end{align*}
which implies that
\begin{align*}
	e^{\lambda_0t}\|f_{\mathfrak{c}}(t)\|_{L^2}\leq C\|w_{\beta}f_{0,\mathfrak{c}}\|_{L^\infty}+C\sup\limits_{0\leq s\leq t}\|e^{\lambda_0s}\nu_{\mathfrak{c}}^{-1}w_{\beta}\mathcal{S}(s)\|_{L^\infty}.
\end{align*}
Recall  \eqref{4.10},  we have
\begin{align}\label{3.70}
	e^{\lambda_0t}|\mathfrak{h}_{\mathfrak{c}}(t,x,p)|&\leq C\|\mathfrak{h}_{0,\mathfrak{c}}\|_{L^\infty}+C\sup\limits_{0\leq s \leq t}\|e^{\lambda_0 s}\nu_{\mathfrak{c}}^{-1}w_{\beta}\mathcal{S}(s)\|_{L^\infty}+C_N\sup\limits_{0\leq s\leq t}\|e^{\lambda_0 s}f_{\mathfrak{c}}(s)\|_{L^2}\nonumber\\
	&\quad +C\max\Big\{\f{1}{\fc},\f{1}{N}\Big\}\sup\limits_{0\leq s\leq t}\|e^{\lambda_0 s}\mathfrak{h}_{\mathfrak{c}}(s)\|_{L^\infty}\nonumber\\
	&\leq C_N\|\mathfrak{h}_{0,\mathfrak{c}}\|_{L^\infty}+C_N\sup\limits_{0\leq s\leq t}\|e^{\lambda_0 s}\nu_{\mathfrak{c}}^{-1}w_{\beta}\mathcal{S}(s)\|_{L^\infty}\nonumber\\
	&\quad+C\max\Big\{\f{1}{\fc},\f{1}{N}\Big\}\sup\limits_{0\leq s \leq t}\|e^{\lambda_0 s}\mathfrak{h}_{\mathfrak{c}}(s)\|_{L^\infty}.
\end{align}
Take $N$ and $\mathfrak{c}$ suitably large such that $C\max\Big\{\f{1}{\fc},\f{1}{N}\Big\}\le \frac{1}{2}$, then we can conclude \eqref{3.69} from \eqref{3.70}. Therefore the proof of Proposition \ref{prop4.6} is completed.
\end{proof}	

\subsection{Local existence}
%%%%%%%%%%%%%%%%%%%%%%%%%%%%%%%%%%%%%%%%%%%%%%%%%%%%%%%%%%%%%%%

The local existence of \eqref{1.3-0}--\eqref{1.4-0} has been proved in \cite[Theorem 3.1]{Wang} for $\mathfrak{c}=1$. In the present paper, we need to obtain  the local solution with uniform-in-$\mathfrak{c}$ estimates.
\begin{Proposition}[Local existence]\label{prop3.5}
	Let $\beta>4$, $F_{0,\mathfrak{c}}(x, p)=J_{\mathfrak{c}}(p)+\sqrt{J_{\mathfrak{c}}(p)} f_{0,\mathfrak{c}}(x, p) \geq 0$, and $\left\|w_{\beta} f_{0,\mathfrak{c}}\right\|_{L^{\infty}}<\infty$. Then there exists a positive time
	$$
	t_1:=\left(8 C_1\left[1+\left\|w_{\beta} f_{0,\mathfrak{c}}\right\|_{L^{\infty}}\right]\right)^{-1}>0,
	$$
	such that the relativistic Boltzmann equation \eqref{1.3-0}, \eqref{1.4-0} has a unique mild solution $F_{\mathfrak{c}}(t, x, p)=J_{\mathfrak{c}}(p)+\sqrt{J_{\mathfrak{c}}(p)} f_{\mathfrak{c}}(t, x, p) \geq 0$ in the time interval $t \in\left[0, t_1\right]$ and satisfies
	$$
	\left\|w_{\beta} f_{\mathfrak{c}}(t)\right\|_{L^{\infty}} \leq 2\left\|w_{\beta} f_{0,\mathfrak{c}}\right\|_{L^{\infty}},\quad 0 \leq t \leq t_1,
	$$
	where the positive constant $C_1 \geq 1$ depends only on $\beta$ and is independent of $\mathfrak{c}$. 
	%	In addition, the conservations of mass, momentum, and energy (1.9)-(1.11) as well as the additional entropy inequality (1.12) hold. 
\end{Proposition}

\begin{proof}
	It follows from Lemma \ref{lem2.7} that 
	\begin{align}\label{3.57-40}
		\|w_{\beta}\mathbf{\Gamma}_{\mathfrak{c}}^{+}\left(f_{1},f_{2}\right)\|_{L^{\infty}}\le C\|w_{\beta}f_{1}\|_{L^{\infty}}\|w_{\beta}f_{2}\|_{L^{\infty}},
	\end{align}
	where the constant $C$ is independent of $\mathfrak{c}$. Then by similar arguments as in \cite[Theorem 3.1]{Wang}, one can obtain the local existence for the problem \eqref{1.3-0}--\eqref{1.4-0}. We omit the details here for brevity of presentation. Therefore the proof of Proposition \ref{prop3.5} is completed.
\end{proof}

%%%%%%%%%%%%%%%%%%%%%%%%%%%%%%%%%%%%%%%%%%%%%%%%%%%%%%%%%%%%%%%
\subsection{Proof of Theorem \ref{thm1.1}} 
%%%%%%%%%%%%%%%%%%%%%%%%%%%%%%%%%%%%%%%%%%%%%%%%%%%%%%%%%%%%%%%

Let $f_{\mathfrak{c}}$ be the solution of \eqref{5.8}--\eqref{5.9} in $[0,t]$. Applying Proposition \ref{prop4.6}, we have
\begin{align}\label{3.88}
	\sup _{0 \le s \le t} \big\{e^{\lambda_{0} s}\|w_{\beta} f_{\mathfrak{c}}(s)\|_{L^{\infty}}\big\} \le \hat{C}\left\|w_{\beta} f_{0,\mathfrak{c}}\right\|_{L^{\infty}}+\hat{C}\sup _{0 \le s \le t} \big\{e^{\lambda_{0} s}\left\|\nu_{\mathfrak{c}}^{-1} w_{\beta} \mathbf{\Gamma}_{\mathfrak{c}}(f_{\mathfrak{c}},f_{\mathfrak{c}})(s)\right\|_{L^{\infty}}\big\}.
\end{align}
Noting Proposition \ref{prop3.5}, we make the {\it a prior} assumption 
\begin{align*}
\sup\limits_{0\leq s\leq t}\|w_{\beta}f_{\mathfrak{c}}(s)\|_{L^\infty}\leq 2\hat{C}\|w_{\beta}f_{0,\mathfrak{c}}\|_{L^\infty}.
\end{align*}
It follows from \eqref{3.88} and \eqref{2.17-10} that
\begin{align*}
	\sup_{0 \le s \le t} \big\{e^{\lambda_{0} s}\|w_{\beta} f_{\mathfrak{c}}(s)\|_{L^{\infty}}\big\} 
	&\le \hat{C}\left\|w_{\beta} f_{0,\mathfrak{c}}\right\|_{L^{\infty}}+ C \hat{C} \sup _{0 \le s \le t} \big\{e^{\lambda_{0} s}\left\|w_{\beta}f_{\mathfrak{c}}(s)\right\|_{L^{\infty}}^2\big\}\nonumber\\
	&\leq \hat{C}\left\|w_{\beta} f_{0,\mathfrak{c}}\right\|_{L^{\infty}} + 2C\hat{C}^2\left\|w_{\beta} f_{0,\mathfrak{c}}\right\|_{L^{\infty}}\cdot \sup _{0 \le s \le t} \big\{e^{\lambda_{0} s}\left\|w_{\beta}f_{\mathfrak{c}}(s)\right\|_{L^{\infty}}\big\}.
\end{align*}
Taking $\left\|w_{\beta} f_{0,\mathfrak{c}}\right\|_{L^{\infty}}$ small so that $2C\hat{C}^2\left\|w_{\beta} f_{0,\mathfrak{c}}\right\|_{L^{\infty}}\le \frac{1}{4}$, then we obtain
\begin{align*}
	\sup_{0 \le s \le t} \big\{e^{\lambda_{0} s}\|w_{\beta} f_{\mathfrak{c}}(s)\|_{L^{\infty}}\big\}\le \frac{4}{3}\hat{C}\|w_{\beta} f_{0,\mathfrak{c}}\|_{L^\infty}.
\end{align*}
Thus we have closed the {\it a priori} assumption. The global existence follows from the standard continuity argument and the estimate \eqref{1.31} follows. Therefore the proof of Theorem \ref{thm1.1} is completed.
$\hfill\Box$

%%%%%%%%%%%%%%%%%%%%%%%%%%%%%%%%%%%%%%%%%%%%%%%%%%%%%%%%%%%%%%%
\section{Global-in-time Newtonian limit in $L^1_pL^{\infty}_{x}$}
%%%%%%%%%%%%%%%%%%%%%%%%%%%%%%%%%%%%%%%%%%%%%%%%%%%%%%%%%%%%%%%

In this section, we aim to prove Theorem \ref{thm1.3}, i.e., the Newtonian limit in $L^1_pL^{\infty}_{x}$. Recall \eqref{2.4-10}, the problem \eqref{1.14-0}, \eqref{1.17-0} admits a global mild solution $F(t,x,p)=\mu(p)+\sqrt{\mu(p)}f(t,x,p)\geq 0$ with
\begin{align}\label{4.4-0}
	f(t,x,p)&=e^{-\nu(p)t}f_0(x-pt,p)+\int_0^te^{-\nu(p)(t-s)}(\mathbf{K}f)(s,x-p(t-s),p)ds\nonumber\\
	&\qquad +\int_0^te^{-\nu(p)(t-s)}\mathbf{\Gamma}(f,f)(s,x-p(t-s),p)ds,
\end{align}
where $f_0(x,p)=\frac{F_0(x,p)-\mu(p)}{\sqrt{\mu(p)}}$. It is clear that  \eqref{1.6}  is equivalent to 
\begin{align}\label{4.3-0}
	\|\sqrt{\mu}(\tau^x_hf_0-f_0)\|_{L^1_pL^{\infty}_x}+\|\sqrt{\mu}(\tau^p_hf_0-f_0)\|_{L^1_pL^{\infty}_x}\le \tilde{C}_3|h|,\quad |h|<1.
\end{align}

%%%%%%%%%%%%%%%%%%%%%%%%%%%%%%%%%%%%%%%%%%%%%%%%%%%%%%%%%%%%%%%
\subsection{Continuity of the mild solution of  Newtonian Boltzmann equation} 
%%%%%%%%%%%%%%%%%%%%%%%%%%%%%%%%%%%%%%%%%%%%%%%%%%%%%%%%%%%%%%%

Motivated by \cite{Strain}, we first present a key lemma which states the continuity of the mild form of the Newtonian Boltzmann equation. Such a continuity of Newtonian Boltzmann equation will play an important role in the proof of Newtonian limit of relativistic Boltzmann equation.

\begin{Lemma}\label{lem4.1}
	Let the assumptions on Newtonian Boltzmann equation in Theorem \ref{thm1.3} hold. For any $0<\varepsilon<1$, there is a constant $C=C(\varepsilon)>0$ such that for any $0 \leq t\le  (-\log h)^{\alpha_0}$ with $0< \alpha_0<\frac{1}{2}$, we have
	\begin{align*}
		\sup_{|u| \leq|h|}\left(\left\|\sqrt{\mu}(\tau_u^x f(t)-f(t))\right\|_{L_p^1 L_x^{\infty}}+\left\|\sqrt{\mu}(\tau_u^p f(t)-f(t))\right\|_{L_p^1 L_x^{\infty}}\right) \leq C(\varepsilon)|h|^{1-\varepsilon}, \quad|h|<1 .
	\end{align*}
\end{Lemma}
\begin{proof} 
	We split the proof into two steps.\\
	\noindent{\textit{Step 1: Estimate on the $x$ difference.}}   
	It follows that \eqref{4.4-0} that
	\begin{align*}
		\tau^x_hf(t,x,p)&=e^{-\nu(p)t}f_0(x+h-pt,p)+\int_0^t e^{-\nu(p)(t-s)}(\mathbf{K}f)(s,x+h-p(t-s),p)ds\nonumber\\
		&\quad +\int_0^te^{-\nu(p)(t-s)}\mathbf{\Gamma}(f,f)(s,x+h-p(t-s),p)ds.
	\end{align*}
    Denote $y:=x-p(t-s)$ and $y_h:=x+h-p(t-s)$, then we have
    \begin{align}\label{4.6-0}
    	&\sqrt{\mu(p)}\big[\tau^x_hf(t,x,p)-f(t,x,p)\big]\nonumber\\
    	&=\sqrt{\mu(p)}e^{-\nu(p)t}\Big[f_0(x+h-pt,p)-f_0(x-pt,p)\Big]\nonumber\\
    	&\quad +\int_0^t\sqrt{\mu(p)}e^{-\nu(p)(t-s)}\Big[(\mathbf{K}f)(s,y_h,p)-(\mathbf{K}f)(s,y,p)\Big]ds\nonumber\\
    	&\quad +\int_0^t\sqrt{\mu(p)}e^{-\nu(p)(t-s)}\Big[\mathbf{\Gamma}(f,f)(s,y_h,p)-\mathbf{\Gamma}(f,f)(s,y,p)\Big]ds\nonumber\\
    	&:=\sum_{i=1}^3\mathcal{J}_i.
    \end{align}

	Using the translation invariance of the $L^{\infty}_x$ norm, we have from \eqref{4.3-0} that
	\begin{align}\label{4.7-0}
		\Big\|\mathcal{J}_1\Big\|_{L^1_pL^{\infty}_x}\le \left\|\sqrt{\mu}(\tau_h^x f_0-f_0)\right\|_{L_p^1 L_x^{\infty}}\lesssim |h|.
	\end{align}
	
	For $\mathcal{J}_2$, one has
	\begin{align*} 
		\mathcal{J}_2&=\int_0^tds\int_{\mathbb{R}^3\times \mathbb{S}^2}\mathcal{K}_{\infty}(p,q,\omega)e^{-\nu(p)(t-s)}\sqrt{\mu(p)\mu(q)\mu(\bar{q}')}\Big[f(s,y_h,\bar{p}')-f(s,y,\bar{p}')\Big]d\omega dq\nonumber\\
		&\quad +\int_0^tds\int_{\mathbb{R}^3\times \mathbb{S}^2}\mathcal{K}_{\infty}(p,q,\omega)e^{-\nu(p)(t-s)}\sqrt{\mu(p)\mu(q)\mu(\bar{p}')}\Big[f(s,y_h,\bar{q}')-f(s,y,\bar{q}')\Big]d\omega dq\nonumber\\
		&\quad -\int_0^tds\int_{\mathbb{R}^3\times \mathbb{S}^2}\mathcal{K}_{\infty}(p,q,\omega)e^{-\nu(p)(t-s)}\mu(p)\sqrt{\mu(q)}\Big[f(s,y_h,q)-f(s,y,q)\Big]d\omega dq\nonumber\\
		&:=\mathcal{J}_{21}+\mathcal{J}_{22}-\mathcal{J}_{23}.
	\end{align*}
	We fix a suitably large $R>0$ and split the integrals into an unbounded region $|p|+|q| \geq R$ and a bounded region $|p|+|q|<R$.

%	We take the $L^{\infty}$-norm in $x$ and integrate over $p$. Also, we will always bound the term $e^{-\nu(p)(t-s)}$ by $e^{-\nu_0(t-s)}$ with some positive constant $\nu_0>0$.  
	On the unbounded region, using the uniform bound $\|f\|_{L^{\infty}_{x,p}}\le \hat{C}_2<\infty$, for $i=1,2,3$, one has
	\begin{align*}
		\|\mathcal{J}_{2i}\|_{L_p^1(|p|+|q|\ge R) L_x^{\infty}}&\lesssim \int_0^t e^{-\nu_0(t-s)} d s \int_{\mathbb{R}^3 \times \mathbb{R}^3 \times \mathbb{S}^{2}}  \mathcal{K}_{\infty}(p,q,\omega) \mathbf{1}_{|p|+|q| \geq R} \sqrt{\mu(p)\mu(q)} d \omega d q d p\nonumber\\
		&\lesssim  e^{-\frac{R^2}{16}}.
	\end{align*}
	
	On the bounded region, in view of $\mathcal{K}_{\infty}(p,q,\omega)=|\omega\cdot (p-q)|\le |p|+|q|<R$, we have 
	\begin{align*}
		&\|\mathcal{J}_{21}\|_{L_p^1(|p|+|q|< R) L_x^{\infty}}\nonumber\\
		&\leq \int_0^t e^{-\nu_0(t-s)}d s \int_{\mathbb{R}^3 \times \mathbb{R}^3 \times \mathbb{S}^{2}}  \mathbf{1}_{|p|+|q|<R} \mathcal{K}_{\infty}(p,q,\omega)\nonumber\\
		&\qquad \times \sqrt{\mu(p)\mu(q)}\left\|\tau_h^x f\left(\bar{p}^{\prime}\right)-f\left(\bar{p}^{\prime}\right)\right\|_{L_x^{\infty}}d \omega d p d q \nonumber\\
		&\leq R\int_0^t e^{-\nu_0(t-s)}d s \int_{\mathbb{R}^3 \times \mathbb{R}^3 \times \mathbb{S}^{2}}  \sqrt{\mu(\bar{p}')\mu(\bar{q}')}\left\|\tau_h^x f\left(\bar{p}^{\prime}\right)-f\left(\bar{p}^{\prime}\right)\right\|_{L_x^{\infty}}d \omega d p d q \nonumber\\
		&= R\int_0^t e^{-\nu_0(t-s)}d s \int_{\mathbb{R}^3 \times \mathbb{R}^3 \times \mathbb{S}^{2}}  \sqrt{\mu(\bar{p}')\mu(\bar{q}')}\left\|\tau_h^x f\left(\bar{p}'\right)-f\left(\bar{p}'\right)\right\|_{L_x^{\infty}}d \omega d \bar{p}' d \bar{q}'  \nonumber\\
		&\lesssim R \int_0^t e^{-\nu_0(t-s)} \left\|\sqrt{\mu}(\tau_h^x f-f)(s)\right\|_{L_p^1 L_x^{\infty}}ds.
	\end{align*}
	We can bound $\|\mathcal{J}_{22}\|_{L_p^1(|p|+|q|< R) L_x^{\infty}}$ and $\|\mathcal{J}_{23}\|_{L_p^1(|p|+|q|< R) L_x^{\infty}}$ similarly to get the same bound. Hence we have
	\begin{align}\label{4.11}
		\|\mathcal{J}_{2}\|_{L_p^1 L_x^{\infty}}&\lesssim e^{-\frac{R^2}{16}}+R \int_0^t e^{-\nu_0(t-s)} \left\|\sqrt{\mu}(\tau_h^x f-f)(s)\right\|_{L_p^1 L_x^{\infty}}ds.
	\end{align}
	
	For $\mathcal{J}_{3}$, a direct calculation shows that
	\begin{align}\label{4.12}
		\mathcal{J}_{3}&=\int_0^tds\int_{\mathbb{R}^3\times \mathbb{S}^2}\mathcal{K}_{\infty}(p,q,\omega)e^{-\nu(p)(t-s)}\sqrt{\mu(p)\mu(q)}\Big[f(s,y_h,\bar{p}')-f(s,y,\bar{p}')\Big]f(s,y_h,\bar{q}')d\omega dq\nonumber\\
		&\quad +\int_0^tds\int_{\mathbb{R}^3\times \mathbb{S}^2}\mathcal{K}_{\infty}(p,q,\omega)e^{-\nu(p)(t-s)}\sqrt{\mu(p)\mu(q)}\Big[f(s,y_h,\bar{q}')-f(s,y,\bar{q}')\Big]f(s,y,\bar{p}')d\omega dq\nonumber\\
		&\quad -\int_0^tds\int_{\mathbb{R}^3\times \mathbb{S}^2}\mathcal{K}_{\infty}(p,q,\omega)e^{-\nu(p)(t-s)}\sqrt{\mu(p)\mu(q)}\Big[f(s,y_h,p)-f(s,y,p)\Big]f(s,y_h,q)d\omega dq\nonumber\\
		&\quad -\int_0^tds\int_{\mathbb{R}^3\times \mathbb{S}^2}\mathcal{K}_{\infty}(p,q,\omega)e^{-\nu(p)(t-s)}\sqrt{\mu(p)\mu(q)}\Big[f(s,y_h,q)-f(s,y,q)\Big]f(s,y,p)d\omega dq\nonumber\\
		&:=\mathcal{J}_{31}+\mathcal{J}_{32}-\mathcal{J}_{33}-\mathcal{J}_{34}.
	\end{align}
	On the unbounded region,  noting $\|f\|_{L^{\infty}_{x,p}}\le \hat{C}_2<\infty$, one has
	\begin{align}\label{4.13}
		\sum_{i=1}^4\|\mathcal{J}_{3i}\|_{L_p^1(|p|+|q|\ge R) L_x^{\infty}}\lesssim e^{-\frac{R^2}{16}}.
	\end{align}
	On the bounded region, we have
	\begin{align}\label{4.14}
		\|\mathcal{J}_{31}\|_{L_p^1(|p|+|q|< R) L_x^{\infty}}&\lesssim \int_0^t e^{-\nu_0(t-s)}d s \int_{\mathbb{R}^3 \times \mathbb{R}^3 \times \mathbb{S}^{2}}  \mathbf{1}_{|p|+|q|<R} \mathcal{K}_{\infty}(p,q,\omega)\nonumber\\
		&\qquad \times \sqrt{\mu(p)\mu(q)}\left\|\tau_h^x f\left(\bar{p}^{\prime}\right)-f\left(\bar{p}^{\prime}\right)\right\|_{L_x^{\infty}}d \omega d q d p \nonumber\\
		&\lesssim R\int_0^t e^{-\nu_0(t-s)}d s \int_{\mathbb{R}^3 \times \mathbb{R}^3 \times \mathbb{S}^{2}}  \sqrt{\mu(\bar{p}')\mu(\bar{q}')}\left\|\tau_h^x f\left(\bar{p}^{\prime}\right)-f\left(\bar{p}^{\prime}\right)\right\|_{L_x^{\infty}}d \omega d \bar{p}' d \bar{q}' \nonumber\\
		&\lesssim R \int_0^t e^{-\nu_0(t-s)} \left\|\sqrt{\mu}(\tau_h^x f-f)(s)\right\|_{L_p^1 L_x^{\infty}}ds.
	\end{align}
	Similarly, we can obtain
	\begin{align}\label{4.15}
		\sum_{i=2}^4\|\mathcal{J}_{3i}\|_{L_p^1(|p|+|q|< R) L_x^{\infty}}&\lesssim R \int_0^t e^{-\nu_0(t-s)} \left\|\sqrt{\mu}(\tau_h^x f-f)(s)\right\|_{L_p^1 L_x^{\infty}}ds.
	\end{align}
	Combining \eqref{4.12}--\eqref{4.15}, one has
	\begin{align}\label{4.16}
		\|\mathcal{J}_{3}\|_{L_p^1 L_x^{\infty}}&\lesssim e^{-\frac{R^2}{16}}+R \int_0^t e^{-\nu_0(t-s)} \left\|\sqrt{\mu}(\tau_h^x f-f)(s)\right\|_{L_p^1 L_x^{\infty}}ds.
	\end{align}
	It follows from \eqref{4.6-0}--\eqref{4.7-0}, \eqref{4.11} and \eqref{4.16} that 
	\begin{align*}
		\left\|\sqrt{\mu}(\tau_h^x f-f)(t)\right\|_{L_p^1 L_x^{\infty}}&\le C|h|+C e^{-\frac{R^2}{16}}+ CR \int_0^t e^{-\nu_0(t-s)} \left\|\sqrt{\mu}(\tau_h^x f-f)(s)\right\|_{L_p^1 L_x^{\infty}}ds.
	\end{align*}

	Now choose $R=(-\log |h|)^{\alpha_1}$ with $\alpha_1=\frac{3}{4}-\frac{\alpha_0}{2}$, which yields that $\frac{1}{2}<\alpha_1<1-\alpha_0$. Using Lemma \ref{lem2.5} with $z=\frac{1}{|h|}$ and $|h|<1$, we have
	\begin{align*}
		\left\|\sqrt{\mu}(\tau_h^x f-f)(t)\right\|_{L_p^1 L_x^{\infty}}&\le C|h|+ CR \int_0^t e^{-\nu_0(t-s)} \left\|\sqrt{\mu}(\tau_h^x f-f)(s)\right\|_{L_p^1 L_x^{\infty}}ds.
	\end{align*}
	By Gr\"{o}nwall type inequality, we have
	\begin{align}\label{4.11-20}
		\left\|\sqrt{\mu}(\tau_h^x f-f)(t)\right\|_{L_p^1 L_x^{\infty}}&\le C|h|\Big( \frac{CR}{CR-\nu_0}e^{(CR-\nu_0)t}-\frac{\nu_0}{CR-\nu_0}\Big)\le C|h|e^{CRt}.
	\end{align}
	For any $0\le t\le (-\log |h|)^{\alpha_0}$, using Lemma \ref{lem2.5}, one has
	\begin{align*}
		e^{CRt}\le e^{C(-\log |h|)^{\alpha_0+\alpha_1}}\lesssim |h|^{-\epsilon},
	\end{align*}
	which, together with \eqref{4.11-20}, yields that
	\begin{align}\label{4.17-0}
		\left\|\sqrt{\mu}(\tau_h^x f-f)(t)\right\|_{L_p^1 L_x^{\infty}}&\le C(\epsilon) |h|^{1-\epsilon}.
	\end{align}

	{\textit{Step 2: Estimate on the $p$ difference.}}  
	It follows from \eqref{4.4-0} that
	\begin{align}\label{4.19-0}
		\tau^p_hf(t,x,p)&=e^{-\nu(p+h)t}f_0(x-(p+h)t,p+h)\nonumber\\
		&\quad +\int_0^t e^{-\nu(p+h)(t-s)}(\mathbf{K}f)(s,x-(p+h)(t-s),p+h)ds \nonumber\\
		&\quad +\int_0^te^{-\nu(p+h)(t-s)}\mathbf{\Gamma}(f,f)(s,x-(p+h)(t-s),p+h)ds.
	\end{align}
	Denote $\tilde{y}_h:=x-(p+h)(t-s)$ and recall that $y=x-p(t-s)$. Using \eqref{4.19-0} and \eqref{4.4-0}, one has
	\begin{align*} 
		&\sqrt{\mu(p)}\big[\tau^p_hf(t,x,p)-f(t,x,p)\big]\nonumber\\
		&=\sqrt{\mu(p)}\Big[e^{-\nu(p+h)t}f_0(x-(p+h)t,p+h)-e^{-\nu(p)t}f_0(x-pt,p)\Big]\nonumber\\
		&\quad +\int_0^t \sqrt{\mu(p)}e^{-\nu(p+h)(t-s)}(\mathbf{K}f)(s,\tilde{y}_h,p+h)ds-\int_0^t \sqrt{\mu(p)}e^{-\nu(p)(t-s)}(\mathbf{K}f)(s,y,p)ds\nonumber\\
		&\quad +\int_0^t \sqrt{\mu(p)}e^{-\nu(p+h)(t-s)}\mathbf{\Gamma}(f,f)(s,\tilde{y}_h,p+h)ds-\int_0^t \sqrt{\mu(p)}e^{-\nu(p)(t-s)}\mathbf{\Gamma}(f,f)(s,y,p)ds\nonumber\\
		&:=\sum_{i=1}^3\mathcal{B}_i.
	\end{align*}
	
	For pre-collisional velocities $(p+h,q)$, we denote the post-collisional velocities as 
	\begin{align*}
		\bar{p}'_h=p+h+(\omega\cdot(q-p-h))\omega,\quad \bar{q}'_h=q-(\omega\cdot(q-p-h))\omega.
	\end{align*}
	We first estimate $\mathcal{B}_1$. It holds that
	\begin{align}\label{4.20-1}
		\mathcal{B}_1&=\sqrt{\mu(p)}\Big[e^{-\nu(p+h)t}-e^{-\nu(p)t}\Big]f_0(x-(p+h)t,p+h)\nonumber\\
		&\quad +\sqrt{\mu(p)}e^{-\nu(p)t}\Big[f_0(x-(p+h)t,p+h)-f_0(x-(p+h)t,p)\Big]\nonumber\\
		&\quad +\sqrt{\mu(p)}e^{-\nu(p)t}\Big[f_0(x-(p+h)t,p)-f_0(x-pt,p)\Big]\nonumber\\
		&:=\mathcal{B}_{11}+\mathcal{B}_{12}+\mathcal{B}_{13}.
	\end{align}
	For $\mathcal{B}_{11}$, noting
\begin{align}\label{4.15-001}
	|\mathcal{K}_{\infty}(p+h,q,\omega)-\mathcal{K}_{\infty}(p,q,\omega)|\le |\omega\cdot (p+h-q)-\omega\cdot (p-q)|\le |h|,
\end{align}
one gets
\begin{align*}
	|\nu(p+h)-\nu(p)|=& \Big|\int_{\mathbb{R}^3} \int_{\mathbb{S}^2} \Big(\mathcal{K}_{\infty}(p+h,q,\omega)-\mathcal{K}_{\infty}(p,q,\omega)\Big)\mu(q) d \omega d q\Big|\nonumber\\
	&\le \int_{\mathbb{R}^3}\int_{\mathbb{S}^2} |h| \mu(q) d \omega d q\lesssim |h|.
\end{align*}
Then one has
\begin{align}\label{4.21-0}
	\|\mathcal{B}_{11}\|_{L^1_pL^{\infty}_x}&\le \int_{\mathbb{R}^3}\sqrt{\mu(p)} |e^{-\nu(p+h)t}-e^{-\nu(p)t} |^{1-\epsilon}\cdot  |e^{-\nu(p+h)t}-e^{-\nu(p)t} |^{\epsilon}\cdot \|f_0\|_{L^{\infty}_{x,p}}dp\nonumber\\
	&\lesssim  \int_{\mathbb{R}^3}\sqrt{\mu(p)}|\nu(p+h)-\nu(p)|^{1-\epsilon}\cdot t^{1-\epsilon}\cdot e^{-\epsilon \nu_0 t}dp\lesssim  |h|^{1-\epsilon}.
\end{align}
It follows from \eqref{4.3-0} that
\begin{align}\label{4.22-0}
	\|\mathcal{B}_{12}\|_{L^1_pL^{\infty}_x}\lesssim \|\sqrt{\mu}(\tau^p_h f_0-f_0)\|_{L^1_pL^{\infty}_x}\lesssim |h|
\end{align} 
and
\begin{align}\label{4.23-00}
	\|\mathcal{B}_{13}\|_{L^1_pL^{\infty}_x}&\lesssim e^{-\nu_0 t}(t+1)\sup_{|u| \leq|h|}\|\sqrt{\mu}(\tau^x_u f_0-f_0)\|_{L^1_pL^{\infty}_x}\nonumber\\
	&\lesssim \sup_{|u| \leq|h|}\|\sqrt{\mu}(\tau^x_u f_0-f_0)\|_{L^1_pL^{\infty}_x}\lesssim |h|.
\end{align}
Combining \eqref{4.20-1}, \eqref{4.21-0}--\eqref{4.23-00}, we obtain that
\begin{align}\label{4.23-0}
	\|\mathcal{B}_{1}\|_{L^1_pL^{\infty}_x}\lesssim |h|^{1-\epsilon}.
\end{align}

For $\mathcal{B}_2$, it holds that
\begin{align*} 
	%		\int_0^t \sqrt{\mu(p)}\Big[e^{-\nu(p+h)(t-s)}-e^{-\nu(p)(t-s)}\Big](\mathbf{K}f)(s,\tilde{y}_h,p+h)ds\nonumber\\
	%		& \quad +\int_0^t \sqrt{\mu(p)}e^{-\nu(p)(t-s)}\Big[(\mathbf{K}f)(s,\tilde{y}_h,p+h)-(\mathbf{K}f)(s,y,p)\Big]ds\nonumber\\
	\mathcal{B}_2&=\int_0^t \sqrt{\mu(p)}\Big[e^{-\nu(p+h)(t-s)}-e^{-\nu(p)(t-s)}\Big](\mathbf{K}f)(s,\tilde{y}_h,p+h)ds\nonumber\\
	&\quad +\int_0^tds \int_{\mathbb{R}^3\times \mathbb{S}^2} \sqrt{\mu(p)\mu(q)}e^{-\nu(p)(t-s)}\nonumber\\ 
	&\qquad \times  \Big[\mathcal{K}_{\infty}(p+h,q,\omega)\sqrt{\mu(\bar{q}_h')}f(s,\tilde{y}_h,\bar{p}_h')-\mathcal{K}_{\infty}(p,q,\omega)\sqrt{\mu(\bar{q}')}f(s, y,\bar{p}')\Big]d\omega dq\nonumber\\
	&\quad +\int_0^t ds \int_{\mathbb{R}^3\times \mathbb{S}^2} \sqrt{\mu(p)\mu(q)}e^{-\nu(p)(t-s)}\nonumber\\
	&\qquad \times \Big[\mathcal{K}_{\infty}(p+h,q,\omega)\sqrt{\mu(\bar{p}_h')}f(s,\tilde{y}_h,\bar{q}_h')-\mathcal{K}_{\infty}(p,q,\omega)\sqrt{\mu(\bar{p}')}f(s,y,\bar{q}')\Big]d\omega dq\nonumber\\
	&\quad -\int_0^t ds \int_{\mathbb{R}^3\times \mathbb{S}^2} \sqrt{\mu(p)\mu(q)}e^{-\nu(p)(t-s)}\nonumber\\ 	
	&\qquad \times \Big[\mathcal{K}_{\infty}(p+h,q,\omega)\sqrt{\mu(p+h)}f(s,\tilde{y}_h,q)-\mathcal{K}_{\infty}(p,q,\omega)\sqrt{\mu(p)}f(s,y,q)\Big]d\omega dq\nonumber\\
	&:=\mathcal{B}_{21}+\mathcal{B}_{22}+\mathcal{B}_{23}-\mathcal{B}_{24}.
\end{align*}
Using the uniform bound $\|(\mathbf{K}f)(s)\|_{L^{\infty}_{x,p}}\lesssim 1$ and taking similar arguments as in \eqref{4.21-0}, we deduce that 
\begin{align}\label{4.20-000}
	\|\mathcal{B}_{21}\|_{L^1_pL^{\infty}_x}\lesssim |h|^{1-\epsilon}.
\end{align}
For other terms, on the unbounded region, $|p|+|q|\ge R$, it is clear that 
\begin{align}\label{4.21-000}
	\sum_{i=2}^4\|\mathcal{B}_{2i}\|_{L^1_p(|p|+|q|\ge R)L^{\infty}_x}\lesssim e^{-\frac{R^2}{16}}.
\end{align}
Now we consider the bounded region $|p|+|q|<R$. One can expand $\mathcal{B}_{22}$ as 
\begin{align}\label{4.27-0}
	\mathcal{B}_{22}&=\int_0^t ds \int_{\mathbb{R}^3\times \mathbb{S}^2}\sqrt{\mu(p)\mu(q)}e^{-\nu(p)(t-s)} \Big[\mathcal{K}_{\infty}(p+h,q,\omega)-\mathcal{K}_{\infty}(p,q,\omega)\Big]\nonumber\\
	&\qquad \times \sqrt{\mu(\bar{q}_h')}f(s,\tilde{y}_h,\bar{p}_h')d\omega dq\nonumber\\
	&\quad +\int_0^tds \int_{\mathbb{R}^3\times \mathbb{S}^2} \sqrt{\mu(p)\mu(q)}e^{-\nu(p)(t-s)}\mathcal{K}_{\infty}(p,q,\omega)\Big[\sqrt{\mu(\bar{q}_h')}-\sqrt{\mu(\bar{q}')}\Big]f(s,\tilde{y}_h,\bar{p}_h')d\omega dq\nonumber\\
	&\quad +\int_0^t ds \int_{\mathbb{R}^3\times \mathbb{S}^2} \sqrt{\mu(p)\mu(q)}e^{-\nu(p)(t-s)}\mathcal{K}_{\infty}(p,q,\omega)\sqrt{\mu(\bar{q}')}\Big[f(s,\tilde{y}_h,\bar{p}_h')-f(s,\tilde{y}_h,\bar{p}')\Big]d\omega dq\nonumber\\
	&\quad +\int_0^t ds \int_{\mathbb{R}^3\times \mathbb{S}^2} \sqrt{\mu(p)\mu(q)}e^{-\nu(p)(t-s)}\mathcal{K}_{\infty}(p,q,\omega)\sqrt{\mu(\bar{q}')}\Big[f(s,\tilde{y}_h,\bar{p}')-f(s,y,\bar{p}')\Big]d\omega dq\nonumber\\
	&:=\sum_{i=1}^4 \mathcal{B}_{22i}.
\end{align}
It follows from \eqref{4.15-001} that
$\|\mathcal{B}_{221}\|_{L^1_pL^{\infty}_x}\lesssim |h|$.
Due to the fact that $|\sqrt{\mu(\bar{q}_h')}-\sqrt{\mu(\bar{q}')}|\lesssim |h|$, one has $\|\mathcal{B}_{222}\|_{L^1_pL^{\infty}_x}\lesssim |h|$. Using $|\bar{p}_h'-\bar{p}'|\le 2|h|$, we have
\begin{align*}
	\|\mathcal{B}_{223}\|_{L^1_p(|p|+|q|\le R)L^{\infty}_x}&\lesssim R\int_0^t e^{-\nu_0(t-s)}\sup_{|u| \leq 2|h|}\|\sqrt{\mu}(\tau^p_u f-f)(s)\|_{L^1_pL^{\infty}_x}ds\nonumber\\
	&\lesssim R\int_0^t e^{-\nu_0(t-s)}\sup_{|u| \leq |h|}\|\sqrt{\mu}(\tau^p_u f-f)(s)\|_{L^1_pL^{\infty}_x}ds+R|h|.
\end{align*}
The fourth term $\mathcal{B}_{224}$ can be estimated as
\begin{align*}
	\|\mathcal{B}_{224}\|_{L^1_p(|p|+|q|\le R)L^{\infty}_x}&\lesssim R\int_0^t e^{-\nu_0(t-s)}[(t-s)+1]\sup_{|u| \leq |h|}\|\sqrt{\mu}(\tau^x_u f-f)(s)\|_{L^1_pL^{\infty}_x}ds\lesssim R|h|^{1-\epsilon},
\end{align*}
where we have used \eqref{4.17-0} in the second inequality.
Hence we have obtained that 
\begin{align*} 
	\|\mathcal{B}_{22}\|_{L^1_p(|p|+|q|\le R)L^{\infty}_x}
	&\lesssim R|h|^{1-\epsilon}+ R\int_0^t e^{-\nu_0(t-s)}\sup_{|u| \leq |h|}\|\sqrt{\mu}(\tau^p_u f-f)(s)\|_{L^1_pL^{\infty}_x}ds.
\end{align*}
The estimates for 	$\|\mathcal{B}_{2i}\|_{L^1_p(|p|+|q|\le R)L^{\infty}_x}\ (i=3,4)$ are similar to the one of $\|\mathcal{B}_{22}\|_{L^1_p(|p|+|q|\le R)L^{\infty}_x}$ and they have the same bound. Thus one has
\begin{align}\label{4.23-000} 
	\sum_{i=2}^4\|\mathcal{B}_{2i}\|_{L^1_p(|p|+|q|\le R)L^{\infty}_x}
	&\lesssim R|h|^{1-\epsilon}+ R\int_0^t e^{-\nu_0(t-s)}\sup_{|u| \leq |h|}\|\sqrt{\mu}(\tau^p_u f-f)(s)\|_{L^1_pL^{\infty}_x}ds.
\end{align}
Combining \eqref{4.20-000}--\eqref{4.21-000} and \eqref{4.23-000}, one has
\begin{align*}
	\|\mathcal{B}_{2}\|_{L^1_pL^{\infty}_x}
	&\lesssim R|h|^{1-\epsilon}+e^{-\frac{R^2}{16}}+ R\int_0^t e^{-\nu_0(t-s)}\sup_{|u| \leq |h|}\|\sqrt{\mu}(\tau^p_u f-f)(s)\|_{L^1_pL^{\infty}_x}ds.
\end{align*}

Finally we estimate $\mathcal{B}_{3}$. It holds that
\begin{align*} 
	%	    &=\int_0^t \sqrt{\mu(p)}\Big[e^{-\nu(p+h)(t-s)}-e^{-\nu(p)(t-s)}\Big]\mathbf{\Gamma}(f,f)(s,\tilde{y}_h,p+h)ds\nonumber\\
	%		&\quad  +\int_0^t \sqrt{\mu(p)}e^{-\nu(p)(t-s)}\Big[\mathbf{\Gamma}(f,f)(s,\tilde{y}_h,p+h)-\mathbf{\Gamma}(f,f)(s,y,p)\Big]ds\nonumber\\
	\mathcal{B}_{3}&=\int_0^t \sqrt{\mu(p)}\Big[e^{-\nu(p+h)(t-s)}-e^{-\nu(p)(t-s)}\Big]\mathbf{\Gamma}(f,f)(s,\tilde{y}_h,p+h)ds\nonumber\\
	&\quad +\int_0^t  ds \int_{\mathbb{R}^3\times \mathbb{S}^2}\sqrt{\mu(p)\mu(q)}e^{-\nu(p)(t-s)}\nonumber\\
	&\qquad \times \Big[\mathcal{K}_{\infty}(p+h,q,\omega)f(s,\tilde{y}_h,\bar{p}_h')f(s,\tilde{y}_h,\bar{q}_h')-\mathcal{K}_{\infty}(p,q,\omega)f(s,y,\bar{p}')f(s,y,\bar{q}')\Big]d\omega dq\nonumber\\
	&\quad -\int_0^t  ds \int_{\mathbb{R}^3\times \mathbb{S}^2}\sqrt{\mu(p)\mu(q)}e^{-\nu(p)(t-s)}\nonumber\\ 
	&\qquad \times \Big[\mathcal{K}_{\infty}(p+h,q,\omega)f(s,\tilde{y}_h,p+h)f(s,\tilde{y}_h,q)-\mathcal{K}_{\infty}(p,q,\omega)f(s,y,p)f(s,y,q)\Big]d\omega dq\nonumber\\
	&:=\mathcal{B}_{31}+\mathcal{B}_{32}-\mathcal{B}_{33}.
\end{align*}
Using the uniform boundedness of $\|\mu^{\frac{1}{4}}\mathbf{\Gamma}(f,f)(s)\|_{L^{\infty}_{x,p}}$, by similar arguments as in \eqref{4.21-0}, we deduce that
\begin{align*}
	\|\mathcal{B}_{31}\|_{L^1_pL^{\infty}_x}\lesssim |h|^{1-\epsilon}.
\end{align*}
For other two terms, on the unbounded region, $|p|+|q|\ge R$, it is clear that
\begin{align*}
	\|\mathcal{B}_{32}\|_{L^1_p(|p|+|q|\ge R)L^{\infty}_x}+\|\mathcal{B}_{33}\|_{L^1_p(|p|+|q|\ge R)L^{\infty}_x}\lesssim e^{-\frac{R^2}{16}}.
\end{align*}
Now we consider the bounded region $|p|+|q|<R$. One can expand $\mathcal{B}_{32}$ as 
\begin{align}\label{2.26-1}
	\mathcal{B}_{32}&=\int_0^t ds \int_{\mathbb{R}^3\times \mathbb{S}^2}\sqrt{\mu(p)\mu(q)}e^{-\nu(p)(t-s)}\nonumber\\
	&\qquad \quad \times \Big[\mathcal{K}_{\infty}(p+h,q,\omega)-\mathcal{K}_{\infty}(p,q,\omega)\Big]f(s,\tilde{y}_h,\bar{p}_h')f(s,\tilde{y}_h,\bar{q}_h')d\omega dq\nonumber\\
	&\quad +\int_0^t ds \int_{\mathbb{R}^3\times \mathbb{S}^2}\mathcal{K}_{\infty}(p,q,\omega)\sqrt{\mu(p)\mu(q)}e^{-\nu(p)(t-s)} \nonumber\\
	&\qquad \quad \times \Big[f(s,\tilde{y}_h,\bar{p}_h')-f(s,\tilde{y}_h,\bar{p}')\Big]f(s,\tilde{y}_h,\bar{q}_h')d\omega dq\nonumber\\
	&\quad +\int_0^t ds \int_{\mathbb{R}^3\times \mathbb{S}^2}\mathcal{K}_{\infty}(p,q,\omega)\sqrt{\mu(p)\mu(q)}e^{-\nu(p)(t-s)} \nonumber\\
	&\qquad \quad \times \Big[f(s,\tilde{y}_h,\bar{q}_h')-f(s,\tilde{y}_h,\bar{q}')\Big]f(s,\tilde{y}_h,\bar{p}')d\omega dq\nonumber\\
	&\quad +\int_0^t ds \int_{\mathbb{R}^3\times \mathbb{S}^2}\mathcal{K}_{\infty}(p,q,\omega)\sqrt{\mu(p)\mu(q)}e^{-\nu(p)(t-s)} \nonumber\\
	&\qquad \quad \times \Big[f(s,\tilde{y}_h,\bar{q}')-f(s,y,\bar{q}')\Big]f(s,\tilde{y}_h,\bar{p}')d\omega dq\nonumber\\
	&\quad +\int_0^t ds \int_{\mathbb{R}^3\times \mathbb{S}^2}\mathcal{K}_{\infty}(p,q,\omega)\sqrt{\mu(p)\mu(q)}e^{-\nu(p)(t-s)} \nonumber\\
	&\qquad \quad \times \Big[f(s,\tilde{y}_h,\bar{p}')-f(s,y,\bar{p}')\Big]f(s,y,\bar{q}')d\omega dq.
\end{align}

It is noted that the estimate of each item on the RHS of \eqref{2.26-1} is similar to that of one item in \eqref{4.27-0} with the same bound. Hence we conclude that 
\begin{align*}
	\sup_{|u| \leq |h|}\left\|\sqrt{\mu}(\tau_u^p f-f)(t)\right\|_{L_p^1 L_x^{\infty}} &\le CR|h|^{1-\epsilon}+C e^{-\frac{R^2}{16}}\nonumber\\
	&\quad + CR \int_0^t e^{-\nu_0(t-s)} \sup_{|u| \leq |h|}\left\|\sqrt{\mu}(\tau_u^p f-f)(s)\right\|_{L_p^1 L_x^{\infty}}ds.
\end{align*}
By similar arguments as in \eqref{4.11-20}--\eqref{4.17-0}, for any $0\le t\le (-\log |h|)^{\alpha_0}$, one has
\begin{align}\label{2.25}
	\sup_{|u| \leq |h|}\left\|\sqrt{\mu}(\tau_u^p f-f)(t)\right\|_{L_p^1 L_x^{\infty}} &\le C(\epsilon)|h|^{1-2\epsilon}.
\end{align}
Combining \eqref{4.17-0} and \eqref{2.25} and taking $\varepsilon=2\epsilon$, we complete the proof of Lemma \ref{lem4.1}.
\end{proof}	

\subsection{Proof of Theorems \ref{thm1.3}}
Recall that 
\begin{align*}
F_{\mathfrak{c}}(t,x,p)=J_{\mathfrak{c}}(p)+\sqrt{J_{\mathfrak{c}}(p)}f_{\mathfrak{c}}(t,x,p),\quad F(t,x,p)=\mu(p)+\sqrt{\mu(p)}f(t,x,p),
\end{align*}
which yields that
\begin{align*}%\label{2.26-0}
F_{\mathfrak{c}}(t,x,p)-F(t,x,p)&=(J_{\mathfrak{c}}(p)-\mu(p))+f(t,x,p)\big(\sqrt{J_{\mathfrak{c}}(p)}-\sqrt{\mu(p)}\big)\nonumber\\
&\quad +\sqrt{J_{\mathfrak{c}}(p)}\big(f_{\mathfrak{c}}(t,x,p)-f(t,x,p)\big).
\end{align*}
It follows from Lemma \ref{lem2.6} and the uniform boundedness of $\|f\|_{L^{\infty}_{x,p}}$ that 
\begin{align*}%\label{2.26}
\|J_{\mathfrak{c}}-\mu\|_{L^1_pL^{\infty}_x}+\|f(t)(\sqrt{J_{\mathfrak{c}}}-\sqrt{\mu})\|_{L^1_pL^{\infty}_x}\lesssim \frac{1}{\mathfrak{c}^{2-\epsilon}}.
\end{align*}

Noting also that 
\begin{align*}
F_{0,\mathfrak{c}}(x,p)-F_0(x,p)&=J_{\mathfrak{c}}(p)-\mu(p)+f_0(x,p)\big(\sqrt{J_{\mathfrak{c}}(p)}-\sqrt{\mu(p)}\big)\nonumber\\
&\quad +\sqrt{J_{\mathfrak{c}}(p)}(f_{0,\mathfrak{c}}(x,p)-f_0(x,p)),
\end{align*}
by \eqref{1.5} and Lemma \ref{lem2.6}, we have 
\begin{align}\label{4.31-10}
\|\sqrt{J_{\mathfrak{c}}}(f_{0,\mathfrak{c}}-f_0)\|_{L^1_pL^{\infty}_x}&\le \|F_{0,\mathfrak{c}}-F_0\|_{L^1_pL^{\infty}_x}+\|J_{\mathfrak{c}}-\mu\|_{L^1_pL^{\infty}_x}+\|f_{0}(\sqrt{J_{\mathfrak{c}}}-\sqrt{\mu})\|_{L^1_pL^{\infty}_x}\nonumber\\
&\lesssim \frac{1}{\mathfrak{c}^{k-\epsilon}}.
\end{align}
In the following, we mainly focus on the estimate for $\|\sqrt{J_{\mathfrak{c}}}(f_{\mathfrak{c}}-f)\|_{L^1_pL^{\infty}_x}$. 

By \eqref{5.8}--\eqref{5.9}, the problem \eqref{1.3-0}--\eqref{1.4-0} admits a mild solution $F_{\mathfrak{c}}(t,x,p)=J_{\mathfrak{c}}(p)+\sqrt{J_{\mathfrak{c}}(p)}f_{\mathfrak{c}}(t,x,p)$ with
\begin{align}\label{2.27}
f_{\mathfrak{c}}(t,x,p)&=e^{-\nu_{\mathfrak{c}}(p)t}f_{0,\mathfrak{c}}(x-\hat{p}t,p)+\int_0^te^{-\nu_{\mathfrak{c}}(p)(t-s)}(\mathbf{K}_{\mathfrak{c}}f_{\mathfrak{c}})(s,x-\hat{p}(t-s),p)ds\nonumber\\
&\qquad +\int_0^te^{-\nu_{\mathfrak{c}}(p)(t-s)}\mathbf{\Gamma}_{\mathfrak{c}}(f_{\mathfrak{c}},f_{\mathfrak{c}})(s,x-\hat{p}(t-s),p)ds.
\end{align}
Denote $y:=x-p(t-s)$ and $y_{\mathfrak{c}}:=x-\hat{p}(t-s)$. It follows from \eqref{2.27} and \eqref{4.4-0} that
\begin{align}\label{4.28-000}
&\sqrt{J_{\mathfrak{c}}(p)}\big(f_{\mathfrak{c}}(t,x,p)-f(t,x,p)\big)\nonumber\\
&=\sqrt{J_{\mathfrak{c}}(p)}\Big[e^{-\nu_{\mathfrak{c}}(p)t}f_{0,\mathfrak{c}}(x-\hat{p}t,p)-e^{-\nu(p)t}f_0(x-pt,p)\Big]\nonumber\\
&\quad +\sqrt{J_{\mathfrak{c}}(p)}\Big[\int_0^t e^{-\nu_{\mathfrak{c}}(p)(t-s)}(\mathbf{K}_{\mathfrak{c}}f_{\mathfrak{c}})(s,y_{\mathfrak{c}},p)ds 
-\int_0^t  e^{-\nu(p)(t-s)}(\mathbf{K}f)(s,y,p)ds\Big]\nonumber\\
&\quad +\sqrt{J_{\mathfrak{c}}(p)}\Big[\int_0^t e^{-\nu_{\mathfrak{c}}(p)(t-s)}\mathbf{\Gamma}_{\mathfrak{c}}(f_{\mathfrak{c}}, f_{\mathfrak{c}})(s,y_{\mathfrak{c}},p)ds 
-\int_0^t e^{-\nu(p)(t-s)}\mathbf{\Gamma}(f,f)(s,y,p)ds\Big]\nonumber\\
&:=\mathcal{L}_1+\mathcal{L}_2+\mathcal{L}_3.
\end{align}

For $\mathcal{L}_1$, it holds that
\begin{align}\label{2.29}
\mathcal{L}_1&=\sqrt{J_{\mathfrak{c}}(p)}(e^{-\nu_{\mathfrak{c}}(p)t}-e^{-\nu(p)t})f_{0,\mathfrak{c}}(x-\hat{p}t,p)\nonumber\\
&\quad +
\sqrt{J_{\mathfrak{c}}(p)}e^{-\nu(p)t}(f_{0,\mathfrak{c}}(x-\hat{p}t,p)-f_0(x-\hat{p}t,p))\nonumber\\
&\quad +\sqrt{J_{\mathfrak{c}}(p)}e^{-\nu(p)t}(f_{0}(x-\hat{p}t,p)-f_0(x-pt,p)):=\mathcal{L}_{11}+\mathcal{L}_{12}+\mathcal{L}_{13}.
\end{align}
For $\mathcal{L}_{11}$, using Lemmas \ref{lem2.4}, \ref{lem2.6}, one has 
\begin{align*}
|\nu_{\mathfrak{c}}(p)-\nu(p)|&=\Big|\iint_{\mathbb{R}^3 \times \mathbb{S}^2}\left[\mathcal{K}_{\mathfrak{c}}(p, q, \omega) J_{\mathfrak{c}}(q)-\mathcal{K}_{\infty}(p, q, \omega) \mu(q)\right] d\omega d q\Big|\nonumber\\
& \le \iint_{\mathbb{R}^3 \times \mathbb{S}^2}|\mathcal{K}_{\mathfrak{c}}(p, q, \omega)-\mathcal{K}_{\infty}(p, q, \omega)| J_{\mathfrak{c}}(q) d\omega d q \nonumber\\
&\quad +\iint_{\mathbb{R}^3 \times \mathbb{S}^2}|J_{\mathfrak{c}}(q)-\mu(q)| \mathcal{K}_{\infty}(p, q, \omega)d\omega d q \nonumber\\
& \lesssim \frac{(1+|p|)^6}{\mathfrak{c}^{2-\epsilon}},
\end{align*}
which, together with similar arguments as in \eqref{4.21-0}, yields that
\begin{align} 
\|\mathcal{L}_{11}\|_{L^1_pL^{\infty}_x}\lesssim \frac{1}{\mathfrak{c}^{2-2\epsilon}}.
\end{align}
For $\mathcal{L}_{12}$, by the translation invariance of the $L^{\infty}$ norm and \eqref{4.31-10}, it holds that
\begin{align}
\|\mathcal{L}_{12}\|_{L^1_pL^{\infty}_x}\le \|\sqrt{J_{\mathfrak{c}}}(f_{0,\mathfrak{c}}-f_0)\|_{L^1_pL^{\infty}_x}\lesssim \frac{1}{\mathfrak{c}^{k-\epsilon}}.
\end{align}
For $\mathcal{L}_{13}$, note that 
$$f_{0}(x-\hat{p}t,p)-f_0(x-pt,p)=f_{0}(x-pt+u,p)-f_0(x-pt,p),$$ 
where $u=(p-\hat{p})t$. By Lemma \ref{lem2.4}, one has $|u|\le \frac{|p|^3}{2\mathfrak{c}^2}t$. When $|p| < r(\mathfrak{c}):=(\log \mathfrak{c})^{\alpha_1}$ with $\frac{1}{2}<\alpha_1<1$, we have $|u| \leq \frac{Ct}{\mathfrak{c}^{2-\epsilon}}$. It follows from Lemma \ref{lem2.6} and \eqref{4.3-0} that
\begin{align}\label{2.31}
\|\mathcal{L}_{13}\|_{L^1_p(|p|< r(\mathfrak{c}))L^{\infty}_x}
&\le e^{-\nu_0t}\sup_{|u| \leq \frac{Ct}{\mathfrak{c}^{2-\epsilon}}}\|(\sqrt{J_{\mathfrak{c}}}-\sqrt{\mu})(\tau^x_uf_0-f_0)\|_{L^1_pL^{\infty}_x}\nonumber\\
&\qquad +e^{-\nu_0t}\sup_{|u| \leq \frac{Ct}{\mathfrak{c}^{2-\epsilon}}}\|\sqrt{\mu}(\tau^x_uf_0-f_0)\|_{L^1_pL^{\infty}_x}\nonumber\\
&\lesssim \frac{1}{\mathfrak{c}^{2-\epsilon}}+ e^{-\nu_0t}\frac{t}{\mathfrak{c}^{2-\epsilon}}\lesssim \frac{1}{\mathfrak{c}^{2-\epsilon}}.
\end{align}
On the other hand, if $|p| \geq r(\mathfrak{c})$, using the second part of Lemma \ref{lem2.5}, we have
\begin{align}\label{2.32}
\|\mathcal{L}_{13}\|_{L^1_p(|p|\ge r(\mathfrak{c}))L^{\infty}_x}
&\le 2\|(\sqrt{J_{\mathfrak{c}}}-\sqrt{\mu})f_{0}\|_{L^1_p(|p|\ge r(\mathfrak{c}))L^{\infty}_x}+2\|\sqrt{\mu}f_{0}\|_{L^1_p(|p|\ge r(\mathfrak{c}))L^{\infty}_x}\nonumber\\
&\lesssim \frac{1}{\mathfrak{c}^{2-\epsilon}}.
\end{align}
Using \eqref{2.29}--\eqref{2.32}, we have
\begin{align}\label{2.33}
\|\mathcal{L}_{1}\|_{L^1_pL^{\infty}_x}\lesssim \frac{1}{\mathfrak{c}^{k-\epsilon}} + \frac{1}{\mathfrak{c}^{2-2\epsilon}}\lesssim \frac{1}{\mathfrak{c}^{k-2\epsilon}}. 
\end{align}

Next we estimate $\mathcal{L}_2$. It holds that
\begin{align}\label{2.34}
%		&=\int_0^t  \sqrt{J_{\mathfrak{c}}(p)}e^{-\nu_{\mathfrak{c}}(p)(t-s)}\Big[(\mathbf{K}_{\mathfrak{c}}f_{\mathfrak{c}})(s,y_{\mathfrak{c}},p)-(\mathbf{K}f)(s,y,p)\Big]ds\nonumber\\
%		& \qquad +\int_0^t \sqrt{J_{\mathfrak{c}}(p)}\Big[e^{-\nu_{\mathfrak{c}}(p)(t-s)}-e^{-\nu(p)(t-s)}\Big](\mathbf{K}f)(s,y,p)ds\nonumber\\
\mathcal{L}_2&=\int_0^t \sqrt{J_{\mathfrak{c}}(p)}\Big[e^{-\nu_{\mathfrak{c}}(p)(t-s)}-e^{-\nu(p)(t-s)}\Big](\mathbf{K}f)(s,y,p)ds\nonumber\\
&\quad +\int_0^tds \int_{\mathbb{R}^3\times \mathbb{S}^2} \sqrt{J_{\mathfrak{c}}(p)}e^{-\nu_{\mathfrak{c}}(p)(t-s)}\nonumber\\ 
&\qquad \times  \Big[\mathcal{K}_{\mathfrak{c}}(p,q,\omega)\sqrt{J_{\mathfrak{c}}(q)J_{\mathfrak{c}}(q')}f_{\mathfrak{c}}(s,y_{\mathfrak{c}},p')-\mathcal{K}_{\infty}(p,q,\omega)\sqrt{\mu(q)\mu(\bar{q}')}f(s,y,\bar{p}')\Big]d\omega dq\nonumber\\
&\quad +\int_0^tds \int_{\mathbb{R}^3\times \mathbb{S}^2} \sqrt{J_{\mathfrak{c}}(p)}e^{-\nu_{\mathfrak{c}}(p)(t-s)}\nonumber\\ 
&\qquad \times  \Big[\mathcal{K}_{\mathfrak{c}}(p,q,\omega)\sqrt{J_{\mathfrak{c}}(q)J_{\mathfrak{c}}(p')}f_{\mathfrak{c}}(s,y_{\mathfrak{c}},q')-\mathcal{K}_{\infty}(p,q,\omega)\sqrt{\mu(q)\mu(\bar{p}')}f(s,y,\bar{q}')\Big]d\omega dq\nonumber\\
&\quad -\int_0^tds \int_{\mathbb{R}^3\times \mathbb{S}^2} \sqrt{J_{\mathfrak{c}}(p)}e^{-\nu_{\mathfrak{c}}(p)(t-s)}\nonumber\\ 
&\qquad \times  \Big[\mathcal{K}_{\mathfrak{c}}(p,q,\omega)\sqrt{J_{\mathfrak{c}}(p)J_{\mathfrak{c}}(q)}f_{\mathfrak{c}}(s,y_{\mathfrak{c}},q)-\mathcal{K}_{\infty}(p,q,\omega)\sqrt{\mu(p)\mu(q)}f(s,y,q)\Big]d\omega dq\nonumber\\
&:=\sum_{i=1}^4 \mathcal{L}_{2i}.
\end{align}
For $\mathcal{L}_{21}$, using the uniform bound $\|(\mathbf{K}f)(s)\|_{L^{\infty}_{x,p}}\lesssim 1$ and taking similar arguments as in \eqref{4.21-0}, one obtains 
\begin{align}
\|\mathcal{L}_{21}\|_{L^1_pL^{\infty}_x}\lesssim \frac{1}{\mathfrak{c}^{2-2\epsilon}}.
\end{align}
For $\mathcal{L}_{22}$, we have 
\begin{align}\label{4.45}
\mathcal{L}_{22}&=\int_0^tds \int_{\mathbb{R}^3\times \mathbb{S}^2} \sqrt{J_{\mathfrak{c}}(p)}e^{-\nu_{\mathfrak{c}}(p)(t-s)}\mathcal{K}_{\mathfrak{c}}(p,q,\omega)\Big[\sqrt{J_{\mathfrak{c}}(q)J_{\mathfrak{c}}(q')}-\sqrt{\mu(q)\mu(\bar{q}')}\Big]f_{\mathfrak{c}}(s,y_{\mathfrak{c}},p')d\omega dq\nonumber\\
&\quad +\int_0^tds \int_{\mathbb{R}^3\times \mathbb{S}^2} \sqrt{J_{\mathfrak{c}}(p)}e^{-\nu_{\mathfrak{c}}(p)(t-s)}\mathcal{K}_{\mathfrak{c}}(p,q,\omega)\sqrt{\mu(q)\mu(\bar{q}')}(f_{\mathfrak{c}}(s,y_{\mathfrak{c}},p')-f(s,y_{\mathfrak{c}},p'))d\omega dq\nonumber\\
&\quad +\int_0^tds \int_{\mathbb{R}^3\times \mathbb{S}^2} \sqrt{J_{\mathfrak{c}}(p)}e^{-\nu_{\mathfrak{c}}(p)(t-s)}\mathcal{K}_{\mathfrak{c}}(p,q,\omega)\sqrt{\mu(q)\mu(\bar{q}')}(f(s,y_{\mathfrak{c}},p')-f(s,y,p'))d\omega dq\nonumber\\
&\quad +\int_0^tds \int_{\mathbb{R}^3\times \mathbb{S}^2} \sqrt{J_{\mathfrak{c}}(p)}e^{-\nu_{\mathfrak{c}}(p)(t-s)}\mathcal{K}_{\mathfrak{c}}(p,q,\omega)\sqrt{\mu(q)\mu(\bar{q}')}(f(s,y,p')-f(s,y,\bar{p}'))d\omega dq\nonumber\\
&\quad +\int_0^tds \int_{\mathbb{R}^3\times \mathbb{S}^2} \sqrt{J_{\mathfrak{c}}(p)}e^{-\nu_{\mathfrak{c}}(p)(t-s)}(\mathcal{K}_{\mathfrak{c}}-\mathcal{K}_{\infty})(p,q,\omega)\sqrt{\mu(q)\mu(\bar{q}')}f(s,y,\bar{p}')d\omega dq\nonumber\\
&:=\sum_{i=1}^5 \mathcal{L}_{22i}.
\end{align}

By Lemmas \ref{lem2.4} and \ref{lem2.6}, it holds that
\begin{align}\label{4.46}
\|\mathcal{L}_{221}\|_{L^1_pL^{\infty}_x}+\|\mathcal{L}_{225}\|_{L^1_pL^{\infty}_x}\lesssim \frac{1}{\mathfrak{c}^{2-\epsilon}}.
\end{align}
On the unbounded region $|p|+|q|\ge r(\mathfrak{c})$, using $\sqrt{J_{\mathfrak{c}}(p)}=\big(\sqrt{J_{\mathfrak{c}}(p)}-\sqrt{\mu(p)}\big)+\sqrt{\mu(p)}$ and Lemma \ref{lem2.6}, one has
\begin{align}\label{4.47}
\sum_{i=2}^4\|\mathcal{L}_{22i}\|_{L^1_p(|p|+|q|\ge r(\mathfrak{c}))L^{\infty}_x}\lesssim \frac{1}{\mathfrak{c}^{2-\epsilon}}+e^{-\frac{r^2(\mathfrak{c})}{16}}\lesssim \frac{1}{\mathfrak{c}^{2-\epsilon}}.
\end{align}
Then we need only to consider the bounded region $|p|+|q|< r(\mathfrak{c})$. Using Lemma \ref{lem2.4}, one has
\begin{align*}
|p'|\le |p'-\bar{p}'|+|\bar{p}'|\lesssim \frac{(|p|+|q|)^3}{\mathfrak{c}^2}+|p|+|p-q|\lesssim \frac{r^3(\mathfrak{c})}{\mathfrak{c}^2}+r(\mathfrak{c})\lesssim r(\mathfrak{c}).
\end{align*}
We can estimate $|q'|$ similarly. Thus we have
\begin{align*}
|p'|+|q'|\le C r(\mathfrak{c})
\end{align*}
for some positive constant $C>0$.
%    Using Lemma \ref{lem2.6} and Lemma \ref{lem2.5}, one has 
%	\begin{align}
%		\|\mathcal{Q}_1\|_{L^1_pL^{\infty}_x}+\|\mathcal{Q}_5\|_{L^1_pL^{\infty}_x}\le \frac{A}{\mathfrak{c}^{2}}.
%	\end{align}
Using Lemma \ref{lem2.4}, \eqref{2.30-10} and the fact that $\mu(q)\lesssim J_{\mathfrak{c}}(q)$, one obtains
\begin{align}\label{4.50}
&\|\mathcal{L}_{222}\|_{L^1_p(|p|+|q|< r(\mathfrak{c}))L^{\infty}_x}\nonumber\\
&\lesssim \int_0^tds \int_{\mathbb{R}^3\times \mathbb{R}^3\times \mathbb{S}^2} \sqrt{J_{\mathfrak{c}}(p)J_{\mathfrak{c}}(q)}e^{-\nu_{0}(t-s)}\mathcal{K}_{\mathfrak{c}}(p,q,\omega)\mathbf{1}_{|p|+|q|< r(\mathfrak{c})}\|(f_{\mathfrak{c}}-f)(s,\cdot,p')\|_{L^{\infty}_x}d\omega dpdq\nonumber\\
%		&\le \int_0^tds \int_{\mathbb{R}^3\times \mathbb{R}^3\times \mathbb{S}^2} \sqrt{\mu(p')\mu(q')}e^{-\nu_{0}(t-s)}\mathcal{K}_{\mathfrak{c}}(p,q,\omega)\mathbf{1}_{|p'|+|q'|<4R}\|(f_{\mathfrak{c}}(p')-f(p'))\|_{L^{\infty}_x}d\omega dpdq\nonumber\\
&\lesssim \int_0^tds \int_{\mathbb{R}^3\times \mathbb{R}^3\times \mathbb{S}^2} \sqrt{J_{\mathfrak{c}}(p')J_{\mathfrak{c}}(q')}e^{-\nu_{0}(t-s)}\mathcal{K}_{\mathfrak{c}}(p',q',\omega)\nonumber\\
&\qquad \qquad\times \mathbf{1}_{|p'|+|q'|\le Cr(\mathfrak{c})}\|(f_{\mathfrak{c}}-f)(s,\cdot,p')\|_{L^{\infty}_x}d\omega dp'dq'\nonumber\\
&\lesssim r(\mathfrak{c})\int_0^te^{-\nu_{0}(t-s)}\|\sqrt{J_{\mathfrak{c}}}(f_{\mathfrak{c}}-f)(s)\|_{L^1_pL^{\infty}_x}ds.
\end{align}
Using Lemma \ref{lem4.1}, one can bound $\|\mathcal{L}_{223}\|_{L^1_p(|p|+|q|< r(\mathfrak{c}))L^{\infty}_x}$ as 
\begin{align}\label{4.51}
\|\mathcal{L}_{223}\|_{L^1_p(|p|+|q|< r(\mathfrak{c}))L^{\infty}_x}&\lesssim \int_0^tds \int_{\mathbb{R}^3\times \mathbb{R}^3 \times \mathbb{S}^2} \sqrt{J_{\mathfrak{c}}(p)J_{\mathfrak{c}}(q)}e^{-\nu_{0}(t-s)}\mathcal{K}_{\mathfrak{c}}(p,q,\omega)\mathbf{1}_{|p|+|q|< r(\mathfrak{c})}\nonumber\\
&\qquad \qquad \times \sup_{|u| \leq \frac{(t-s)r^3(\mathfrak{c})}{2\mathfrak{c}^2}}\|\tau^x_u f(p')-f(p')\|_{L^{\infty}_x}d\omega dpdq\nonumber\\
&\lesssim \int_0^tds \int_{\mathbb{R}^3\times \mathbb{R}^3 \times \mathbb{S}^2} \sqrt{J_{\mathfrak{c}}(p')J_{\mathfrak{c}}(q')}e^{-\nu_{0}(t-s)}\mathcal{K}_{\mathfrak{c}}(p',q',\omega)\mathbf{1}_{|p'|+|q'|\le Cr(\mathfrak{c})}\nonumber\\
&\qquad \qquad \times \sup_{|u| \leq \frac{(t-s)r^3(\mathfrak{c})}{2\mathfrak{c}^2}}\|\tau^x_u f(p')-f(p')\|_{L^{\infty}_x}d\omega dp'dq'\nonumber\\
&\lesssim r(\mathfrak{c})\int_0^te^{-\nu_{0}(t-s)}\sup_{|u| \leq \frac{(t-s)r^3(\mathfrak{c})}{2\mathfrak{c}^2}}\|\sqrt{\mu}(\tau^x_u f-f)\|_{L^1_pL^{\infty}_x}ds+\frac{1}{\mathfrak{c}^{2-\epsilon}}\nonumber\\
&\lesssim r(\mathfrak{c})\int_0^t e^{-\nu_{0}(t-s)}\Big(\frac{(t-s)r^3(\mathfrak{c})}{2\mathfrak{c}^2}\Big)^{1-\epsilon}ds+\frac{1}{\mathfrak{c}^{2-\epsilon}}\lesssim \frac{1}{\mathfrak{c}^{2-3\epsilon}}.
\end{align}
By similar arguments as above, we obtain
\begin{align}\label{4.52}
\|\mathcal{L}_{224}\|_{L^1_p(|p|+|q|<r(\mathfrak{c}))L^{\infty}_x}	&\lesssim r(\mathfrak{c})\int_0^te^{-\nu_{0}(t-s)}\sup_{|u| \leq \frac{Cr^3(\mathfrak{c})}{\mathfrak{c}^2}}\|\sqrt{\mu}(\tau^p_u f-f)\|_{L^1_pL^{\infty}_x}ds+\frac{1}{\mathfrak{c}^{2-\epsilon}}\nonumber\\
&\lesssim \frac{1}{\mathfrak{c}^{2-3\epsilon}}.
\end{align}
Combining \eqref{4.45}--\eqref{4.52}, one has
\begin{align*}
\|\mathcal{L}_{22}\|_{L^1_pL^{\infty}_x}\lesssim r(\mathfrak{c})\int_0^te^{-\nu_{0}(t-s)}\|\sqrt{J_{\mathfrak{c}}}(f_{\mathfrak{c}}-f)(s)\|_{L^1_pL^{\infty}_x}ds+ \frac{1}{\mathfrak{c}^{2-3\epsilon}}.
\end{align*}

The estimates for $\|\mathcal{L}_{23}\|_{L^1_pL^{\infty}_x}$ and $\|\mathcal{L}_{24}\|_{L^1_pL^{\infty}_x}$ are very similar to the one for $\|\mathcal{L}_{22}\|_{L^1_pL^{\infty}_x}$ with the same bound. Thus we obtain 
\begin{align}\label{2.41}
\|\mathcal{L}_{2}\|_{L^1_pL^{\infty}_x}\lesssim r(\mathfrak{c})\int_0^te^{-\nu_{0}(t-s)}\|\sqrt{J_{\mathfrak{c}}}(f_{\mathfrak{c}}-f)(s)\|_{L^1_pL^{\infty}_x}ds+ \frac{1}{\mathfrak{c}^{2-3\epsilon}}.
\end{align}

Finally we estimate $\mathcal{L}_{3}$. It holds that
\begin{align*}
%		&=\int_0^t  \sqrt{J_{\mathfrak{c}}(p)}e^{-\nu_{\mathfrak{c}}(p)(t-s)}\Big[\mathbf{\Gamma}_{\mathfrak{c}}(f_{\mathfrak{c}}, f_{\mathfrak{c}})(s,y_{\mathfrak{c}},p)-\mathbf{\Gamma}(f,f)(s,y,p)\Big]ds\nonumber\\
%		& \quad +\int_0^t \sqrt{J_{\mathfrak{c}}(p)}\Big[e^{-\nu_{\mathfrak{c}}(p)(t-s)}-e^{-\nu(p)(t-s)}\Big]\mathbf{\Gamma}(f,f)(s,y,p)ds\nonumber\\
\mathcal{L}_{3}&=\int_0^t \sqrt{J_{\mathfrak{c}}(p)}\Big[e^{-\nu_{\mathfrak{c}}(p)(t-s)}-e^{-\nu(p)(t-s)}\Big]\mathbf{\Gamma}(f,f)(s,y,p)ds\nonumber\\
&\quad +\int_0^tds \int_{\mathbb{R}^3\times \mathbb{S}^2} \sqrt{J_{\mathfrak{c}}(p)}e^{-\nu_{\mathfrak{c}}(p)(t-s)}\left[\mathcal{K}_{\mathfrak{c}}(p,q,\omega)\sqrt{J_{\mathfrak{c}}(q)}
f_{\mathfrak{c}}(s,y_{\mathfrak{c}},p')f_{\mathfrak{c}}(s,y_{\mathfrak{c}},q')\right.\nonumber\\
&\qquad \qquad \left.-\mathcal{K}_{\infty}(p,q,\omega)\sqrt{\mu(q)}f(s,y,\bar{p}')f(s,y,\bar{q}')\right]d\omega dq\nonumber\\
&\quad -\int_0^tds \int_{\mathbb{R}^3\times \mathbb{S}^2} \sqrt{J_{\mathfrak{c}}(p)}e^{-\nu_{\mathfrak{c}}(p)(t-s)}\left[\mathcal{K}_{\mathfrak{c}}(p,q,\omega)\sqrt{J_{\mathfrak{c}}(q)}
f_{\mathfrak{c}}(s,y_{\mathfrak{c}},p)f_{\mathfrak{c}}(s,y_{\mathfrak{c}},q)\right.\nonumber\\ 
&\qquad  \qquad \left. -\mathcal{K}_{\infty}(p,q,\omega)\sqrt{\mu(q)}f(s,y,p)f(s,y,q)\right]d\omega dq\nonumber\\
&:=\sum_{i=1}^3 \mathcal{L}_{3i}.
\end{align*}
For $\mathcal{L}_{31}$, noting that $\|J_{\mathfrak{c}}^{\frac{1}{4}}(p)\mathbf{\Gamma}(f,f)(s)\|_{L^{\infty}_{x,p}}\lesssim 1$, by similar arguments as in \eqref{4.21-0}, one obtains 
\begin{align*}
\|\mathcal{L}_{31}\|_{L^1_pL^{\infty}_x}\lesssim \frac{1}{\mathfrak{c}^{2-2\epsilon}}.
\end{align*}
We then expand $\mathcal{L}_{32}$ as 
\begin{align}\label{2.44}
&\int_0^tds \int_{\mathbb{R}^3\times \mathbb{S}^2} \sqrt{J_{\mathfrak{c}}(p)}e^{-\nu_{\mathfrak{c}}(p)(t-s)}\mathcal{K}_{\mathfrak{c}}(p,q,\omega)(\sqrt{J_{\mathfrak{c}}(q)}-\sqrt{\mu(q)})f_{\mathfrak{c}}(s,y_{\mathfrak{c}},p')f_{\mathfrak{c}}(s,y_{\mathfrak{c}},q')d\omega dq\nonumber\\
&+\int_0^tds \int_{\mathbb{R}^3\times \mathbb{S}^2} \sqrt{J_{\mathfrak{c}}(p)}e^{-\nu_{\mathfrak{c}}(p)(t-s)}\mathcal{K}_{\mathfrak{c}}(p,q,\omega)\sqrt{\mu(q)}(f_{\mathfrak{c}}(s,y_{\mathfrak{c}},p')-f(s,y_{\mathfrak{c}},p'))f_{\mathfrak{c}}(s,y_{\mathfrak{c}},q')d\omega dq\nonumber\\
&+\int_0^tds \int_{\mathbb{R}^3\times \mathbb{S}^2} \sqrt{J_{\mathfrak{c}}(p)}e^{-\nu_{\mathfrak{c}}(p)(t-s)}\mathcal{K}_{\mathfrak{c}}(p,q,\omega)\sqrt{\mu(q)}(f_{\mathfrak{c}}(s,y_{\mathfrak{c}},q')-f(s,y_{\mathfrak{c}},q'))f(s,y_{\mathfrak{c}},p')d\omega dq\nonumber\\
&+\int_0^tds \int_{\mathbb{R}^3\times \mathbb{S}^2} \sqrt{J_{\mathfrak{c}}(p)}e^{-\nu_{\mathfrak{c}}(p)(t-s)}\mathcal{K}_{\mathfrak{c}}(p,q,\omega)\sqrt{\mu(q)}(f(s,y_{\mathfrak{c}},q')-f(s,y,q'))f(s,y_{\mathfrak{c}},p')d\omega dq\nonumber\\
&+\int_0^tds \int_{\mathbb{R}^3\times \mathbb{S}^2} \sqrt{J_{\mathfrak{c}}(p)}e^{-\nu_{\mathfrak{c}}(p)(t-s)}\mathcal{K}_{\mathfrak{c}}(p,q,\omega)\sqrt{\mu(q)}(f(s,y_{\mathfrak{c}},p')-f(s,y,p'))f(s,y,q')d\omega dq\nonumber\\
&+\int_0^tds \int_{\mathbb{R}^3\times \mathbb{S}^2} \sqrt{J_{\mathfrak{c}}(p)}e^{-\nu_{\mathfrak{c}}(p)(t-s)}\mathcal{K}_{\mathfrak{c}}(p,q,\omega)\sqrt{\mu(q)}(f(s,y,q')-f(s,y,\bar{q}'))f(s,y,p')d\omega dq\nonumber\\
&+\int_0^tds \int_{\mathbb{R}^3\times \mathbb{S}^2} \sqrt{J_{\mathfrak{c}}(p)}e^{-\nu_{\mathfrak{c}}(p)(t-s)}\mathcal{K}_{\mathfrak{c}}(p,q,\omega)\sqrt{\mu(q)}(f(s,y,p')-f(s,y,\bar{p}'))f(s,y,\bar{q}')d\omega dq\nonumber\\
&+\int_0^tds \int_{\mathbb{R}^3\times \mathbb{S}^2} \sqrt{J_{\mathfrak{c}}(p)}e^{-\nu_{\mathfrak{c}}(p)(t-s)}(\mathcal{K}_{\mathfrak{c}}-\mathcal{K}_{\infty})(p,q,\omega)\sqrt{\mu(q)}f(s,y,\bar{p}')f(s,y,\bar{q}')d\omega dq.
\end{align}

It is noted that the estimate of each item on the RHS of \eqref{2.44} is similar to that of one item on the RHS of \eqref{4.45} with the same bound. The estimate for $\|\mathcal{L}_{33}\|_{L^1_pL^{\infty}_x}$ is easier than the one for $\|\mathcal{L}_{32}\|_{L^1_pL^{\infty}_x}$. Thus we have
\begin{align}\label{2.45}
\|\mathcal{L}_{3}\|_{L^1_pL^{\infty}_x}\lesssim r(\mathfrak{c})\int_0^te^{-\nu_{0}(t-s)}\|\sqrt{J_{\mathfrak{c}}}(f_{\mathfrak{c}}-f)(s)\|_{L^1_pL^{\infty}_x}ds+\frac{1}{\mathfrak{c}^{2-3\epsilon}}.
\end{align}

Combining \eqref{4.28-000}, \eqref{2.33}, \eqref{2.41} and \eqref{2.45}, we conclude that
\begin{align*}
\|\sqrt{J_{\mathfrak{c}}}(f_{\mathfrak{c}}-f)(t)\|_{L^1_pL^{\infty}_x} &\le \frac{C}{\mathfrak{c}^{k-3\epsilon}} +Cr(\mathfrak{c})\int_0^te^{-\nu_{0}(t-s)}\|\sqrt{J_{\mathfrak{c}}}(f_{\mathfrak{c}}-f)(s)\|_{L^1_pL^{\infty}_x}ds.
\end{align*}
By similar arguments as in \eqref{4.11-20}--\eqref{4.17-0}, for any $0\le t\le (\log \mathfrak{c})^{\alpha_0}$, one has
\begin{align*}
\|\sqrt{J_{\mathfrak{c}}}(f_{\mathfrak{c}}-f)(t)\|_{L^1_pL^{\infty}_x} &\le C(\epsilon) / \mathfrak{c}^{k-4\epsilon}.
\end{align*}
Take $\delta=4\epsilon$ and we obtain \eqref{1.7}. 
%	By Gr\"{o}nwall type inequality, we have
%	\begin{align}\label{2.47}
%		\|\sqrt{\mu}(f_{\mathfrak{c}}-f)(t)\|_{L^1_pL^{\infty}_x} &\le \frac{C}{\mathfrak{c}^{2-3\epsilon}} \frac{Cr(\mathfrak{c})}{Cr(\mathfrak{c})-\nu_0}e^{(Cr(\mathfrak{c})-\nu_0)t}-\frac{C}{\mathfrak{c}^{2-3\epsilon}}\frac{\nu_0}{Cr(\mathfrak{c})-\nu_0}e^{-(Cr(\mathfrak{c})-\nu_0)t}\nonumber\\ 
%		&\le \frac{C}{\mathfrak{c}^{2-3\epsilon}}e^{Cr(\mathfrak{c})t}.
%	\end{align}
%    For any $0<\alpha_0<\frac{1}{2}$, we choose $\alpha_1=\frac{3}{4}-\frac{\alpha_0}{2}$, which yields that $\frac{1}{2}<\alpha_1<1-\alpha_0$. For any $0\le t\le (\log \mathfrak{c})^{\alpha_0}$, one has
%    \begin{align}
%    	e^{Cr(\mathfrak{c})t}\le e^{C(\log \mathfrak{c})^{\alpha_0+\alpha_1}}\lesssim \mathfrak{c}^{\epsilon},
%    \end{align}
%    which, together with \eqref{2.47}, yields that
%    \begin{align*}
%   	\|\sqrt{\mu}(f_{\mathfrak{c}}-f)(t)\|_{L^1_pL^{\infty}_x} &\le C(\epsilon) / \mathfrak{c}^{k-4\epsilon}.
%    \end{align*}
%	Combining \eqref{2.26-0}, \eqref{2.26} and \eqref{2.47}, we finally conclude that 
%	\begin{align*} 
%		\sup_{t\ge0}\left\|F_{\mathfrak{c}}(t)-F(t)\right\|_{L_p^1 L_x^{\infty}} \leq C(\epsilon) / \mathfrak{c}^{k-4\epsilon}.
%	\end{align*}

On the other hand, using the exponential decay property of $f_{\mathfrak{c}}(t,x,p)$ and $f(t,x,p)$, one has
\begin{align*}
\|\sqrt{J_{\mathfrak{c}}}(f_{\mathfrak{c}}-f)(t)\|_{L^1_pL^{\infty}_x}\lesssim e^{-\tilde{\sigma}t}, \quad t\ge 0.
\end{align*}
Hence for $(\log \mathfrak{c})^{\alpha_0}\le t\le (\log \mathfrak{c})^{\beta_0}$ with $\beta_0>1$, one has
\begin{align*}
\|\sqrt{J_{\mathfrak{c}}}(f_{\mathfrak{c}}-f)(t)\|_{L^1_pL^{\infty}_x}\lesssim e^{-\tilde{\sigma}(\log \mathfrak{c})^{\alpha_0}}.
\end{align*}
For $t\ge (\log \mathfrak{c})^{\beta_0}$, using Lemma \ref{lem2.5}, one has
\begin{align*}
\|\sqrt{J_{\mathfrak{c}}}(f_{\mathfrak{c}}-f)(t)\|_{L^1_pL^{\infty}_x}\lesssim e^{-\tilde{\sigma}(\log \mathfrak{c})^{\beta_0}}\lesssim \frac{1}{\mathfrak{c}^2}.
\end{align*}
Therefore the proof of Theorem \ref{thm1.3} is completed.
$\hfill\Box$

%%%%%%%%%%%%%%%%%%%%%%%%%%%%%%%%%%%%%%%%%%%%%%%%%%%%%%%%%%%%%%%
\section{Global-in-time Newtonian limit in $L^{\infty}_{x,p}$}
%%%%%%%%%%%%%%%%%%%%%%%%%%%%%%%%%%%%%%%%%%%%%%%%%%%%%%%%%%%%%%%

In this section, we are devoted to proving Theorem \ref{thm1.4}, i.e., the Newtonian limit in $L^{\infty}_{x,p}$. We first consider the propagation of regularities for the Newtonian Boltzmann equation. The point is that we do not assume any smallness on the derivatives of initial data, and such regularity plays a key role in the Newtonian limit in $L^\infty_{x,p}$.

\begin{Lemma}\label{lem5.1}
Under the assumptions of Theorem \ref{thm1.4}, there exists a positive constant $\hat{\lambda}>0$ such that
	\begin{align*}
		\|w_{\beta-1}\nabla_xf(t)\|_{L^\infty_{x,p}}\lesssim e^{-\hat{\lambda}t},\text{ for }t\geq0.
	\end{align*}
\end{Lemma}
\begin{proof}
It is easy to prove the  $W^{1,\infty}$ regularity in local time, so we need only to establish  the corresponding {\it a priori} estimate.
	 Differentiating \eqref{2.4-10} with respect to $x_i$ leads to
%	\begin{align*}
%		\begin{cases}
%			\partial_t\partial_{x_i}f+p\cdot\nabla_x\partial_{x_i}f+\FL\partial_{x_i}f=\mathbf{\Gamma}(\partial_{x_i}f,f)+\mathbf{\Gamma}(f,\partial_{x_i}f),\\
%			\partial_{x_i}f(t,x,p)|_{t=0}=\partial_{x_i}f_0(x,p).
%		\end{cases}
%	\end{align*}
%	Denote $\partial_{x_i}f$ as $\mathfrak{y}_i,$ then the equation for $\mathfrak{y}_i$ is
	\begin{align}\label{5.3}
		\begin{cases}
			\partial_t\mathfrak{y}_i+p\cdot\nabla_x\mathfrak{y}_i+\FL\mathfrak{y}_i=\mathbf{\Gamma}(\mathfrak{y}_i,f)+\mathbf{\Gamma}(f,\mathfrak{y}_i),\\
			\mathfrak{y}_i|_{t=0}=\partial_{x_i}f_0(x,p),
		\end{cases}
	\end{align}
    where $\mathfrak{y}_i:=\partial_{x_i}f$. 
%    By similar arguments as in Proposition \ref{prop3.2} and Lemma \ref{lem4.5}, one obtains 
%	\begin{align}\label{5.4}
%		\|\mathfrak{y}_i(t)\|_{L^2_{x,p}}^2\leq Ce^{-\sigma_1t}\left\{\|\mathfrak{y}_{i}(0)\|_{L^2_{x,p}}^2+\int_0^t e^{\sigma_1s}\big(\|\mathbf{\Gamma}(\mathfrak{y}_i,f)(s)\|_{L^2_{x,p}}^2+\|\mathbf{\Gamma}(f,\mathfrak{y}_i)(s)\|_{L^2_{x,p}}^2\big)ds\right\}
%	\end{align}
%	for some positive constant $\sigma_1>0$. 
	
	For later use, we rewrite $\eqref{5.3}_1$  as 
	\begin{align}\label{5.5}
		\partial_t\mathfrak{y}_i+p\cdot\nabla_x\mathfrak{y}_i+\nu_f\mathfrak{y}_i=\mathbf{K}\mathfrak{y}_i+\mathbf{\Gamma}^+(\mathfrak{y}_i,f)+\mathbf{\Gamma}^+(f,\mathfrak{y}_i)-\mathbf{\Gamma}^{-}(f,\mathfrak{y}_i),
	\end{align}
	where
	\begin{align*}
%		\mathbf{\Gamma}^+(\mathfrak{y}_i,f)+\mathbf{\Gamma}^+(f,\mathfrak{y}_i)&=\iint_{\R^3\times\mathbb{S}^2}|(p-q)\cdot\omega|\sqrt{\mu(q)}[\mathfrak{y}_i(\bar{p}')f(\bar{q}')+\mathfrak{y}_i(\bar{q}')f(\bar{p}')]dqd\omega,\\
%	    \mathbf{\Gamma}^{-}(f,\mathfrak{y}_i)&=\iint_{\R^3\times\mathbb{S}^2}|(p-q)\cdot\omega|\sqrt{\mu(q)}\mathfrak{y}_i(q)f(p)dqd\omega,\\
		\nu_f(t,x,p)&:=\iint_{\R^3\times\mathbb{S}^2}|(p-q)\cdot\omega|\left\{\mu(q)+\sqrt{\mu(q)}f(t,x,q)\right\}dqd\omega\geq 0.
	\end{align*}
	Noting \cite[Lemma 2.2]{Duan-2}, we have
	\begin{align*}
		|w_{\beta-1}\mathbf{\Gamma}^+(\mathfrak{y}_i,f)|+|w_{\beta-1}\mathbf{\Gamma}^+(f,\mathfrak{y}_i)|\lesssim \|w_{\beta-1} f\|_{L^\infty_p}\cdot \|w_{\beta-1}\mathfrak{y}_i\|_{L^\infty_p},
	\end{align*}
    and one can easily get
	\begin{align*}
		|w_{\beta-1} \mathbf{\Gamma}^{-}(f,\mathfrak{y}_i)|\leq \|\mathfrak{y}_i\|_{L^\infty_p}\cdot \|w_\beta f\|_{L^\infty_p}.
	\end{align*}
    
    Set $\tilde{\nu}_0:=\inf_{p\in \mathbb{R}^3}\nu(p)$. Noting \eqref{1.30}, one can choose $T_0$ suitably large so that $\nu_f(t,x,p)\geq \f12\nu(p)\ge \frac{1}{2}\tilde{\nu}_0$ for $t\ge T_0$. 
    
    The mild form of equation \eqref{5.5} can be written as
\begin{align}\label{5.4-05}
&\displaystyle w_{\beta-1}(p)\mathfrak{y}_i(t,x,p)\nonumber\\
&\displaystyle=e^{-\int_0^t
\nu_f(s,x,p)ds}w_{\beta-1}(p)\mathfrak{y}_{i}(0,x-pt,p)\nonumber\\
&\quad \displaystyle +\int_0^te^{-\int_s^t
\nu_f(s_1,x,p)ds_1}w_{\beta-1}(p)\mathbf{K}\mathfrak{y}_i\left(s,x-p(t-s),p\right)ds\nonumber\\
&\quad\displaystyle+\int_0^t e^{-\int_s^t
\nu_f(s_1,x,p)ds_1} w_{\beta-1}(p)\{\mathbf{\Gamma}^+(\mathfrak{y}_i,f)+\mathbf{\Gamma}^+(f,\mathfrak{y}_i)-\mathbf{\Gamma}^{-}(f,\mathfrak{y}_i)\}\left(s,x-p(t-s),p\right)ds\nonumber\\
&:=I_1+I_2+I_3, \quad\text{ for } t\leq T_0,
	\end{align}
and
	\begin{align}\label{5.8-1}
	&w_{\beta-1}(p)\mathfrak{y}_i(t,x,p)\nonumber\\
	&=e^{-\int_{T_0}^t
\nu_f(s,x,p)ds}w_{\beta-1}(p)\mathfrak{y}_{i}(T_0,x-p(t-T_0),p)\nonumber\\
	&\quad+\int_{T_0}^te^{-\int_s^t
\nu_f(s_1,x,p)ds_1}w_{\beta-1}(p)\mathbf{K}\mathfrak{y}_i\left(s,x-p(t-s),p\right)ds\nonumber\\
	&\quad +\int_{T_0}^t e^{-\int_s^t
\nu_f(s_1,x,p)ds_1} w_{\beta-1}(p)\{\mathbf{\Gamma}^+(\mathfrak{y}_i,f)+\mathbf{\Gamma}^+(f,\mathfrak{y}_i)-\mathbf{\Gamma}^{-}(f,\mathfrak{y}_i)\}\left(s,x-p(t-s),p\right)ds\nonumber\\
	&:=M_1+M_2+M_3, \quad\text{ for } t> T_0.
\end{align}

	\noindent{\it Case 1. $t\in[0,T_0]$}. It follows from \eqref{5.4-05} that
	\begin{align*} 
		\|w_{\beta-1}\mathfrak{y}_i(s)\|_{L^\infty_{x,p}} \lesssim  \|w_{\beta-1}\partial_{x_i}f_0\|_{L^\infty_{x,p}} + C\int_0^t \|w_{\beta-1}\mathfrak{y}_i(s)\|_{L^\infty_{x,p}}ds,
	\end{align*}
which, together with Gr\"{o}nwall's inequality and \eqref{1.25}, yields that 
	\begin{align}\label{5.8-0}
		\sup_{0 \le s \le T_0}\|w_{\beta-1}\mathfrak{y}_i(s)\|_{L^\infty_{x,p}}\leq Ce^{CT_0}.
	\end{align}
	
	\noindent{\it Case 2. $t>T_0$}. Now it is clear that $
\nu_f(t,x,p)\geq \f12\nu(p)\geq \f12\tilde{\nu}_0>0$. Thus one has
	\begin{align*}
		|M_1|\leq C_{T_0}e^{-\f12\tilde{\nu}_0t}
	\end{align*}
	and
	\begin{align*} 
		|M_3|&\leq C\int_{T_0}^{t} e^{-\f12\tilde{\nu}_0(t-s)}e^{-\sigma_0s}\|w_{\beta-1}\mathfrak{y}_i(s)\|_{L^\infty_{x,p}}ds\nonumber\\
		%&\le Ce^{-\min\{\f12\tilde{\nu}_0,\sigma_0\}t}\int_{T_0}^t\|w_{\beta-1}\mathfrak{y}_i(s)\|_{L^\infty_{x,p}}ds\nonumber\\
		&\leq Ce^{-\hat{\lambda}t}\int_{T_0}^te^{-\hat{\lambda}s}\|e^{\hat{\lambda}s}w_{\beta-1}\mathfrak{y}_i(s)\|_{L^\infty_{x,p}}ds.
	\end{align*}
    Here we choose $\hat{\lambda}$ such that
    \begin{align*}
   	0<\hat{\lambda}<\min\Big\{\f18 \nu_0,\f18\tilde{\nu}_0,\frac{\sigma_0}{2},\frac{\sigma_1}{2},\frac{\sigma_2}{2},\frac{\sigma_3}{2}\Big\},
   \end{align*}
   where $\sigma_i\ (i=1,2,3)$ are the constants in \eqref{5.6-000}, \eqref{5.25-000} and \eqref{5.78-000}, respectively.
%	noting \eqref{1.30} and assuming $\sigma_0\leq \f12\nu_0$. Actually, if $\sigma_0\geq \f12\nu_0$, then
%	\begin{align*}
%		|M_3|&\leq e^{-\f12\nu_0t}\int_0^te^{-(\sigma_0-\f12\nu_0)s}\|w_{\beta-1}\mathfrak{y}_i(s)\|_{L^\infty_{x,p}}\\
%		&\leq e^{-\f12\nu_0t}\int_0^t\|w_{\beta-1}\mathfrak{y}_i(s)\|_{L^\infty_{x,p}}ds,
%	\end{align*}
%	which together with \eqref{5.9-0}, yields that
%	\begin{align*}
%		|M_3|\leq e^{-\min\{\f12\nu_0,\sigma_0\}t}\int_{T_0}^t\|w_{\beta-1}\mathfrak{y}_i(s)\|_{L^\infty_{x,p}}ds.
%	\end{align*}
	
	For $M_2$, using \eqref{5.8-1} once again, one obtains
	\begin{align*}
		M_2=&\int_{T_0}^te^{-\int_s^t
\nu_f(\tau,x,p)d\tau}e^{-\int_{T_0}^s
\nu_f(s_1,y,\tilde{p})ds_1}ds\int_{\R^3}\tilde{k}_w(p,\tilde{p})w_{\beta-1}(\tilde{p})\mathfrak{y}_i\left(T_0,y-\tilde{p}(s-T_0),\tilde{p}\right)d\tilde{p}\nonumber\\
		&+\int_{T_0}^t e^{-\int_s^t
\nu_f(\tau,x,p)d\tau}ds\int_{T_0}^se^{-\int_{s_1}^s
\nu_f(s_2,y,\tilde{p})ds_2}ds_1\iint_{\R^3\times\R^3}\tilde{k}_w(p,\tilde{p})\tilde{k}_w(\tilde{p},\tilde{p}')\\
		&\qquad\qquad \times w_{\beta-1}(\tilde{p}')\mathfrak{y}_i(s_1,y-\tilde{p}(s-s_1),\tilde{p}')d\tilde{p}d\tilde{p}'\nonumber\\
		&+\int_{T_0}^te^{-\int_s^t
\nu_f(\tau,x,p)d\tau}ds\int_{T_0}^s e^{-\int_{s_1}^s
\nu_f(s_2,y,\tilde{p})ds_2}ds_1\int_{\R^3}\tilde{k}_w(p,\tilde{p})w_{\beta-1}(\tilde{p})\\
		&\qquad\qquad \times \{\mathbf{\Gamma}^+(\mathfrak{y}_i,f)+\mathbf{\Gamma}^+(f,\mathfrak{y}_i)-\mathbf{\Gamma}^{-}(f,\mathfrak{y}_i)\} \left(s_1,y-\tilde{p}(s-s_1),\tilde{p}\right)d\tilde{p}\\
		:=&M_{2,1}+M_{2,2}+M_{2,3}.
	\end{align*}
 It is direct to have 
	\begin{align*}
		|M_{2,1}|\leq C_{T_0}\int_{T_0}^t e^{-\f12\tilde{\nu}_0(t-s)}e^{-\f12\tilde{\nu}_0s}ds\cdot \|w_{\beta-1}\mathfrak{y}_i(T_0)\|_{L^\infty_{x,p}}\leq C_{T_0}e^{-\f14\tilde{\nu}_0t}
	\end{align*}
	and 
	\begin{align*}
		|M_{2,3}|\leq Ce^{-\hat{\lambda}t}\int_{T_0}^te^{-\hat{\lambda}s}\|e^{\hat{\lambda}s}w_{\beta-1}\mathfrak{y}_i(s)\|_{L^\infty_{x,p}}ds.
	\end{align*}
	
	For $M_{2,2}$, similar to Lemma \ref{lem4.3}, one can obtain
	\begin{align*}
		|M_{2,2}|%&\leq \f{C}{N}e^{-\hat{\lambda}t}\sup_{T_0 \le s \le t}\|e^{\hat{\lambda}s}w_{\beta-1}\mathfrak{y}_i(s)\|_{L^\infty_{x,p}}  +C_N\int_{T_0}^te^{-\f14\tilde{\nu}_0(t-s)}\|\mathfrak{y}_i(s)\|_{L^2_{x,p}}ds\nonumber\\
		&\leq \f{C}{N}e^{-\hat{\lambda}t}\sup_{T_0 \le s \le t}\|e^{\hat{\lambda}s}w_{\beta-1}\mathfrak{y}_i(s)\|_{L^\infty_{x,p}}+C_N\int_{T_0}^te^{-2\hat{\lambda}(t-s)}\|\mathfrak{y}_i(s)\|_{L^2_{x,p}}ds.
	\end{align*}
    By similar arguments as in Proposition \ref{prop3.2} and Lemma \ref{lem4.5}, there exists a positive constant $\sigma_1>0$, such that
    \begin{align}\label{5.6-000}
    	&\|\mathfrak{y}_i(s)\|_{L^2_{x,p}}\nonumber\\
    	&\le  e^{-\frac{\sigma_1}{2}s}\|\mathfrak{y}_i(T_0)\|_{L^2_{x,p}}+C\Big(\int_{T_0}^se^{-\sigma_1(s-s_1)}[\|\Gamma(f,\mathfrak{y}_i)(s_1)\|_{L^2_{x,p}}^2+\|\Gamma(\mathfrak{y}_i,f)(s_1)\|_{L^2_{x,p}}^2]ds_1\Big)^{\frac{1}{2}}\nonumber\\
    	&\le C_{T_0}e^{-\frac{\sigma_1}{2}s}+C\Big(\int_{T_0}^se^{-\sigma_1(s-s_1)}e^{-2\sigma_0s_1}\|w_{\beta-1}\mathfrak{y}_i(s_1)\|_{L^{\infty}_{x,p}}^2ds_1\Big)^{\frac{1}{2}}.
    \end{align}
    In view of the choice of $\hat{\lambda}$, it is easy to see that
    \begin{align*}
    \int_{T_0}^te^{-2\hat{\lambda}(t-s)}C_{T_0}e^{-\frac{\sigma_1}{2}s}ds\le C_{N,T_0}e^{-\hat{\lambda}t}.
    \end{align*}
	Hence we obtain that
	\begin{align}\label{5.9-40}
	&\|w_{\beta-1}\mathfrak{y}_i(t)\|_{L^\infty_{x,p}}\nonumber\\
	&\leq C_{N,T_0}e^{-\hat{\lambda}t} +Ce^{-\hat{\lambda}t}\int_{T_0}^te^{-\hat{\lambda}s}\|e^{\hat{\lambda}s}w_{\beta-1}\mathfrak{y}_i(s)\|_{L^\infty_{x,p}}ds + \f{C}{N}e^{-\hat{\lambda}t}\sup_{T_0 \le s \le t}\|e^{\hat{\lambda}s}w_{\beta-1}\mathfrak{y}_i(s)\|_{L^\infty_{x,p}}\nonumber\\
		&\quad +C_{N}\int_{T_0}^te^{-2\hat{\lambda}(t-s)}\left\{\int_{T_0}^se^{-\sigma_1(s-s_1)}e^{-2\sigma_0s_1}\|w_{\beta-1}\mathfrak{y}_i(s_1)\|_{L^\infty_{x,p}}^2ds_1\right\}^{\f12}ds,
	\end{align}
    which, together with H\"{o}lder's inequality, yields that
    	\begin{align*}
    	\|w_{\beta-1}\mathfrak{y}_i(t)\|^2_{L^\infty_{x,p}} 
    	%&\leq C_{T_0}e^{-\frac{\tilde{\nu}_0}{2}t}+Ce^{-2\hat{\lambda}t}\int_{T_0}^te^{-\hat{\lambda}s}\|e^{\hat{\lambda}s}w_{\beta-1}\mathfrak{y}_i(s)\|^2_{L^\infty_{x,p}}ds\nonumber\\
    	%&\quad +C_{N,T_0}e^{-2\hat{\lambda}t}+\f{C}{N^2}e^{-2\hat{\lambda}t}\sup_{T_0 \le s \le t}\|e^{\hat{\lambda}s}w_{\beta-1}\mathfrak{y}_i(s)\|^2_{L^\infty_{x,p}}\nonumber\\
    	%&\quad +C_{N}e^{-2\hat{\lambda}t}\int_{T_0}^te^{-\hat{\lambda}(t-s)}ds\nonumber\\
    	%&\qquad \times \int_{T_0}^tds \int_{T_0}^s e^{-3\hat{\lambda}(t-s)}e^{2\hat{\lambda}(t-s_1)}e^{-\sigma_1(s-s_1)}e^{-2\sigma_0s_1}\|e^{\hat{\lambda}s_1}w_{\beta-1}\mathfrak{y}_i(s_1)\|_{L^\infty_{x,p}}^2ds_1\nonumber\\
    	&\leq C_{N,T_0}e^{-2\hat{\lambda}t}+C_{N}e^{-2\hat{\lambda}t}\int_{T_0}^te^{-\hat{\lambda}s}\|e^{\hat{\lambda}s}w_{\beta-1}\mathfrak{y}_i(s)\|^2_{L^\infty_{x,p}}ds\nonumber\\
    	&\quad +\f{C}{N^2}e^{-2\hat{\lambda}t}\sup_{T_0 \le s \le t}\|e^{\hat{\lambda}s}w_{\beta-1}\mathfrak{y}_i(s)\|^2_{L^\infty_{x,p}}.
    \end{align*}

    Choosing $N$ suitably large, we conclude that
    \begin{align*}
    	\sup_{T_0 \le \tau \le t}\|e^{\hat{\lambda}\tau}w_{\beta-1}\mathfrak{y}_i(\tau)\|_{L^\infty_{x,p}}^2&\leq C_{T_0}+C\int_{T_0}^te^{-\hat{\lambda}s}\sup_{T_0 \le \tau \le s}\|e^{\hat{\lambda}\tau}w_{\beta-1}\mathfrak{y}_i(\tau)\|^2_{L^\infty_{x,p}}ds,
    \end{align*}
which, together with Gr\"{o}nwall's inequality, yields that 
    \begin{align*}
    	\sup_{T_0 \le \tau \le t}\|e^{\hat{\lambda}\tau}w_{\beta-1}\mathfrak{y}_i(\tau)\|_{L^\infty_{x,p}}^2\le C_{T_0}.
    \end{align*}
   Thus we have
     \begin{align*}
     	\|w_{\beta-1}\mathfrak{y}_i(s)\|_{L^\infty_{x,p}}\le C_{T_0}e^{-\hat{\lambda}s},\quad \text{for}\ T_0\le s\le t,
     \end{align*}
%	Then it follows 
%	\begin{align}\label{5.10}
%		\sup_{T_0 \le s \le t}\|w_{\beta-1}\mathfrak{y}_i(s)\|_{L^\infty_{x,p}}^2\leq& Ce^{-\min\{2\sigma_0,\f12\tilde{\nu}_0,\sigma_1\}t}+C\int_{T_0}^te^{-\sigma_1(t-s)}e^{-2\sigma_0s}\|w_{\beta-1}\mathfrak{y}_i(s)\|_{L^\infty_{x,p}}^2ds\nonumber\\
%		\leq& Ce^{-\min\{2\sigma_0,\f12\tilde{\nu}_0,\sigma_1\}t}+Ce^{-\min\{\sigma_1,2\sigma_0\}t}\int_{T_0}^t\|w_{\beta-1}\mathfrak{y}_i(s)\|_{L^\infty_{x,p}}^2ds.
%	\end{align}
%	
%	Using the Gr\"{o}nwall inequality, one has that
%	\begin{align*}
%		\sup_{T_0 \le s \le t}\|w_{\beta-1}\mathfrak{y}_i(s)\|_{L^\infty_{x,p}}\lesssim e^{-\min\{\sigma_0,\f14\tilde{\nu}_0,\f12\sigma_1\}t},
%	\end{align*}
	which, together with \eqref{5.8-0}, yields that
	\begin{align}\label{4.49-1}
		\|w_{\beta-1} \partial_{x_i} f(t)\|_{L^\infty_{x,p}} &\lesssim e^{-\hat{\lambda}t},\quad i=1,2,3
	\end{align}
	for all $t\ge 0$. Therefore the proof of Lemma \ref{lem5.1} is completed.
\end{proof}

\begin{Lemma}\label{lem5.2}
		Under the assumptions of Theorem \ref{thm1.4}, it holds that 
	\begin{align*}
		\|w_{\beta-2}\nabla_pf(t)\|_{L^\infty_{x,p}}\lesssim 1,\text{ for }t\geq0.
	\end{align*}
\end{Lemma}
\begin{proof}
	 Differentiating \eqref{2.4-10} with respect to $p_i$, we obtain 
	\begin{align*}
		\partial_t\partial_{p_i}f+p\cdot\nabla_x\partial_{p_i}f+\partial_{x_i}f=\partial_{p_i}\Big(\f{1}{\sqrt{\mu}}Q(F,F)\Big).
	\end{align*}
	A direct calculation shows that
	\begin{align*}
		&\partial_{p_i}\Big(\f{1}{\sqrt{\mu}}Q(F,F)\Big)-\f{p_i}{2}\f{1}{\sqrt{\mu}}Q(F,F)\nonumber\\
		&=\f{1}{\sqrt{\mu}}\partial_{p_i}\Big\{\int_{\R^3\times\mathbb{S}^2}\mathcal{K}_\infty(p,q,\omega)[F(\bar{p}')F(\bar{q}')-F(p)F(q)] d\omega dq\Big\}\nonumber\\
		&=\f{1}{\sqrt{\mu}}\partial_{p_i}\Big\{\int_{\R^3\times\mathbb{S}^2}|\xi\cdot\omega|\Big[F(p+(\xi\cdot\omega)\omega)F(p+\xi-(\xi\cdot\omega)\omega)-F(p)F(p+\xi)\Big] d\omega d\xi\Big\}\nonumber\\
		&=\f{1}{\sqrt{\mu}}\big\{Q(\partial_{p_i}F,F)+Q(F,\partial_{p_i}F)\big\}\nonumber\\
		&=-\FL\partial_{p_i}f+\mathbf{\Gamma}(\partial_{p_i}f,f)+\mathbf{\Gamma}(f,\partial_{p_i}f)+\mathbf{\Gamma}(-\f{p_i}{2}f,f)+\mathbf{\Gamma}(f,-\f{p_i}{2}f)\nonumber\\
		&\qquad +\FL \{\f{p_i}{2}f\}+\mathbf{\Gamma}(-p_i\sqrt{\mu},f)+\mathbf{\Gamma}(f,-p_i\sqrt{\mu}).
	\end{align*}
Denote $\mathfrak{f}_i:=\partial_{p_i}f$, then the equation of $\mathfrak{f}_i$ is
	\begin{align}\label{5.9-a}
		\begin{cases}
			\partial_t \mathfrak{f}_i+p\cdot\nabla_x\mathfrak{f}_i+\FL\mathfrak{f}_i=\mathbf{\Gamma}(f,\mathfrak{f}_i)+\mathbf{\Gamma}(\mathfrak{f}_i,f)+\mathcal{R}_i,\\
			\mathfrak{f}_i(t,x,p)|_{t=0}=\partial_{p_i}f_0(x,p),
		\end{cases}
	\end{align}
	where
	\begin{align}\label{5.9-b}
		\mathcal{R}_i:=&-\partial_{x_i}f+\f{p_i}{2}\f{1}{\sqrt{\mu}}Q(F,F)+\FL \{\f{p_i}{2}f\}+\mathbf{\Gamma}(-\f{p_i}{2}f,f)+\mathbf{\Gamma}(f,-\f{p_i}{2}f)\nonumber\\
		&\quad +\mathbf{\Gamma}(-p_i\sqrt{\mu},f)+\mathbf{\Gamma}(f,-p_i\sqrt{\mu}).
	\end{align}
	
	It is direct to see 
	\begin{align}\label{5.9-1}
	\int_{\mathbb{T}^3}\int_{\R^3} \mathcal{R}_i
	\begin{pmatrix}
		1  \\  p \vspace{1.0ex}\\  \f{|p|^2-3}{\sqrt{6}}
	\end{pmatrix}
	\sqrt{\mu(p)} dpdx	\equiv \int_{\mathbb{T}^3}\int_{\R^3}\f{p_i}{2}\f{1}{\sqrt{\mu}}Q(F,F)
		\begin{pmatrix}
			1  \\  p \vspace{1.0ex}\\  \f{|p|^2-3}{\sqrt{6}}
		\end{pmatrix}
		\sqrt{\mu(p)} dpdx\neq \mathbf{0},
	\end{align}
Due to \eqref{5.9-1}, it is hard to apply the $L^2-L^\infty$ argument. For this, we introduce a  decomposition $\mathfrak{f}_i=\mathfrak{f}_{i,1}+\mathfrak{f}_{i,2}$ such that 
	\begin{align*}
		\int_{\mathbb{T}^3}\int_{\R^3}\big\{\partial_t\mathfrak{f}_{i,1}+p\cdot\nabla_x\mathfrak{f}_{i,1}-\mathcal{R}_i\big\}\begin{pmatrix}
			1  \\  p \vspace{1.0ex}\\  \f{|p|^2-3}{\sqrt{6}}
		\end{pmatrix}\sqrt{\mu(p)}dpdx=\mathbf{0},
	\end{align*}
	which is equivalent to
	\begin{align}\label{4.33-0}
		\int_{\mathbb{T}^3}\int_{\R^3}\partial_t\mathfrak{f}_{i,1}\begin{pmatrix}
			1  \\  p \vspace{1.0ex}\\  \f{|p|^2-3}{\sqrt{6}}
		\end{pmatrix}
	   \sqrt{\mu(p)} dpdx=\int_{\mathbb{T}^3}\int_{\R^3}\f{p_i}{2}Q(F,F)
	   \begin{pmatrix}
	   	1  \\  p \vspace{1.0ex}\\  \f{|p|^2-3}{\sqrt{6}}
	   \end{pmatrix}
       dpdx.
	\end{align}
	We also require that 
	\begin{align}\label{5.12}
		\int_{\mathbb{T}^3}\int_{\R^3}\{\mathfrak{f}_{i}-\mathfrak{f}_{i,1}\}(0,x,p)
		\begin{pmatrix}
			1  \\  p \vspace{1.0ex}\\  \f{|p|^2-3}{\sqrt{6}}
		\end{pmatrix}
		\sqrt{\mu(p)}dpdx=\mathbf{0}.
	\end{align}
    
    The equation for $\mathfrak{f}_{i,2}$ takes the form 
     \begin{align}\label{5.16-30}
    	\partial_t\mathfrak{f}_{i,2}+p\cdot\nabla_x\mathfrak{f}_{i,2}+\FL\mathfrak{f}_{i,2}=\mathbf{\Gamma}(f,\mathfrak{f}_{i,2})+\mathbf{\Gamma}(\mathfrak{f}_{i,2},f)+\bar{\mathcal{R}}_i,
    \end{align}
%    \begin{align}\label{5.18}
%    	\partial_t\mathfrak{f}_{i,2}+p\cdot\nabla_x\mathfrak{f}_{i,2}+\FL\mathfrak{f}_{i,2}=\mathbf{\Gamma}(f,\mathfrak{f}_i)+\mathbf{\Gamma}(\mathfrak{f}_i,f)-\partial_t\mathfrak{f}_{i,1}-p\cdot\nabla_x\mathfrak{f}_{i,1}+\mathcal{R}_i.
%    \end{align}
    where
    \begin{align*}
    	\bar{\mathcal{R}}_i:=\mathbf{\Gamma}(f,\mathfrak{f}_{i,1})+\mathbf{\Gamma}(\mathfrak{f}_{i,1},f)-\FL\mathfrak{f}_{i,1}-\{\partial_t\mathfrak{f}_{i,1}+p\cdot\nabla_x\mathfrak{f}_{i,1}-\mathcal{R}_i\}.
    \end{align*}
    Based on our choice of $\mathfrak{f}_{i,1}$, for the equation \eqref{5.16-30} of $\mathfrak{f}_{i,2}$, we now have
    \begin{align}\label{5.22-06}
    	\int_{\mathbb{T}^3}\int_{\mathbb{R}^3}\Big(\mathbf{\Gamma}(f,\mathfrak{f}_{i,2})+\mathbf{\Gamma}(\mathfrak{f}_{i,2},f)+\bar{\mathcal{R}}_i\Big)
    	\begin{pmatrix}
    		1  \\  p \vspace{1.0ex}\\  \f{|p|^2-3}{\sqrt{6}}
    	\end{pmatrix}
    	\sqrt{\mu(p)}dpdx=\mathbf{0}
    \end{align}
    and
    \begin{align}\label{5.23-06}
    	\int_{\mathbb{T}^3}\int_{\R^3}\mathfrak{f}_{i,2}(0,x,p)
    	\begin{pmatrix}
    		1  \\  p \vspace{1.0ex}\\  \f{|p|^2-3}{\sqrt{6}}
    	\end{pmatrix}
    	\sqrt{\mu(p)}dpdx=\mathbf{0}.
    \end{align}

	Let $\mathfrak{f}_{i,1}=\Big(\mathfrak{b}_i(t,x)\cdot p+\mathfrak{r}_i(t,x)\f{|p|^2-3}{\sqrt{6}}\Big)\sqrt{\mu(p)}$ such that
	\begin{align}\label{4.50-02}
	\partial_t\mathfrak{b}_{i}=\int_{\R^3}\f{p_ip}{2}Q(F,F)dp,\quad \partial_t\mathfrak{r}_i=\int_{\R^3}\f{p_i(|p|^2-3)}{2\sqrt{6}}Q(F,F)dp.
	\end{align}
    For such  $\mathfrak{b}_i$ and $\mathfrak{r}_i$, it is direct to check that \eqref{4.33-0} holds true. 
    
    For the initial data of $\mathfrak{f}_{i,1}$, using \eqref{1.18-10} and integrating by parts, one reduces  \eqref{5.12} to 
    \begin{align}
    	\int_{\mathbb{T}^3}\mathfrak{b}_{i}(0,x)dx&=\f12\int_{\mathbb{T}^3}\int_{\R^3}f_0(x,p)p_ip\sqrt{\mu(p)}dpdx,\label{5.15-06}\\
    	\int_{\mathbb{T}^3}\mathfrak{r}(0,x)dx&=\f12\int_{\mathbb{T}^3}\int_{\R^3}f_0(x,p)\f{p_i(|p|^2-3)}{\sqrt{6}}\sqrt{\mu(p)}dpdx. \label{5.16-06}
    \end{align}
    In view of \eqref{5.15-06}--\eqref{5.16-06} and the boundedness of $f_0$, we take $\mathfrak{b}_{i}(0,x)$ and $\mathfrak{r}_i(0,x)$ as constants 
    \begin{align}\label{4.50-08}
    	\mathfrak{b}_{i}(0,x)\equiv \mathfrak{b}^0_i,\quad \mathfrak{r}_i(0,x)\equiv \mathfrak{r}^0_i,
    \end{align}
    so that \eqref{5.15-06}--\eqref{5.16-06} hold true.
%    
%	\begin{align}
%		\|(\mathfrak{b}_i(0,\cdot),\mathfrak{r}_i(0,\cdot)\|_{L_x^{\infty}}\lesssim 1,\quad \|(\nabla_x\mathfrak{b}_i(0,\cdot),\nabla_x\mathfrak{r}_i(0,\cdot)\|_{L_x^{\infty}}\lesssim 1.
%	\end{align}
%	It is clear that the condition \eqref{4.33-0} reduces to
	
	Using \eqref{1.30}, we have from \eqref{4.50-02} that
	\begin{align}\label{5.13-40}
		&\|(\partial_t\mathfrak{b}_{i}(t,\cdot),\partial_t\mathfrak{r}_i(t,\cdot))\|_{L^{\infty}_x}\lesssim e^{-\sigma_0t},\quad i=1,2,3,
	\end{align}
	which, together with $\eqref{4.50-08}$, yields that
	\begin{align*}%\label{4.51-1}
		\|(\mathfrak{b}_{i},\mathfrak{r}_i)\|_{L^{\infty}_{t,x}}\lesssim1,\quad i=1,2,3.
	\end{align*}
    Thus one has
    \begin{align}\label{5.16-50}
    	\|w_{\beta-2}\mathfrak{f}_{i,1}(t)\|_{L^{\infty}_{x,p}}\lesssim 1,\quad \text{for}\ t\ge 0.
    \end{align}
	Applying the operator $\nabla_x$ to \eqref{4.50-02} and using \eqref{4.49-1}, $\eqref{4.50-08}$, we can further obtain
	\begin{align}\label{4.52-1}
		\|\nabla_x(\mathfrak{b}_i,\mathfrak{r}_i)\|_{L^{\infty}_{t,x}}\lesssim 1.
	\end{align}
    
	The mild form of equation \eqref{5.16-30} can be written as
	\begin{align*}
		&w_{\beta-2}(p)\mathfrak{f}_{i,2}(t,x,p)\nonumber\\
		&=e^{-\int_0^t
\nu_f(s,x,p)ds}w_{\beta-2}(p)\mathfrak{f}_{i,2}(0,x-pt,p)\nonumber\\
		&\quad +\int_0^te^{-\int_s^t
\nu_f(s_1,x,p)ds_1}w_{\beta-2}(p)\mathbf{K}\mathfrak{f}_{i,2}\left(s,x-p(t-s),p\right)ds\nonumber\\
		&\quad+\int_0^t e^{-\int_s^t
\nu_f(s_1,x,p)ds_1} w_{\beta-2}(p)\{\mathbf{\Gamma}^+(\mathfrak{f}_{i,2},f)+\mathbf{\Gamma}^+(f,\mathfrak{f}_{i,2})-\mathbf{\Gamma}^{-}(f,\mathfrak{f}_{i,2})\}\left(s,x-p(t-s),p\right)ds\nonumber\\
		&\quad+\int_0^t e^{-\int_s^t
\nu_f(s_1,x,p)ds_1} w_{\beta-2}(p)\bar{\mathcal{R}}_i\left(s,x-p(t-s),p\right)ds,\quad\text{for }t\leq T_0,
	\end{align*}
	and 
	\begin{align*}
		&w_{\beta-2}(p)\mathfrak{f}_{i,2}(t,x,p)\nonumber\\
		&=e^{-\int_{T_0}^t
\nu_f(s,x,p)ds}w_{\beta-2}(p)\mathfrak{f}_{i,2}(T_0,x-p(t-T_0),p)\nonumber\\
		&\quad+\int_{T_0}^te^{-\int_s^t
\nu_f(s_1,x,p)ds_1}w_{\beta-2}(p)\mathbf{K}\mathfrak{f}_{i,2}\left(s,x-p(t-s),p\right)ds\nonumber\\
		&\quad+\int_{T_0}^t e^{-\int_s^t
\nu_f(s_1,x,p)ds_1} w_{\beta-2}(p)\{\mathbf{\Gamma}^+(\mathfrak{f}_{i,2},f)+\mathbf{\Gamma}^+(f,\mathfrak{f}_{i,2})-\mathbf{\Gamma}^{-}(f,\mathfrak{f}_{i,2})\}\left(s,x-p(t-s),p\right)ds\nonumber\\
		&\quad+\int_{T_0}^t e^{-\int_s^t
\nu_f(s_1,x,p)ds_1} w_{\beta-2}(p)\bar{\mathcal{R}}_i\left(s,x-p(t-s),p\right)ds,\quad \text{for }t> T_0.
	\end{align*}
	
	Using \eqref{1.30}, \eqref{4.49-1} and \eqref{5.13-40}--\eqref{4.52-1}, for all $t\ge 0$, one has
	\begin{align}\label{5.20-08}
		&\sup_{0 \le s \le t}\|w_{\beta-2}\bar{\mathcal{R}}_i(s)\|_{L^\infty_{x,p}}\nonumber\\
		&\lesssim \sup_{0 \le s \le t}\|\nu^{-1}w_{\beta-1}\left(\mathbf{\Gamma}(f,\mathfrak{f}_{i,1})+\mathbf{\Gamma}(\mathfrak{f}_{i,1},f)-\partial_t\mathfrak{f}_{i,1}-p\cdot\nabla_x\mathfrak{f}_{i,1}+\mathcal{R}_i\right)(s)\|_{L^\infty_{x,p}}\nonumber\\
		&\lesssim \sup_{0 \le s \le t}\Big\{\|w_{\beta-1} f(s)\|_{L^\infty_{x,p}}\cdot \|w_{\beta-1} \mathfrak{f}_{i,1}(s)\|_{|L^\infty_{x,p}}+\|w_{\beta-2} \partial_t\mathfrak{f}_{i,1}(s)\|_{L^\infty_{x,p}}\nonumber\\
		&\quad +\|\nu^{-1}w_{\beta-1} p\cdot\nabla_x\mathfrak{f}_{i,1}(s)\|_{L^\infty_{x,p}}
		+\|\nu^{-1}w_{\beta-1} \partial_{x_i}f(s)\|_{L^\infty_{x,p}}\nonumber\\
		&\quad +\|\nu^{-1}w_{\beta-1} \f{p_i}{\sqrt{\mu}}Q(F(s),F(s))\|_{L^\infty_{x,p}} +\|\nu^{-1}w_{\beta-1} \FL \{p_if(s)\}\|_{L^\infty_{x,p}}\nonumber\\
		&\quad+\|\nu^{-1}w_{\beta-1}\{\mathbf{\Gamma}(-p_if(s),f(s))+\mathbf{\Gamma}(f(s),-p_if(s))\}\|_{L^\infty_{x,p}}\nonumber\\
		&\quad+\|\nu^{-1}w_{\beta-1} \{\mathbf{\Gamma}(-p_i\sqrt{\mu},f(s))+\mathbf{\Gamma}(f(s),-p_i\sqrt{\mu})\}\|_{L^\infty_{x,p}}\Big\}\nonumber\\
		&\lesssim1.
	\end{align}

	For $t\leq T_0$, it is clear that
	\begin{align*}
		\sup_{0 \le s \le t}\|w_{\beta-2}\mathfrak{f}_{i,2}(s)\|_{L^\infty_{x,p}}&\leq C\left(\|w_{\beta-2}\partial_{p_i}f_0\|_{L^\infty_{x,p}}+\|w_{\beta-2}\mathfrak{f}_{i,1}(0)\|_{L^\infty_{x,p}}\right)\nonumber\\
		&\quad +CT_0\sup_{0 \le s \le t}\|w_{\beta-2}\bar{\mathcal{R}}_i(s)\|_{L^\infty_{x,p}}+C\int_0^t\|w_{\beta-2}\mathfrak{f}_{i,2}(s)\|_{L^\infty_{x,p}}ds,
	\end{align*}
	which, together with Gr\"{o}nwall's inequality and \eqref{1.25}, \eqref{5.16-50} as well as \eqref{5.20-08}, yields that
	\begin{align}\label{5.19-50}
		\sup_{0 \le s \le T_0}\|w_{\beta-2}\mathfrak{f}_{i,2}(s)\|_{L^\infty_{x,p}}\leq CT_0e^{CT_0}.
	\end{align} 

	For $t\ge T_0$, similar to the deduction of \eqref{5.9-40}, we obtain
	\begin{align}\label{5.23-03}
		&\|w_{\beta-2}\mathfrak{f}_{i,2}(t)\|_{L^\infty_{x,p}}\nonumber\\
		&\leq C_{T_0}+C\int_{T_0}^te^{-\hat{\lambda}s}\|w_{\beta-2}\mathfrak{f}_{i,2}(s)\|_{L^\infty_{x,p}}ds+\f{C}{N}\sup_{T_0 \le s \le t}\|w_{\beta-2}\mathfrak{f}_{i,2}(s)\|_{L^\infty_{x,p}}\nonumber\\
		&\quad +C_{N}\int_{T_0}^te^{-2\hat{\lambda}(t-s)}\|\mathfrak{f}_{i,2}(s)\|_{L^2_{x,p}}ds.
	\end{align}
In view of \eqref{5.22-06}--\eqref{5.23-06}, one can carry out similar arguments as in Proposition \eqref{prop3.2} and Lemma \ref{lem4.5} to obtain
\begin{align}\label{5.25-000}
	&\|\mathfrak{f}_{i,2}(s)\|_{L^2_{x,p}}\leq C_{T_0}e^{-\f{\sigma_2s}{2}}\|\mathfrak{f}_{i,2}(T_0)\|_{L^2_{x,p}}\nonumber\\
	&\qquad +C \left\{\int_{T_0}^se^{-\sigma_2(s-s_1)}\left[\|\mathbf{\Gamma}(f,\mathfrak{f}_{i,2})(s_1)\|_{L^2_{x,p}}^2+\|\mathbf{\Gamma}(\mathfrak{f}_{i,2},f)(s_1)\|_{L^2_{x,p}}^2+\|\bar{\mathcal{R}}_i(s_1)\|_{L^2_{x,p}}^2\right]ds_1\right\}^{\f12}
\end{align}
for some positive constant $\sigma_2>0$. Then one has
\begin{align}\label{5.21-05}
	&\|\mathfrak{f}_{i,2}(s)\|_{L^2_{x,p}}\leq C_{T_0}e^{-\f{\sigma_2s}{2}}\|w_{\beta-2}\mathfrak{f}_{i,2}(T_0)\|_{L^\infty_{x,p}}\nonumber\\
	&\qquad+C \left\{\int_{T_0}^se^{-\sigma_2(s-s_1)}\left[e^{-2\sigma_0s_1}\|w_{\beta-2}\mathfrak{f}_{i,2}(s_1)\|_{L^\infty_{x,p}}^2+\|w_{\beta-2}\bar{\mathcal{R}}_i(s_1)\|_{L^\infty_{x,p}}^2\right]ds_1\right\}^{\f12}.
\end{align}
Plugging \eqref{5.21-05} into \eqref{5.23-03}, one obtains
	\begin{align*}%\label{5.24-05}
		\|w_{\beta-2}\mathfrak{f}_{i,2}(t)\|^2_{L^\infty_{x,p}} 
		%&\leq C_{N,T_0}+C\int_{T_0}^te^{-\hat{\lambda}s}\|w_{\beta-2}\mathfrak{f}_{i,2}(s)\|^2_{L^\infty_{x,p}}ds+\f{C}{N^2}\sup_{T_0 \le s \le t}\|w_{\beta-2}\mathfrak{f}_{i,2}(s)\|^2_{L^\infty_{x,p}}\nonumber\\
		%&\quad +C_{N}\int_{T_0}^te^{-\hat{\lambda}(t-s)}ds\cdot \int_{T_0}^tds \int_{T_0}^s e^{-3\hat{\lambda}(t-s)}e^{-\sigma_1(s-s_1)}e^{-2\sigma_0s_1}\|w_{\beta-2}\mathfrak{f}_{i,2}(s_1)\|_{L^\infty_{x,p}}^2ds_1\nonumber\\
		&\leq C_{N,T_0}+C_{N}\int_{T_0}^te^{-\hat{\lambda}s}\|w_{\beta-2}\mathfrak{f}_{i,2}(s)\|^2_{L^\infty_{x,p}}ds+\f{C}{N^2}\sup_{T_0 \le s \le t}\|w_{\beta-2}\mathfrak{f}_{i,2}(s)\|^2_{L^\infty_{x,p}}.
	\end{align*}

	Taking $N$ suitably large, then  using Gr\"{o}nwall's inequality, one gets
	\begin{align*}
		\sup_{T_0\le s\le t}\|w_{\beta-2} \mathfrak{f}_{i,2}(s)\|^2_{L^\infty_{x,p}}\le  C_{T_0}, 
	\end{align*}
	which, together with \eqref{5.19-50}, yields that
	\begin{align}\label{5.18-1}
		\|w_{\beta-2} \mathfrak{f}_{i,2}(t)\|_{L^\infty_{x,p}}\lesssim 1,\quad i=1,2,3,
	\end{align}
    for all $t\ge 0$. Combining \eqref{5.16-50} and \eqref{5.18-1}, we 
	complete the proof of Lemma \ref{lem5.2}.
\end{proof}

\vspace{0.2cm}

We are now in a position to show Theorem \ref{thm1.4}.
\begin{proof}[Proof of Theorem \ref{thm1.4}]
Due to the complexity of the proof, we divide it into four steps.\\ 
\noindent{\it Step 1. Decomposition of $f_\fc-f$.} It follows from  \eqref{2.4-10} and \eqref{5.8} that
 \begin{align*}
 	\partial_t(f_\fc-f)+\hat{p}\cdot\nabla_x(f_\fc-f)+\FL_\fc(f_\fc-f)=\mathbf{\Gamma}_\fc(f_\fc,f_\fc)-\mathbf{\Gamma}_\fc(f,f)+\mathcal{R},
 \end{align*}
 where 
 \begin{align}\label{4.51-00}
 	\mathcal{R}:=(p-\hat{p})\cdot\nabla_xf+(\FL-\FL_\fc)f+(\mathbf{\Gamma}_\fc-\mathbf{\Gamma})(f,f).
 \end{align}
Define $\mathfrak{g}:=f_\fc-f$, then  $\mathfrak{g}$ satisfies
 \begin{align*}%\label{4.50-11}
 	\begin{cases}
 		\partial_t\mathfrak{g}+\hat{p}\cdot\nabla_x\mathfrak{g}+\FL_\fc \mathfrak{g}=\mathbf{\Gamma}_\fc(\mathfrak{g},f_\fc)+\mathbf{\Gamma}_\fc(f,\mathfrak{g})+\mathcal{R},\\
 		\mathfrak{g}(t,x,p)|_{t=0}=f_{0,\fc}-f_0.
 	\end{cases}
 \end{align*}
 
It is direct to see that 
\begin{align}\label{5.30-000}
\int_{\mathbb{T}^3}\int_{\R^3} \mathcal{R}
\begin{pmatrix}
	1  \\  p \vspace{1.0ex}\\  \f{p^0-A_3}{\sqrt{A_2-A_3^2}}
\end{pmatrix}
\sqrt{J_{\mathfrak{c}}(p)} dpdx
\equiv \int_{\mathbb{T}^3}\int_{\R^3}[\FL f-\mathbf{\Gamma}(f,f)]
	\begin{pmatrix}
		1  \\  p \vspace{1.0ex}\\  \f{p^0-A_3}{\sqrt{A_2-A_3^2}}
	\end{pmatrix}
	\sqrt{J_{\mathfrak{c}}(p)} dpdx\neq \mathbf{0}.
\end{align}
%and so $\mathcal{R}$ also does not satisfy this kind of condition. 
%Hence we need to make a decomposition towards $\mathfrak{g}$.
Due to \eqref{5.30-000}, it is hard to apply the $L^2-L^\infty$ argument. For this, we introduce a  decomposition $\mathfrak{g}=\mathfrak{g}_1+\mathfrak{g}_2$ with
% Now we decompose $\mathfrak{g}$ as $\mathfrak{g}=\mathfrak{g}_1+\mathfrak{g}_2$ such that $\mathfrak{g}_1$ satisfies  
 \begin{align*}%\label{3.60-0}
 	\int_{\mathbb{T}^3}\int_{\mathbb{R}^3}\Big\{\partial_t \mathfrak{g}_1+\hat{p}\cdot\nabla_x\mathfrak{g}_1-\mathcal{R}\Big\}
 	\begin{pmatrix}
 		1  \\  p \vspace{1.0ex}\\  \f{p^0-A_3}{\sqrt{A_2-A_3^2}}
 	\end{pmatrix}
 \sqrt{J_\fc(p)}dpdx=\mathbf{0},
 \end{align*}
 which is equivalent to
 \begin{align}\label{3.60-1}
 	&\int_{\mathbb{T}^3}\int_{ \mathbb{R}^3}\partial_t\mathfrak{g}_1
 	\begin{pmatrix}
 		1  \\  p \vspace{1.0ex}\\  \f{p^0-A_3}{\sqrt{A_2-A_3^2}}
 	\end{pmatrix}
 \sqrt{J_\fc(p)}dpdx=	\int_{\mathbb{T}^3}\int_{ \mathbb{R}^3}[\FL f-\mathbf{\Gamma}(f,f)]
 	\begin{pmatrix}
 		1  \\  p \vspace{1.0ex}\\  \f{p^0-A_3}{\sqrt{A_2-A_3^2}}
 	\end{pmatrix}
 	\sqrt{J_\fc(p)}dpdx. 
 \end{align}
We also require that
\begin{align}\label{5.28}
	\int_{\mathbb{T}^3}\int_{ \mathbb{R}^3}\{\mathfrak{g}-\mathfrak{g}_{1}\}(0,x,p)
	\begin{pmatrix}
		1  \\  p \vspace{1.0ex}\\  \f{p^0-A_3}{\sqrt{A_2-A_3^2}}
	\end{pmatrix}
    \sqrt{J_\fc(p)}dpdx=\mathbf{0}.
\end{align}

Now the equation for $\mathfrak{g}_2$ takes the form
\begin{align}\label{5.70-00}
	\partial_t\mathfrak{g}_2+\hat{p}\cdot\nabla_x\mathfrak{g}_2+\FL_\fc \mathfrak{g}_2=\mathbf{\Gamma}_\fc(\mathfrak{g}_2,f_\fc)+\mathbf{\Gamma}_\fc(f,\mathfrak{g}_2)+\tilde{\mathcal{R}},
\end{align}
%\begin{align}\label{4.52-00}
%	\partial_t\mathfrak{g}_2+\hat{p}\cdot\nabla_x\mathfrak{g}_2+\FL_\fc \mathfrak{g}_2=\mathbf{\Gamma}_\fc(\mathfrak{g},f_\fc)+\mathbf{\Gamma}_\fc(f,\mathfrak{g})+\mathcal{R}-\{\partial_t\mathfrak{g}_1+\hat{p}\cdot\nabla_x\mathfrak{g}_1 + \FL_\fc \mathfrak{g}_1\}.
%\end{align}
where
\begin{align*}
	\tilde{\mathcal{R}}:=\mathbf{\Gamma}_\fc(\mathfrak{g}_1,f_\fc)+\mathbf{\Gamma}_\fc(f,\mathfrak{g}_1)+\mathcal{R}-\{\partial_t\mathfrak{g}_1+\hat{p}\cdot\nabla_x\mathfrak{g}_1+\mathbf{L}_{\mathfrak{c}}\mathfrak{g}_1\}.
\end{align*}

Based on the choice of $\mathfrak{g}_{1}$, we  have
\begin{align}\label{5.75-06}
	\int_{\mathbb{T}^3}\int_{\mathbb{R}^3}\Big(\mathbf{\Gamma}_\fc(\mathfrak{g}_2,f_\fc)+\mathbf{\Gamma}_\fc(f,\mathfrak{g}_2)+\tilde{\mathcal{R}}\Big)
	\begin{pmatrix}
		1  \\  p \vspace{1.0ex}\\  \f{p^0-A_3}{\sqrt{A_2-A_3^2}}
	\end{pmatrix}
	\sqrt{J_\fc(p)}dpdx=\mathbf{0}
\end{align}
and
\begin{align} \label{5.76-06}
	\int_{\mathbb{T}^3}\int_{ \mathbb{R}^3}\mathfrak{g}_{2}(0,x,p)
	\begin{pmatrix}
		1  \\  p \vspace{1.0ex}\\  \f{p^0-A_3}{\sqrt{A_2-A_3^2}}
	\end{pmatrix}
	\sqrt{J_\fc(p)}dpdx=\mathbf{0}.
\end{align}

\noindent{\it Step 2. Construction of $\mathfrak{g}_1$}.
  Let $\mathfrak{g}_1=\Big(\mathfrak{e}(t,x)+\mathfrak{p}(t,x)\cdot p+\mathfrak{l}(t,x)\f{p^0-A_3}{\sqrt{A_2-A_3^2}}\Big)\sqrt{J_\fc(p)}$ such that
%   \begin{align}\label{4.57-30}
%  	\|(\mathfrak{e}(0,\cdot),\mathfrak{p}(0,\cdot),\mathfrak{l}(0,\cdot))\|_{L^{\infty}_x}\lesssim\f{1}{\fc^k},\quad 	\|(\nabla_x\mathfrak{e}(0,\cdot),\nabla_x\mathfrak{p}(0,\cdot),\nabla_x\mathfrak{l}(0,\cdot))\|_{L^{\infty}_x}\lesssim\f{1}{\fc^k}.
%   \end{align}
%  Then it follows from \eqref{3.60-1} that
 \begin{align}
 	&\partial_t \mathfrak{e}=\int_{\R^3}\{\FL f-\mathbf{\Gamma}(f,f)\}\sqrt{J_\fc(p)}dp,\label{4.58-11}\\
 	&\partial_t \mathfrak{p}=\f{1}{A_1}\int_{\R^3}\{\FL f-\mathbf{\Gamma}(f,f)\}p\sqrt{J_\fc(p)}dp,\label{4.58-12}\\
 	&\partial_t \mathfrak{l}=\int_{\R^3}\{\FL f-\mathbf{\Gamma}(f,f)\}\f{p^0-A_3}{\sqrt{A_2-A_3^2}}\sqrt{J_\fc(p)}dp.\label{4.58-13}
 \end{align}
 For such choices of $\mathfrak{e}$, $\mathfrak{p}$ and $\mathfrak{l}$, it is clear that \eqref{3.60-1} holds true. For the initial data of $\mathfrak{g}_{1}$, using \eqref{1.8-10}, one can reduce the condition \eqref{5.28} to 
 \begin{align}
 	\int_{\mathbb{T}^3}\mathfrak{e}(0,x)dx&=-\int_{\mathbb{T}^3}\int_{\R^3}f_0(x,p)\sqrt{J_\fc(p)}dpdx,\label{5.32-06}\\
 	\int_{\mathbb{T}^3}\mathfrak{p}(0,x)dx&=-\frac{1}{A_1}\int_{\mathbb{T}^3}\int_{\R^3}f_0(x,p)p\sqrt{J_\fc(p)}dpdx,\label{5.33-06}\\
 	\int_{\mathbb{T}^3}\mathfrak{l}(0,x)dx&=-\int_{\mathbb{T}^3}\int_{\R^3}f_0(x,p)\f{p^0-A_3}{\sqrt{A_2-A_3^2}}\sqrt{J_\fc(p)}dpdx.\label{5.34-06}
 \end{align}
From Lemma \ref{lem2.6}, it is clear that
 \begin{align}
 	|\partial_t \mathfrak{e}|&=\Big|\int_{\R^3}\{\FL f-\mathbf{\Gamma}(f,f)\}\{\sqrt{J_\fc(p)}-\sqrt{\mu(p)}\}dp\Big|\lesssim\f{1}{\fc^{2-\epsilon}}e^{-\sigma_0t}, \label{4.46-1}\\
    |\partial_t \mathfrak{p}|&=\Big|\int_{\R^3}\{\FL f-\mathbf{\Gamma}(f,f)\}p\{\sqrt{J_\fc(p)}-\sqrt{\mu(p)}\}dp\Big|
 	\lesssim\f{1}{\fc^{2-\epsilon}}e^{-\sigma_0t}.\label{4.46-2}
 \end{align}

 For the estimate of $\partial_t \mathfrak{l}$, using Lemma \ref{lem2.2} and \eqref{2.2-30}--\eqref{2.3-30}, one has
 \begin{align*} 
   \f{1}{\mathfrak{c}\sqrt{A_2-A_3^2}}=\sqrt{\f23}+O(\mathfrak{c}^{-2}),\quad \mathfrak{c}^2-\mathfrak{c}A_3+\f32=O(\mathfrak{c}^{-2}),
 \end{align*} 
 which, together with Lemma \ref{lem2.6} and  $\fc p^0-\fc^2-\f{|p|^2}{2}=-\f{|p|^4}{2(p^0+\fc)^2}$, yields that
% \begin{align}\label{4.47-1}
% 	\left|\int_{\R^3}\FL f(\fc p^0-\fc^2)\sqrt{J_\fc(p)}dp\right|&=\left|\int_{\R^3}\FL f\{(\fc p^0-\fc^2)\sqrt{J_\fc(p)}-\f{|p|^2}{2}\sqrt{\mu(p)}\}dp\right|\nonumber\\
% 	&=\left|\int_{\R^3}\FL f\Big\{(\fc p^0-\fc^2-\f{|p|^2}{2})\sqrt{J_\fc(p)}+\f{|p|^2}{2}(\sqrt{J_\fc(p)}-\sqrt{\mu(p)})\Big\}dp\right|\nonumber\\
% 	&\lesssim\f{1}{\fc^{2-\epsilon}}e^{-\sigma_0t},
% \end{align}
%where we used
%\begin{align*}
%	\fc p^0-\fc^2-\f{|p|^2}{2}=-\f{|p|^4}{2(p^0+\fc)^2}.
%\end{align*}
%Lemma \ref{lem2.2} gives that
%\begin{align*}
%	\left|\f{1}{\sqrt{A_2-A_3^2}}-\sqrt{\f23}\fc\right|\lesssim\f{1}{\fc},\quad |\fc(\fc-A_3)+\f32|\lesssim \f{1}{\fc^2},
%\end{align*}
%which, together with \eqref{4.47-1}, yields that
\begin{align}\label{4.48-1}
%	&\Big|\int_{\R^3}\{\FL f-\mathbf{\Gamma}(f,f)\}\f{p^0-A_3}{\sqrt{A_2-A_3^2}}\sqrt{J_\fc(p)}dp\Big|\nonumber\\
	|\partial_t \mathfrak{l}|&\le\frac{1}{\mathfrak{c}\sqrt{A_2-A_3^2}}\Big|\int_{\R^3}\{\FL f-\mathbf{\Gamma}(f,f)\}\Big(\mathfrak{c}p^0-\mathfrak{c}^2-\frac{|p|^2}{2}\Big)\sqrt{J_\fc(p)}dp\Big|\nonumber\\
	&\quad + \frac{1}{\mathfrak{c}\sqrt{A_2-A_3^2}}\Big|\int_{\R^3}\{\FL f-\mathbf{\Gamma}(f,f)\}\Big(\mathfrak{c}^2-\mathfrak{c}A_3+\frac{3}{2}\Big)\sqrt{J_\fc(p)}dp\Big|\nonumber\\
	&\quad + \frac{1}{\mathfrak{c}\sqrt{A_2-A_3^2}}\Big|\int_{\R^3}\{\FL f-\mathbf{\Gamma}(f,f)\}\Big(\frac{|p|^2}{2}-\frac{3}{2}\Big)\Big(\sqrt{J_\fc(p)}-\sqrt{\mu(p)}\Big)dp\Big|\nonumber\\
%	\leq&\f{1}{\fc\sqrt{A_2-A_3^2} }\Big|\int_{\R^3}\{\FL f-\mathbf{\Gamma}(f,f)\}(\fc p^0-\fc^2-\f{|p|^2}{2})\sqrt{J_\fc(p)}dp\right|\nonumber\\
%	&\quad+\f{1}{\fc\sqrt{A_2-A_3^2} }\left| \int_{\R^3}\{\FL f-\mathbf{\Gamma}(f,f)\}(\fc^2-\fc A_3+\f{|p|^2}{2})\{\sqrt{J_\fc(p)}-\sqrt{\mu}(p)\}dp\right| \nonumber\\
	&\lesssim \f{1}{\fc^{2-\epsilon}}e^{-\sigma_0t},
\end{align}
where we have used the fact that 
\begin{align*}
	\int_{\R^3}\{\FL f-\mathbf{\Gamma}(f,f)\}\Big(\frac{|p|^2}{2}-\frac{3}{2}\Big)\sqrt{\mu(p)}dp=0.
\end{align*}
 Hence it follows from \eqref{4.46-1}--\eqref{4.48-1} that 
 \begin{align}\label{4.56-00}
 	\|\partial_t(\mathfrak{e},\mathfrak{p},\mathfrak{l})(t,\cdot)\|_{L^{\infty}_x}\lesssim \f{1}{\fc^{2-\epsilon}}e^{-\sigma_0t}.
 \end{align}
% with which we can construct $\mathfrak{e},\mathfrak{p},\mathfrak{l}$ with initial data
% \begin{align*}
% 	|(\mathfrak{e},\mathfrak{p},\mathfrak{l})|(0)\lesssim\f{1}{\fc^k},
% \end{align*}

Similarly, it follows from \eqref{1.18-10} and \eqref{5.32-06}--\eqref{5.34-06} that
\begin{align}\label{5.39-06}
		\Big|\int_{\mathbb{T}^3}\mathfrak{e}(0,x)dx\Big|+\Big|\int_{\mathbb{T}^3}\mathfrak{p}(0,x)dx\Big|+\Big|\int_{\mathbb{T}^3}\mathfrak{l}(0,x)dx\Big| \lesssim \f{1}{\fc^{2-\epsilon}}.
\end{align}
 In view of \eqref{5.39-06}, we can choose  $\mathfrak{e}(0,x)$, $\mathfrak{p}(0,x)$ and $\mathfrak{l}(0,x)$ which are only functions of $\mathfrak{c}$ 
 \begin{align*}%\label{4.50-1} 	
 \mathfrak{e}(0,x)=\mathfrak{e}^0(\mathfrak{c}),\quad \mathfrak{p}(0,x)=\mathfrak{p}^0(\mathfrak{c}),\quad
 \mathfrak{l}(0,x)=\mathfrak{l}^0(\mathfrak{c}),
 \end{align*} 
so that \eqref{5.32-06}--\eqref{5.34-06} hold true. Consequently, one has
\begin{align}\label{5.41-06}
	|\mathfrak{e}(0,x)|+|\mathfrak{p}(0,x)|+|\mathfrak{l}(0,x)|\lesssim \f{1}{\fc^{2-\epsilon}},\quad \nabla_x(\mathfrak{e}(0,x),\mathfrak{p}(0,x),\mathfrak{l}(0,x))\equiv \mathbf{0}.
\end{align}

 Combining \eqref{4.56-00} and $\eqref{5.41-06}_1$ gives
 \begin{align}\label{4.56-01}
 	\|(\mathfrak{e},\mathfrak{p},\mathfrak{l})\|_{L^{\infty}_{t,x}}\lesssim \f{1}{\fc^{2-\epsilon}}.
 \end{align}
Applying $\nabla_x$ to \eqref{4.58-11}--\eqref{4.58-13}, then using \eqref{4.49-1} and Lemma \ref{lem5.1}, we can further obtain
\begin{align}\label{4.57-00}
	\|\nabla_x(\mathfrak{e},\mathfrak{p},\mathfrak{l})\|_{L^{\infty}_{t,x}}\lesssim \f{1}{\fc^{2-\epsilon}}.
\end{align}

\noindent{\it Step 3. Estimate on $\mathcal{R}$.} 
Noting \eqref{2.5} and \eqref{2.21-03}, one has
 \begin{align}\label{4.67-30}
 	(\FL-\FL_\fc)f=&\int_{\R^3}\int_{\mathbb{S}^2}\Big[\mathcal{K}_\infty(p,q,\omega)\mu(q)-\mathcal{K}_\fc(p,q,\omega)J_\fc(q)\Big]d\omega dq \cdot f(p)\nonumber\\
 	&+\int_{\R^3}\int_{\mathbb{S}^2}\Big[\mathcal{K}_\infty(p,q,\omega)\sqrt{\mu(p)\mu(q)}-\mathcal{K}_\fc(p,q,\omega)\sqrt{J_\fc(p)J_\fc(q)}\Big]f(q)d\omega dq\nonumber\\
 	&-\int_{\R^3}\int_{\mathbb{S}^2}\Big[\mathcal{K}_\infty(p,q,\omega)\sqrt{\mu(q)\mu(\bar{q}')}f(\bar{p}')-\mathcal{K}_\fc(p,q,\omega)\sqrt{J_\fc(q)J_\fc(q')}f(p')\Big]d\omega dq\nonumber\\
 	&-\int_{\R^3}\int_{\mathbb{S}^2}\Big[\mathcal{K}_\infty(p,q,\omega)\sqrt{\mu(q)\mu(\bar{p}')}f(\bar{q}')-\mathcal{K}_\fc(p,q,\omega)\sqrt{J_\fc(q)J_\fc(p')}f(q')\Big]d\omega dq\nonumber\\
 	:=&\mathcal{M}_1+\mathcal{M}_2+\mathcal{M}_3+\mathcal{M}_4,
 \end{align}
where $(\bar{p}',\bar{q}')$ and $(p',q')$  are defined in \eqref{1.16-0} and \eqref{2.16-0}, respectively. 
Using \eqref{1.30} and Lemmas \ref{lem2.4} $\&$ \ref{lem2.6}, one can obtain  
\begin{align}\label{4.48-0}
	|w_{\beta-6}\mathcal{M}_1|&\leq\Big|\int_{\R^3}\int_{\mathbb{S}^2}[\mathcal{K}_\infty(p,q,\omega)-\mathcal{K}_\fc(p,q,\omega)]\mu(q)d\omega dq\cdot w_{\beta-6}(p)f(p)\Big| \nonumber\\
	&\quad +\Big|\int_{\R^3}\int_{\mathbb{S}^2}\mathcal{K}_\fc(p,q,\omega)[\mu(q)-J_\fc(q)]d\omega dq\cdot w_{\beta-6}(p)f(p)\Big|\nonumber\\
	&\lesssim\f{1}{\fc^{2-\epsilon}}\|w_\beta f\|_{L^\infty_{x,p}}\lesssim \f{1}{\fc^{2-\epsilon}},
\end{align}
and
\begin{align}\label{4.69-30}
	|w_\beta\mathcal{M}_2|&\leq\Big|\int_{\R^3}\int_{\mathbb{S}^2}[\mathcal{K}_\infty(p,q,\omega)-\mathcal{K}_\fc(p,q,\omega)]w_\beta(p)\sqrt{\mu(p)\mu(q)}f(q)d\omega dq\Big|\nonumber\\
	&\quad +\Big|\int_{\R^3}\int_{\mathbb{S}^2}\mathcal{K}_\fc(p,q,\omega)[\sqrt{\mu(p)\mu(q)}-\sqrt{J_\fc(p)J_\fc(q)}]w_\beta(p)f(q)d\omega dq\Big|\nonumber\\
	&\lesssim\f{1}{\fc^{2-\epsilon}}.
\end{align}
%For $\mathcal{M}_3+\mathcal{M}_4$, referring to \cite{Wang-Xiao}, we have
%\begin{align*}
%	|w_\beta(\mathcal{M}_3+\mathcal{M}_4)|&=\left|w_\beta(p)\int_{\R^3}(k_{\fc 2}-k_2)(p,q)f(q)dq\right|\nonumber\\
%	&\leq\int_{\R^3}(1+|p-q|)^\beta |k_{\fc_2}-k_2(p,q)|dq\cdot \|w_\beta f\|_{L^\infty_{x,p}}\nonumber\\
%	&\lesssim \fc^{-\f38}e^{-\sigma_0t}.
%\end{align*}
%Here $\mathcal{M}_{2}+\mathcal{M}_3+\mathcal{M}_4=(\mathbf{K}-\mathbf{K}_\fc)f$,

It follows from Lemma \ref{lem2.3} and \eqref{2.12-20} that
\begin{align*}
	|\mathcal{M}_{2}+\mathcal{M}_3+\mathcal{M}_4|&=|(\mathbf{K}-\mathbf{K}_\fc)f|\nonumber\\
	&\lesssim \int_{\R^3}\{|k(p,q)|+|k_\fc(p,q)|\}\f{1}{(1+|q|)^{\beta}}\cdot| w_{\beta}(q)f(q)|dq\nonumber\\
	&\lesssim  \f{1}{(1+|p|)^{\beta}}\|w_\beta f\|_{L^\infty_{x,p}}\lesssim \f{1}{(1+|p|)^{\beta}}.
\end{align*}
For $|p|\geq\fc^{\f12}$, we obtain
\begin{align}\label{4.69-40}
	w_{\beta-4}|\mathcal{M}_{2}+\mathcal{M}_3+\mathcal{M}_4|\lesssim  \f{1}{\fc^{2}}.
\end{align}

Next we consider $|p|\le \fc^{\f12}$. It holds that
\begin{align}\label{4.71-03}
	\mathcal{M}_3&=\int_{\R^3}\int_{\mathbb{S}^2}\Big(\mathcal{K}_\fc(p,q,\omega)-\mathcal{K}_\infty(p,q,\omega)\Big)\sqrt{J_\fc(q)J_\fc(q')}f(p')d\omega dq\nonumber\\
	&\quad +\int_{\R^3}\int_{\mathbb{S}^2}\mathcal{K}_\infty(p,q,\omega)\Big(\sqrt{J_\fc(q)J_\fc(q')}-\sqrt{\mu(q)\mu(q')}\Big)f(p')d\omega dq\nonumber\\
	&\quad +\int_{\R^3}\int_{\mathbb{S}^2}\mathcal{K}_\infty(p,q,\omega)\Big(\sqrt{\mu(q)\mu(q')}-\sqrt{\mu(q)\mu(\bar{q}')}\Big)f(p')d\omega dq\nonumber\\
	&\quad +\int_{\R^3}\int_{\mathbb{S}^2}\mathcal{K}_\infty(p,q,\omega)\sqrt{\mu(q)\mu(\bar{q}')}\Big(f(p')-f(\bar{p}')\Big)d\omega dq\nonumber\\
	&:=\mathcal{M}_{3,1}+\mathcal{M}_{3,2}+\mathcal{M}_{3,3}+\mathcal{M}_{3,4}.
\end{align}
It follows from \eqref{1.30} that
\begin{align}\label{4.72-03}
	|w_{\beta-6} \mathcal{M}_{3,1}|\lesssim w_{\beta-6}(p)\int_{\R^3}\int_{\mathbb{S}^2}\f{(1+|p|+|q|)^6}{\fc^2}\sqrt{J_\fc(q)J_\fc(q')}w_\beta ^{-1}(p')d\omega dq\cdot \|w_\beta f\|_{L^\infty_{x,p}}\lesssim \f{1}{\fc^2},
\end{align}
where we have used the fact (see the proof of  \cite[Lemma 6.1]{Wang-Xiao}) that 
\begin{align}\label{4.49-0}
\sqrt{J_\fc(q')}w_\beta ^{-1}(p')\lesssim w_\beta ^{-1}(q')w_\beta ^{-1}(p')\lesssim\f{1}{(1+|p|)^\beta}.
\end{align}
Similarly, one has
\begin{align}\label{4.74-03}
	|w_{\beta-1}\mathcal{M}_{3,2}|&\lesssim \int_{\R^3}w_{\beta-1}(p)w_{\beta}^{-1}(p')\mathcal{K}_\infty(p,q,\omega)\Big|\sqrt{J_\fc(q)J_\fc(q')}-\sqrt{\mu(q)\mu(q')}\Big|dq\cdot \|w_\beta f\|_{L^\infty_{x,p}}\nonumber\\
&\lesssim\f{1}{\fc^{2-\epsilon}}.
\end{align}

For $|p|\leq \fc^{\f12}$, a direct calculation shows that
\begin{align*}
	|p|\lesssim|p'|+|q'|\lesssim|p'|+|\bar{q}'|+\f{(|p|+|q|)^3}{\fc^2}\lesssim|p'|+|\bar{q}'|+(1+|q|)^3,
\end{align*}
which implies that
\begin{align}\label{4.69-3}
	w_{\beta}(p)\lesssim w_{\beta}(p')w_{\beta}(\bar{q}')w_{3\beta}(q).
\end{align}
It follows from Lemma \ref{lem2.4} that
\begin{align}\label{4.70-3}
	|\sqrt{\mu(q')}-\sqrt{\mu(\bar{q}')}|\lesssim &e^{-\f{|q'|^2}{8}}\cdot\Big|e^{-\f{|q'|^2}{8}}-e^{-\f{|\bar{q}'|^2}{8}}\Big|+e^{-\f{|\bar{q}'|^2}{8}}\cdot \Big|e^{-\f{|q'|^2}{8}}-e^{-\f{|\bar{q}'|^2}{8}}\Big|\nonumber\\
	\lesssim& (e^{-\f{|q'|^2}{8}}+e^{-\f{|\bar{q}'|^2}{8}})\cdot\f{(|p|+|q|)^3}{\fc^2}.
\end{align}
Noting \eqref{4.49-0} and \eqref{4.69-3}--\eqref{4.70-3}, we obtain
\begin{align}\label{4.77-03}
	|w_{\beta-4} \mathcal{M}_{3,3}|\lesssim\f{1}{\fc^2},\quad |p|\le \fc^{\f12}.
\end{align}

For $\mathcal{M}_{3,4}$, it holds that
\begin{align*}
	\mathcal{M}_{3,4}=&\int_{|q|\geq\fc^{\f12}}\int_{\mathbb{S}^2}\mathcal{K}_\infty(p,q,\omega)\sqrt{\mu(q)\mu(\bar{q}')}[f(p')-f(\bar{p}')]d\omega dq\nonumber\\
	&+\int_{|q|\le  \fc^{\f12}}\int_{\mathbb{S}^2}\mathcal{K}_\infty(p,q,\omega)\sqrt{\mu(q)\mu(\bar{q}')}[f(p')-f(\bar{p}')]d\omega dq.
%	&\lesssim\f{1}{\fc^{2}}e^{-\sigma_0t}+\int_{|q|\lesssim \fc^{\f12}}\int_{\mathbb{S}^2}\mathcal{K}_\infty(p,q,\omega)\sqrt{\mu(q)\mu(\bar{q}')}[f(p')-f(\bar{p}')]d\omega dq.
\end{align*}
It is clear that
\begin{align}\label{4.71-30}
	\Big|w_{\beta-5}(p)\int_{|q|\geq\fc^{\f12}}\int_{\mathbb{S}^2}\mathcal{K}_\infty(p,q,\omega)\sqrt{\mu(q)\mu(\bar{q}')}[f(p')-f(\bar{p}')]d\omega dq\Big|\lesssim \f{1}{\fc^{2}},\quad |p|\le \mathfrak{c}^{\frac{1}{2}}.
\end{align}
For $|p|\le \fc^{\f12}$, $|q|\le \fc^{\f12}$, we have from Lemma \ref{lem2.4} that  $|p'-\bar{p}'|\lesssim \f{(|p|+|q|)^3}{\fc^2}\lesssim \fc^{-\f12}$. It follows that $|p'-\bar{p}'|\le \frac{1}{2}$ for suitably large $\mathfrak{c}$. Thus one has 
%then there is a positive constant $c_1$ which is independent of $\mathfrak{c}$, such that
\begin{align}\label{4.72-30}
	1+|\bar{p}'|-|p'-\bar{p}'|\geq \frac{1}{2} (1+|\bar{p}'|).
%	\quad 1+|p'|-|p'-\bar{p}'|\geq \frac{1}{2} (1+|p'|).
\end{align}
%hence for $\mathcal{M}_{3,4}$ with $|p|,|q|\lesssim \fc^{\f12}$, recall Lemma \ref{lem5.2}, 
By the mean value theorem, there exists $p_{\vartheta}=\bar{p}'+\vartheta(p'-\bar{p}')$ with $0<\vartheta<1$, such that 
%satisfying $|p_\xi|\geq \min\{|\bar{p}'|-|p'-\bar{p}'|,|p'|-|p'-\bar{p}'|\}$, s.t.
\begin{align*}
	|f(\bar{p}')-f(p')|=|\nabla_pf(p_{\vartheta})|\cdot |\bar{p}'-p'|= w_{\beta-2}^{-1}(p_{\vartheta})|w_{\beta-2}(p_{\vartheta})\nabla_pf(p_{\vartheta})|\cdot |\bar{p}'-p'|.
\end{align*}
It follows from \eqref{4.72-30} that $w_{\beta-2}^{-1}(p_{\vartheta})\lesssim w_{\beta-2}^{-1}(\bar{p}')$. Using Lemma \ref{lem5.2}, one has
\begin{align*}
	&\Big|w_{\beta-6}(p)\int_{|q|\le\fc^{\f12}}\int_{\mathbb{S}^2}\mathcal{K}_\infty(p,q,\omega)\sqrt{\mu(q)\mu(\bar{q}')}[f(p')-f(\bar{p}')]d\omega dq\Big|\nonumber\\
	&\lesssim \int_{|q|\le\fc^{\f12}}w_{\beta-6}(p)|p-q|\sqrt{\mu(q)\mu(\bar{q}')}w_{\beta-2}^{-1}(\bar{p}')\f{(|p|+|q|)^3}{\fc^2}d\omega dq\cdot\|w_{\beta-2} \nabla_pf\|_{L^\infty_{x,p}} \lesssim\f{1}{\fc^2},
\end{align*}
which, together with \eqref{4.71-30}, yields that
\begin{align}\label{4.50-0}
	|w_{\beta-6} \mathcal{M}_{3,4}|\lesssim\f{1}{\fc^2},\quad |p|\le \mathfrak{c}^{\frac{1}{2}}.
\end{align}
Combining \eqref{4.71-03}--\eqref{4.72-03}, \eqref{4.74-03}, \eqref{4.77-03} and \eqref{4.50-0}, we have
\begin{align}\label{4.81-03}
	|w_{\beta-6} \mathcal{M}_{3}|  \lesssim\f{1}{\fc^{2-\epsilon}},\quad |p|\le \mathfrak{c}^{\frac{1}{2}}.
\end{align}

Similar as above arguments, one can also obtain
\begin{align}\label{4.82-03}
	|w_{\beta-6} \mathcal{M}_{4}|  \lesssim\f{1}{\fc^{2-\epsilon}},\quad |p|\le \mathfrak{c}^{\frac{1}{2}}.
\end{align}
Consequently, we have from \eqref{4.67-30}--\eqref{4.69-40} and \eqref{4.81-03}--\eqref{4.82-03} that
\begin{align}\label{4.50-2}
	|w_{\beta-6}(\FL-\FL_\fc)f|\lesssim\f{1}{\fc^{2-\epsilon}}.
\end{align}

Next, we estimate $	w_{\beta-6}(\mathbf{\Gamma}-\mathbf{\Gamma}_\fc)(f,f)$. It holds that
 \begin{align*}
 	(\mathbf{\Gamma}-\mathbf{\Gamma}_\fc)(f,f)=&\int_{\R^3}\int_{\mathbb{S}^2}\mathcal{K}_{\infty}(p,q,\omega)\sqrt{\mu(q)}[f(\bar{p}')f(\bar{q}')-f(p)f(q)]d\omega dq\nonumber\\
 	&\quad-\int_{\R^3}\int_{\mathbb{S}^2}\mathcal{K}_\fc(p,q,\omega)\sqrt{J_\fc(q)}[f(p')f(q')-f(p)f(q)]d\omega dq.
 \end{align*}
For $|p|\geq \fc^{\f12}$, one has
\begin{align}\label{4.56-1}
	|w_{\beta-5}(\mathbf{\Gamma}-\mathbf{\Gamma}_\fc)(f,f)|&\lesssim w_{\beta-5}(p)(1+|p|)\big\{w_\beta^{-1}(\bar{p}')w_\beta^{-1}(\bar{q}')+w_\beta^{-1}(p')w_\beta^{-1}(q')\big\}\cdot\|w_{\beta}f\|^2_{L^{\infty}_{x,p}}\nonumber\\
	&\lesssim\f{1}{(1+|p|)^4}\lesssim\f{1}{\fc^2}.
\end{align}
For $|p|\le \fc^{\f12}$, we depart $(\mathbf{\Gamma}-\mathbf{\Gamma_\fc})(f,f)$ as
\begin{align*}
	(\mathbf{\Gamma}-\mathbf{\Gamma}_\fc)(f,f)=&\int_{|q|\geq\fc^{\f12}}\int_{\mathbb{S}^2}\mathcal{K}_{\infty}(p,q,\omega)\sqrt{\mu(q)}[f(\bar{p}')f(\bar{q}')-f(p)f(q)]d\omega dq\nonumber\\
	&\quad-\int_{|q|\geq \fc^{\f12}}\int_{\mathbb{S}^2}\mathcal{K}_\fc(p,q,\omega)\sqrt{J_\fc(q)}[f(p')f(q')-f(p)f(q)]d\omega dq\nonumber\\
	&\quad+\int_{|q|\le \fc^{\f12}}\int_{\mathbb{S}^2}[\mathcal{K}_\infty(p,q,\omega)\sqrt{\mu(q)}-\mathcal{K}_\fc(p,q,\omega)\sqrt{J_\fc(q)}]f(p')f(q')d\omega dq\nonumber\\
	&\quad +\int_{|q|\le \fc^{\f12}}\int_{\mathbb{S}^2}\mathcal{K}_\infty(p,q,\omega)\sqrt{\mu(q)}f(\bar{p}')[f(\bar{q}')-f(q')]d\omega dq\nonumber\\
	&\quad +\int_{|q|\le \fc^{\f12}}\int_{\mathbb{S}^2}\mathcal{K}_\infty(p,q,\omega)\sqrt{\mu(q)}[f(\bar{p}')-f(p')]f(q')d\omega dq\nonumber\\
	&\quad +\int_{|q|\le \fc^{\f12}}\int_{\mathbb{S}^2}[\mathcal{K}_\fc(p,q,\omega)\sqrt{J_\fc(q)}-\mathcal{K}_\infty(p,q,\omega)\sqrt{\mu(q)}]f(p)f(q)d\omega dq\nonumber\\
	:=&\sum_{i=1}^6\mathcal{E}_i.
\end{align*}
It is clear that
\begin{align}\label{4.66-0}
	|w_{\beta-1}(\mathcal{E}_1,\mathcal{E}_2)|\lesssim\f{1}{\fc^2}  \quad \mbox{and} \quad |w_{\beta-6}(\mathcal{E}_3,\mathcal{E}_6)|\lesssim\f{1}{\fc^{2-\epsilon}}.
\end{align}
%When $|p|,|q|\lesssim\fc^{\f12}$, one can have decomposition as
%\begin{align}	
%\mathcal{E}_3+\mathcal{E}_4=&\int_{|q|\le \fc^{\f12}}\int_{\mathbb{S}^2}[\mathcal{K}_\infty(p,q,\omega)\sqrt{\mu(q)}-\mathcal{K}_\fc(p,q,\omega)\sqrt{J_\fc(q)}]f(p')f(q')d\omega dq\nonumber\\
%&\quad +\int_{|q|\le \fc^{\f12}}\int_{\mathbb{S}^2}\mathcal{K}_\infty(p,q,\omega)\sqrt{\mu(q)}f(\bar{p}')[f(\bar{q}')-f(q')]d\omega dq\nonumber\\
%&\quad +\int_{|q|\le \fc^{\f12}}\int_{\mathbb{S}^2}\mathcal{K}_\infty(p,q,\omega)\sqrt{\mu(q)}[f(\bar{p}')-f(p')]f(q')d\omega dq\nonumber\\
%&\quad +\int_{|q|\le \fc^{\f12}}\int_{\mathbb{S}^2}[\mathcal{K}_\fc(p,q,\omega)\sqrt{J_\fc(q)}-\mathcal{K}_\infty(p,q,\omega)\sqrt{\mu(q)}]f(p)f(q)d\omega dq\nonumber\\
%:=&\mathcal{G}_1+\mathcal{G}_2+\mathcal{G}_3+\mathcal{G}_4.
%\end{align}
%and
%\begin{align}\label{4.55-1}
%	|w_{\beta-6}\mathcal{E}_3|\lesssim\f{1}{\fc^{2-\epsilon}},\quad |w_{\beta-6}\mathcal{E}_6|\lesssim\f{1}{\fc^{2-\epsilon}}.
%\end{align}
By similar arguments as in the estimate of $\mathcal{M}_{3,4}$, we have
\begin{align}\label{4.57-0}
	|w_{\beta-6}(\mathcal{E}_4,\mathcal{E}_5)|\lesssim \f{1}{\fc^2}.%,\quad |w_{\beta-6}\mathcal{E}_5|&\lesssim \f{1}{\fc^2}. 
\end{align}
Then it follows from \eqref{4.66-0}--\eqref{4.57-0} that
\begin{align*}
	|w_{\beta-6}(\mathbf{\Gamma}-\mathbf{\Gamma}_\fc)(f,f)|\lesssim \f{1}{\fc^{2-\epsilon}},\quad |p|\le \mathfrak{c}^{\frac{1}{2}},
\end{align*}
which, together with \eqref{4.56-1}, yields that
\begin{align}\label{4.57-1}
	|w_{\beta-6}(\mathbf{\Gamma}-\mathbf{\Gamma}_\fc)(f,f)|\lesssim \f{1}{\fc^{2-\epsilon}}.
\end{align}

For the remaining term $(p-\hat{p})\cdot\nabla_x f$  on the RHS of \eqref{4.51-00}, it follows from Lemma \ref{lem5.1} that
\begin{align*}%\label{4.58-1}
	|w_{\beta-4}(p-\hat{p})\cdot\nabla_x f|=\Big|w_{\beta-4}\f{|p|^2}{p^0(p^0+\fc)}p\cdot\nabla_xf\Big| \lesssim  \f{1}{\fc^2} \|w_{\beta-1}\nabla_xf\|_{L^\infty_{x,p}} \lesssim\f{1}{\fc^2},
\end{align*}
which, together with \eqref{4.51-00}, \eqref{4.50-2} and \eqref{4.57-1}, yields that
\begin{align}\label{4.58-00}
	|w_{\beta-6}\mathcal{R}|\lesssim \f{1}{\fc^{2-\epsilon}}.
\end{align}

\noindent{\it Step 4. Newtonian limit.} Recall that 
\begin{align*}
	\tilde{\mathcal{R}}:=\mathbf{\Gamma}_\fc(\mathfrak{g}_1,f_\fc)+\mathbf{\Gamma}_\fc(f,\mathfrak{g}_1)+\mathcal{R}-\{\partial_t\mathfrak{g}_1+\hat{p}\cdot\nabla_x\mathfrak{g}_1+\mathbf{L}_{\mathfrak{c}}\mathfrak{g}_1\}.
\end{align*}
It follows from \eqref{1.29}--\eqref{1.30}, \eqref{4.56-00}, \eqref{4.56-01}--\eqref{4.57-00} and \eqref{4.58-00} that
\begin{align}\label{5.60-1}
	|w_{\beta-6}\tilde{\mathcal{R}}|\lesssim\f{1}{\fc^{k-\epsilon}}.
\end{align}

The mild form of equation \eqref{5.70-00} can be written as
\begin{align}\label{5.60}
	&w_{\beta-6}(p)\mathfrak{g}_2(t,x,p)\nonumber\\
	&=e^{-\int_0^t\nu_{f_\fc}(s,x,p)ds}w_{\beta-6}(p)\mathfrak{g}_{2}(0,x-\hat{p}t,p)\nonumber\\
	&\quad +\int_0^te^{-\int_s^t\nu_{f_\fc}(s_1,x,p)ds_1}w_{\beta-6}(p)\mathbf{K}_\fc\mathfrak{g}_2\left(s,x-\hat{p}(t-s),p\right)ds\nonumber\\
	&\quad +\int_0^t e^{-\int_s^t\nu_{f_\fc}(s_1,x,p)ds_1} w_{\beta-6}(p)\{\mathbf{\Gamma}^{+}_{\mathfrak{c}}(\mathfrak{g}_2,f_{\mathfrak{c}})+\mathbf{\Gamma}^{+}_{\mathfrak{c}}(f,\mathfrak{g}_2)-\mathbf{\Gamma}^{-}_{\mathfrak{c}}(f,\mathfrak{g}_2)\}\left(s,x-\hat{p}(t-s),p\right)ds\nonumber\\
	&\quad +\int_0^t e^{-\int_s^t\nu_{f_\fc}(s_1,x,p)ds_1} w_{\beta-6}(p)\tilde{\mathcal{R}}\left(s,x-\hat{p}(t-s),p\right)ds,\quad\text{for }t\leq T_1,
\end{align}
and 
\begin{align}\label{5.61}
	&w_{\beta-6}(p)\mathfrak{g}_2(t,x,p)\nonumber\\
	&=e^{-\int_{T_1}^t\nu_{f_\fc}(s,x,p)ds}w_{\beta-6}(p)\mathfrak{g}_{2}(T_1,x-\hat{p}(t-T_1),p)\nonumber\\
	&\quad +\int_{T_1}^te^{-\int_s^t\nu_{f_\fc}(s_1,x,p)ds_1}w_{\beta-6}(p)\mathbf{K}_\fc\mathfrak{g}_2\left(s,x-\hat{p}(t-s),p\right)ds\nonumber\\
	&\quad +\int_{T_1}^t e^{-\int_s^t\nu_{f_\fc}(s_1,x,p)ds_1} w_{\beta-6}(p)\{\mathbf{\Gamma}^{+}_{\mathfrak{c}}(\mathfrak{g}_2,f_{\mathfrak{c}})+\mathbf{\Gamma}^{+}_{\mathfrak{c}}(f,\mathfrak{g}_2)-\mathbf{\Gamma}^{-}_{\mathfrak{c}}(f,\mathfrak{g}_2)\}\left(s,x-\hat{p}(t-s),p\right)ds\nonumber\\
	&\quad +\int_{T_1}^t e^{-\int_s^t\nu_{f_\fc}(s_1,x,p)ds_1} w_{\beta-6}(p)\tilde{\mathcal{R}}\left(s,x-\hat{p}(t-s),p\right)ds,\quad \text{for }t\geq T_1,
\end{align}
where
\begin{align*}
	&\nu_{f_\fc}(t,x,p):=\iint_{\R^3\times \mathbb{S}^2}\mathcal{K}_\fc(p,q,\omega)\Big(J_\fc(q)+\sqrt{J_\fc(q)}f_\fc(t,x,q)\Big)d\omega dq\geq 0.
%	&\mathbf{\Gamma}_{+,\fc}=\iint_{\R^3\times \mathbb{S}^2}\mathcal{K}_\fc(p,q)\sqrt{J_\fc(q)}\mathfrak{g}_2(p')f_\fc(q')dq+\iint_{\R^3\times \mathbb{S}^2}\mathcal{K}_\fc(p,q)\sqrt{J_\fc(q)}\mathfrak{g}_2(q')f_\fc(p')dq,
%\end{align*}
%and
%\begin{align*}
%	&\mathbf{\Gamma}_{-,2,\fc}=\iint_{\R^3\times \mathbb{S}^2}\mathcal{K}_\fc(p,q)\sqrt{J_c(q)}f_\fc(p)\mathfrak{g}_2(q)dq.
\end{align*}
Noting \eqref{1.29}, we choose $T_1$ suitably large so that $\nu_{f_\fc}(t,x,p)\geq\f12\nu_0$ for $t\geq T_1$. 

A direct calculation shows that
\begin{align*}
	&w_{\beta-6}\Big(|\mathbf{\Gamma}^{+}_{\mathfrak{c}}(\mathfrak{g}_2,f_{\mathfrak{c}})(s)|+|\mathbf{\Gamma}^{+}_{\mathfrak{c}}(f,\mathfrak{g}_2)(s)|+|\mathbf{\Gamma}^{-}_{\mathfrak{c}}(f,\mathfrak{g}_2)(s)|\Big)\nonumber\\
	&\leq \Big(\|w_{\beta-6}f_\fc(s)\|_{L^\infty_p}+\|w_{\beta-5}f(s)\|_{L^\infty_p}\Big)\|w_{\beta-6}\mathfrak{g}_2(s)\|_{L^\infty_{p}}.
\end{align*}
For $t\leq T_1$, it follows from \eqref{5.60} that
 \begin{align*}
 	\sup_{0 \le s \le t}\|w_{\beta-6}\mathfrak{g}_2(s)\|_{L^\infty_{x,p}}&\leq C\left(\|w_{\beta-6}\mathfrak{g}(0)\|_{L^\infty_{x,p}}+\|w_{\beta-6}\mathfrak{g}_1(0)\|_{L^\infty_{x,p}}\right)\nonumber\\
 	&\quad +C\int_0^t\|w_{\beta-6}\mathfrak{g}_2(s)\|_{L^\infty_{x,p}}ds+CT_1\sup_{0\leq s\leq t}\|w_{\beta-6}\tilde{\mathcal{R}}(s)\|_{L^\infty_{x,p}},
 \end{align*}
which, together with Gr\"{o}nwall's inequality and \eqref{1.26}, \eqref{5.41-06}, \eqref{5.60-1}, yields that
\begin{align}\label{5.67-03}
	\sup_{0 \le s \le T_1}\|w_{\beta-6}\mathfrak{g}_2(s)\|_{L^\infty_{x,p}}\leq \f{1}{\fc^{k-\epsilon}}CT_1e^{CT_1}.
\end{align}

Using \eqref{5.61}--\eqref{5.67-03}, for $t\geq T_1$, similar to the deduction of \eqref{5.9-40}, we obtain
\begin{align}\label{5.75-00}
	\|w_{\beta-6}\mathfrak{g}_{2}(t)\|_{L^\infty_{x,p}}
	&\leq C_{T_1}\frac{1}{\mathfrak{c}^{k-\epsilon}}+C\int_{T_1}^te^{-\hat{\lambda}s}\|w_{\beta-6}\mathfrak{g}_{2}(s)\|_{L^\infty_{x,p}}ds+\f{C}{N}\sup_{T_1 \le s \le t}\|w_{\beta-6}\mathfrak{g}_{2}(s)\|_{L^\infty_{x,p}}\nonumber\\
	&\quad +C_{N}\int_{T_1}^te^{-2\hat{\lambda}(t-s)}\|\mathfrak{g}_{2}(s)\|_{L^2_{x,p}}ds.
\end{align}
Thus one can use \eqref{5.75-06}--\eqref{5.76-06} and carry out similar arguments as in Proposition \ref{prop3.2} and Lemma \ref{lem4.5} to obtain
\begin{align}\label{5.78-000}
	&\|\mathfrak{g}_{2}(s)\|_{L^2_{x,p}}\leq C_{T_1}e^{-\f{\sigma_3s}{2}}\|\mathfrak{g}_{2}(T_1)\|_{L^2_{x,p}}\nonumber\\
	&\qquad +C \left\{\int_{T_1}^se^{-\sigma_3(s-s_1)}\left[\|\mathbf{\Gamma}_{\mathfrak{c}}(\mathfrak{g}_{2},f_{\mathfrak{c}})(s_1)\|_{L^2_{x,p}}^2+\|\mathbf{\Gamma}_{\mathfrak{c}}(f,\mathfrak{g}_{2})(s_1)\|_{L^2_{x,p}}^2+\|\tilde{\mathcal{R}}(s_1)\|_{L^2_{x,p}}^2\right]ds_1\right\}^{\f12}
\end{align}
for some positive constant $\sigma_3>0$. Then one gets
\begin{align}\label{5.78-00}
	&\|\mathfrak{g}_{2}(s)\|_{L^2_{x,p}}\leq C_{T_1}e^{-\f{\sigma_3s}{2}}\|w_{\beta-6}\mathfrak{g}_{2}(T_1)\|_{L^\infty_{x,p}}\nonumber\\
	&\qquad+C \left\{\int_{T_1}^se^{-\sigma_3(s-s_1)}\left[e^{-2\sigma_0s_1}\|w_{\beta-6}\mathfrak{g}_{2}(s_1)\|_{L^\infty_{x,p}}^2+\|w_{\beta-6}\bar{\mathcal{R}}(s_1)\|_{L^\infty_{x,p}}^2\right]ds_1\right\}^{\f12}.
\end{align}
Plugging \eqref{5.78-00} into \eqref{5.75-00}, one obtains
\begin{align*} 
	\|w_{\beta-6}\mathfrak{g}_{2}(t)\|^2_{L^\infty_{x,p}} 
	%&\leq C_{N,T_1}\frac{1}{\mathfrak{c}^{2k-2\epsilon}}+C\int_{T_1}^te^{-\hat{\lambda}s}\|w_{\beta-6}\mathfrak{g}_{2}(s)\|^2_{L^\infty_{x,p}}ds+\f{C}{N^2}\sup_{T_1 \le s \le t}\|w_{\beta-6}\mathfrak{g}_{2}(s)\|^2_{L^\infty_{x,p}}\nonumber\\
	%&\quad +C_{N}\int_{T_1}^te^{-\hat{\lambda}(t-s)}ds\cdot\int_{T_1}^tds \int_{T_1}^s e^{-3\hat{\lambda}(t-s)}e^{-\sigma_3(s-s_1)}e^{-2\sigma_0s_1}\|w_{\beta-6}\mathfrak{g}_{2}(s_1)\|_{L^\infty_{x,p}}^2ds_1\nonumber\\
	&\leq C_{N,T_1}\frac{1}{\mathfrak{c}^{2k-2\epsilon}}+C_{N}\int_{T_1}^te^{-\hat{\lambda}s}\|w_{\beta-6}\mathfrak{g}_{2}(s)\|^2_{L^\infty_{x,p}}ds+\f{C}{N^2}\sup_{T_1 \le s \le t}\|w_{\beta-6}\mathfrak{g}_{2}(s)\|^2_{L^\infty_{x,p}},
\end{align*}
%\begin{align*}
%	&\|w_{\beta-6} \mathfrak{g}_2(t)\|_{L^\infty_{x,p}}^2\nonumber\\
%	&\lesssim Ce^{-\min\{2\lambda_0,\nu_0,\lambda_2\}t}\left(\|w_{\beta-6}\mathfrak{g}(T_0)\|_{L^\infty_{x,p}}^2+\|w_{\beta-6}\mathfrak{g}_1(T_0)\|_{L^\infty_{x,p}}^2\right)\nonumber\\
%	&\quad+Ce^{-\min\{2\lambda_0,\nu_0,\lambda_2\}t}\int_{T_0}^t\|w_{\beta-6}\mathfrak{g}_2(s)\|_{L^\infty_{x,p}}^2ds\nonumber\\
%	&\quad+\sup_{0\leq s\leq t}\|\nu_\fc^{-1}w_{\beta-6}\tilde{\mathcal{R}}\|_{L^\infty_{x,p}}^2,
%\end{align*}
which, together with Gr\"{o}nwall's inequality and \eqref{5.67-03}, yields that
\begin{align}\label{5.47-01}
	\|w_{\beta-6}\mathfrak{g}_2(t)\|_{L^\infty_{x,p}}\lesssim \frac{1}{\mathfrak{c}^{k-\epsilon}},\text{ for } t\geq 0.
\end{align}
Combining \eqref{4.56-01} and \eqref{5.47-01}, one finally obtains that
\begin{align*}
	\sup_{t\ge 0}\|w_{\beta-6}(f_{\mathfrak{c}}-f)(t)\|_{L^\infty_{x,p}}=\sup_{t\ge 0}\|w_{\beta-6}\mathfrak{g}(t)\|_{L^\infty_{x,p}}\lesssim\f{1}{\fc^{k-\epsilon}}.
\end{align*}
Hence we conclude \eqref{1.27}. Therefore the proof of Theorem \ref{thm1.4} is completed.
\end{proof}

	\noindent{\bf Acknowledgments.}
	Chuqi Cao's research is partially supported by Research Centre for Nonlinear Analysis, Hong Kong Polytechnic University. Yong Wang's research is partially supported by National Key R\&D Program of China No. 2021YFA1000800, National Natural Science Foundation of China No. 12022114, 12288201,  CAS Project for Young Scientists in Basic Research, Grant No. YSBR-031, and Youth Innovation Promotion Association of the Chinese Academy of Science No. 2019002. Changguo Xiao's research is partially supported by National Natural Science Foundation of China No. 12361045 and Guangxi Natural Science Foundation Grant No. 2023GXNSFAA026066. 
	
	\

	\noindent{\bf Conflict of interest.} The authors declare that they have no conflict of interest.	
	

\begin{thebibliography}{99}
		
	\bibitem{Calogero} Calogero, S.: The Newtonian limit of the relativistic Boltzmann equation. \textit{J. Math. Phys.} \textbf{45}, 4042--4052 (2004)
	
	\bibitem{Chapman} Chapman, J., Jang, J.W. , Strain, R.M.: On the determinant problem for the relativistic
	Boltzmann equation. \textit{Comm. Math. Phys.} \textbf{384}, 1913--1943 (2021)
%	\bibitem{Cercignani} Cercignani, C., Illner, R., Pulvirenti, M.: \textit{The Mathematical Theory of Dilute Gases}. Springer-Verlag, New York 1994
	
	\bibitem{Duan} Duan, R.J., Huang, F.M., Wang, Y., Yang, T.: Global well-posedness of the Boltzmann equation with large amplitude initial data. \textit{Arch. Ration. Mech. Anal.} \textbf{225}, 375--424 (2017)
	
	
	
%	 \bibitem{DRWZ}Duan, R.J., Huang, F.M., Wang, Y., Zhang, Z.: Effects of Soft Interaction and Non-isothermal Boundary
%	Upon Long-Time Dynamics of Rarefied Gas. \textit{Arch. Rational. Mech. Anal.} \textbf{234}, 925--1006 (2019)
%	 
	\bibitem{Duan-2} Duan, R.J., Wang, Y.: The Boltzmann equation with large-amplitude initial data in bounded domains. \textit{Adv. Math.} \textbf{343}, 36--109 (2019)
	
	\bibitem{Esposito} Esposito, R., Guo, Y., Kim, C., Marra, R.: Non-isothermal boundary in the Boltzmann theory and Fourier law. \textit{Commun. Math. Phys.} \textbf{323}, 177--239 (2013)
	
	\bibitem{Gla-Str}  Glassey, R.T., Strauss, W.A.: Asymptotic stability of the relativistic Maxwellian. \textit{Publ.
	Res. Inst. Math. Sci.} \textbf{29}, 301--347 (1993)
	
	 \bibitem{Glassey} Glassey, R.T.: \textit{The Cauchy problem in kinetic theory}. Society for Industrial and Applied Mathematics (SIAM), Philadelphia 1996
	 
	 \bibitem{Glassey1} Glassey, R.T., Strauss, W.A.: On the derivatives of the collision map of relativistic particles. \textit{Transport Theory Statist. Phys.} \textbf{20}, 55--68 (1991)
	 
	\bibitem{Groot} de Groot, S.R., van Leeuwen, W.A., van Weert, Ch.G.: \textit{Relativistic Kinetic Theory}. North-Holland, Amsterdam 1980
	
	\bibitem{Guo4} Guo, Y.: The Vlasov-Poisson-Boltzmann system near Maxwellians. \textit{Comm. Pure Appl. Math.} \textbf{55}, 1104--1135 (2002)
	
	
	\bibitem{Guo6} Guo, Y.: The Vlasov-Maxwell-Boltzmann system near Maxwellians. \textit{Invent. Math.} \textbf{153}, 593--630 (2003)
	
%	\bibitem{Lebedev} Lebedev, N.N.: \textit{Special functions and their applications}. New York: Dover Publications Inc., 1972, revised edition, translated from the Russian and edited by Richard A. Silverman, Unabridged and corrected republication
	\bibitem{Kim} Kim, C.: Boltzmann equation with a large potential in a periodic box. \textit{Comm. Partial Differ. Equ.} \textbf{39}, 1393--1423 (2014)
	
	\bibitem{Olver} Olver, F.W.J.: \textit{Asymptotics and special functions.} AKP Classics, Wellesley, MA: A K Peters Ltd., 1997. Reprint of the 1974 original. New York: Academic Press
	
	\bibitem{Ruggeri} Ruggeri, T., Xiao, Q.H., Zhao, H.J.: Nonlinear hyperbolic waves in relativistic gases of massive particles with Synge energy. \textit{ Arch. Ration. Mech. Anal.} \textbf{239}, 1061--1109 (2021)
		
	\bibitem{Strain} Strain, R.M.: Global Newtonian limit for the relativistic Boltzmann equation near vacuum. \textit{SIAM J. Math. Anal.} {\bf 42}, 1568--1601 (2010)	
	
	\bibitem{Strain1} Strain, R.M.: Coordinates in the relativistic Boltzmann theory. \textit{Kinet. Relat. Models} \textbf{4}, 345--359 (2011)
	
	\bibitem{Strain2} Strain, R.M.: Asymptotic stability of the relativistic Boltzmann equation for the soft potentials. \textit{Comm. Math. Phys.} {\bf 300}, 529--597, 2010
	
	\bibitem{Wang} Wang, Y.: Global Well-Posedness of the Relativistic Boltzmann Equation. \textit{SIAM J. Math. Anal.} \textbf{50}, 5637--5694 (2018)

	\bibitem{Wang-Xiao} Wang Y., Xiao  C.G.: Hydrodynamic limit and Newtonian limit from the relativistic Boltzmann equation to the classical Euler equations. arXiv:2308.16646v1 
	
	\bibitem{Watson} Watson, G.N.: \textit{A Treatise on the Theory of Bessel Functions.} second edition,
	Cambridge University Press, London 1952
	
    \end{thebibliography}
\end{document}